\documentclass[10pt,letterpaper]{article}

\usepackage[letterpaper,left=0.84in,right=0.84in,top=0.72in,bottom=0.66in,includefoot]{geometry}
\usepackage[utf8]{inputenc}
\usepackage{amsmath,amssymb,amsfonts,mathtools}
\usepackage{bm}
\usepackage{amsthm}
\usepackage{microtype}
\usepackage{enumitem}
\usepackage{cite}
\usepackage{titlesec}
\usepackage{etoolbox}
\AtBeginEnvironment{thebibliography}{\setlength{\itemsep}{0pt}\setlength{\parsep}{0pt}\setlength{\parskip}{0pt}}
\usepackage{booktabs}
\usepackage{graphicx}
\usepackage{placeins}
\usepackage[colorlinks=true,linkcolor=blue,citecolor=blue,urlcolor=black]{hyperref}

\allowdisplaybreaks[4]
\renewcommand{\topfraction}{0.88}

\renewcommand{\textfraction}{0.10}
\renewcommand{\floatpagefraction}{0.78}
\titleformat{\section}{\normalfont\large\bfseries}{\thesection}{1em}{}
\titleformat{\subsection}{\normalfont\normalsize\bfseries}{\thesubsection}{1em}{}
\titleformat{\subsubsection}{\normalfont\normalsize\bfseries}{\thesubsubsection}{1em}{}
\titlespacing*{\section}{0pt}{11pt plus 2pt minus 2pt}{5pt plus 1pt}
\titlespacing*{\subsection}{0pt}{9pt plus 2pt minus 1pt}{4pt plus 1pt}
\titlespacing*{\subsubsection}{0pt}{7pt plus 2pt minus 1pt}{3pt plus 1pt}

\numberwithin{equation}{section}

\theoremstyle{plain}
\newtheorem{theorem}{Theorem}[section]
\newtheorem{proposition}[theorem]{Proposition}
\theoremstyle{definition}
\newtheorem{assumption}{Assumption}[section]
\theoremstyle{remark}
\newtheorem{remark}[theorem]{Remark}
\newtheorem*{remarkstar}{Remark}

\newcommand{\R}{\mathbb R}
\newcommand{\bu}{\boldsymbol u}
\newcommand{\bv}{\boldsymbol v}
\newcommand{\bw}{\boldsymbol w}
\newcommand{\bfv}{\boldsymbol f}
\newcommand{\bJ}{\boldsymbol J}
\newcommand{\bxi}{\boldsymbol\xi}
\newcommand{\bzeta}{\boldsymbol\zeta}
\newcommand{\bg}{\boldsymbol g}
\newcommand{\bX}{\boldsymbol X}
\newcommand{\bF}{\boldsymbol F}
\newcommand{\Cn}{\mathit{Cn}}
\newcommand{\Pe}{\mathit{Pe}}
\newcommand{\Rey}{\mathit{Re}}
\newcommand{\We}{\mathit{We}}
\newcommand{\Fr}{\mathit{Fr}}
\newcommand{\Th}{\mathcal T_h}
\newcommand{\That}{\widehat{\mathcal T}_h}
\newcommand{\Sh}{S_h}
\newcommand{\Vh}{V_h}
\newcommand{\Ph}{P_h}

\newcommand{\norm}[1]{\left\lVert #1\right\rVert}

\begin{document}

\begin{center}
{\large History-Compatible Energy-Stable Finite Element Schemes for Variable-Density Cahn--Hilliard--Navier--Stokes Flows on Evolving Meshes}
\vspace{6pt}

Wenbin Wang\textsuperscript{1},
Yunqing Huang\textsuperscript{2},
Yin Yang\textsuperscript{4}, and
Huayi Wei\textsuperscript{1,2,3,*}

\vspace{4pt}
{\small
\textsuperscript{1}School of Mathematics and Computational Science, Xiangtan University, Xiangtan 411105, China

\textsuperscript{2}National Center for Applied Mathematics in Hunan, Xiangtan University, Xiangtan 411105, China

\textsuperscript{3}Hunan Key Laboratory for Computation and Simulation in Science and Engineering, Xiangtan University, Xiangtan 411105, China

\textsuperscript{4}Hunan Research Center of the Basic Discipline Fundamental Algorithmic Theory and Novel Computational Methods, Xiangtan University, Xiangtan 411105, China

\textsuperscript{*}Corresponding author.
}

\end{center}
\vspace{2pt}
\begin{abstract}
We consider variable-density Cahn--Hilliard--Navier--Stokes (CHNS) discretizations on finite element meshes that may change between accepted time levels through fixed-topology motion or topology-changing remeshing. When the discrete spaces vary in time, the phase, kinetic, and pressure histories entering a multistep scheme are measured in different discrete structures and cannot, in general, be transferred by a single operator. We develop decoupled backward Euler (BE) and second-order backward differentiation formula (BDF2) schemes by combining exact physical cross-mesh pairings with history representations compatible with the corresponding phase-energy, kinetic-energy, and pressure-gradient storages. The phase update also determines an Abels--Garcke--Grün-consistent mass flux used in the momentum transport. A scalar capillary-exchange equation separates the phase and fluid solves while retaining the discrete energy exchange. The resulting field subproblems are linear, the scalar equation has a unique positive solution, and the schemes satisfy modified energy balances without a time-step restriction under the stated admissibility assumptions. Numerical experiments confirm second-order temporal convergence under both mesh updates, phase-mass conservation, modified-energy decay in the unforced tests, and comparable Rayleigh--Taylor and rising-bubble dynamics.
\end{abstract}
\noindent\textbf{Keywords:} Cahn--Hilliard--Navier--Stokes equations; variable density; evolving meshes; cross-mesh histories; energy stability.\par
\noindent\textbf{MSC codes:} 65M60, 65M50, 65M12, 76T06.

\section{Introduction}
Phase-field methods describe diffuse interfaces by continuous order parameters and provide partial differential equation models for interface motion, topological changes, and surface energy \cite{CahnHilliard1958,HohenbergHalperin1977,LowengrubTruskinovsky1998,Jacqmin1999,YueFengLiuShen2004,DingSpeltShu2007}. For two fluids with different densities, simply inserting variable-density coefficients into the classical Model H does not automatically preserve thermodynamic consistency. The Abels--Garcke--Gr\"un (AGG) model reorganizes the continuity and momentum equations through a volume-averaged solenoidal velocity, an affine density law, and a diffusion-induced mass flux, and thereby preserves the energy law at large density ratios \cite{AbelsGarckeGruen2012}. The numerical developments of Shen--Yang \cite{ShenYang2010} and Gr\"un--Klingbeil \cite{GruenKlingbeil2014}, among others, have established a substantial literature on variable-density CHNS discretizations.

Fixed-mesh CHNS discretizations are well developed. Finite element stability and convergence were established in representative works by Feng \cite{Feng2006} and Kay--Styles--Welford \cite{KayStylesWelford2008}, while Han--Wang \cite{HanWang2015} and Diegel et al.~\cite{DiegelWangWangWise2017} developed and analyzed second-order schemes. For decoupling, explicit extrapolation of phase--fluid coupling is attractive but generally breaks the exact energy exchange. Zhao--Han \cite{ZhaoHan2021} used a constraint-gradient-flow reformulation to construct second-order decoupled energy-stable schemes, and Li--Shen \cite{LiShen2022} combined the multiple scalar auxiliary variables approach with pressure-correction techniques to obtain linear fully decoupled schemes. Auxiliary-variable ideas have also been incorporated into finite element discretizations. Chen--Li--Yang--Yin \cite{ChenLiYangYin2026} proposed first- and second-order finite element schemes based on invariant energy quadratization (IEQ) and studied a hybrid treatment of the IEQ variable, while Zou et al.~\cite{ZouEtAl2026} combined scalar auxiliary variable (SAV) and zero-energy-contribution ideas, discontinuous Galerkin discretization, and a second-order projection method for a linear fully decoupled non-isothermal CHNS scheme. These works extend second-order energy-stable decoupling to different finite element settings and pressure treatments.

The variable-density AGG model imposes an additional constraint: Cahn--Hilliard diffusion determines density transport, and the same mass flux enters momentum convection and variable-density kinetic-energy storage. Gr\"un--Klingbeil \cite{GruenKlingbeil2014} constructed a thermodynamically consistent discretization of the AGG model. Shen--Yang \cite{ShenYang2015} developed decoupled, unconditionally energy-stable schemes for Cahn--Hilliard phase-field models of two-phase incompressible flow, while Liu--Shen--Yang \cite{LiuShenYang2015} and Chen--Yang \cite{ChenYang2021} developed decoupled energy-stable schemes for variable-density formulations. More recently, Wang--Li--Wang \cite{WangLiWang2024} constructed second-order linear decoupled finite element schemes for the AGG model based on multiple scalar auxiliary variables, and Zhang--Luo--Wang \cite{ZhangLuoWang2025} developed a second-order fully decoupled variable-density scheme using a decoupling constant scalar auxiliary variable formulation, separating zero-energy-contribution and non-zero-energy-contribution terms. The latter two works provide representative second-order decoupled energy-stable discretizations of the AGG system on fixed finite element spaces.

Adaptive spatial resolution is particularly natural for diffuse-interface flow because the interfacial layer occupies only a small part of the domain. Chen--Shen \cite{ChenShen2016} developed energy-stable CHNS schemes with adaptive mesh and time stepping. Chen--Huang--Yi \cite{ChenHuangYi2021} further combined a decoupled energy-stable finite element discretization with a posteriori estimators based on superconvergent cluster recovery and a time--space adaptive algorithm; a subsequent study established a priori error estimates for a totally decoupled, linear, unconditionally energy-stable finite element method for CHNS \cite{ChenHuangYi2022}. Khanwale et al.~\cite{KhanwaleEtAl2020,KhanwaleEtAl2023} used parallel adaptive octree meshes in thermodynamically consistent and projection-based CHNS discretizations. Variational moving meshes \cite{HuangRussell2011,HuangKamenski2015,HuangKamenskiRussell2015,HuangKamenski2018,WangHuangWei2026} redistribute vertices while keeping the connectivity fixed, whereas metric-based anisotropic remeshing \cite{ArpaiaEtAl2022,DapognyDobrzynskiFrey2014,DobrzynskiFrey2008} may also change the topology. In arbitrary Lagrangian--Eulerian (ALE) finite element methods, the work of Donea et al.~\cite{DoneaGiulianiHalleux1982} was followed by analyses of stability, geometric conservation, and high-order temporal accuracy \cite{FormaggiaNobile1999,FarhatGeuzaineGrandmont2001,GeuzaineGrandmontFarhat2003,FormaggiaNobile2004,EtienneGaronPelletier2009,Liu2013}. Two-phase ALE formulations include \cite{ZhengKarniadakis2016,DuanLiYang2022,GarckeNuernbergZhao2023,GarckeNuernbergZhao2024}, while a moving-mesh finite element method with energy-stability properties for moving contact lines was developed in \cite{ZhaoRen2020}. On fixed-topology mesh sequences, reference mappings and their Jacobian weights provide a natural cross-time connection, whereas topology-changing remeshing removes that identification.

Conservative transfer across nonmatching meshes can be carried out by supermesh interpolation, local Galerkin projection, and constrained finite element transfer \cite{FarrellEtAl2009,FarrellMaddison2011,PontCodinaBaiges2017}. In a multistep variable-density CHNS method, however, transfer accuracy alone is not sufficient. When accepted time levels are represented in different finite element spaces, the phase history enters the free-energy terms, the square-root kinetic variable enters the BE/BDF2 kinetic energy, the transported phase determines the density flux used in the momentum equation, and the pressure history in a pressure-correction method is measured in a weighted gradient norm. These quantities need not be preserved by the same transfer. We therefore define a separate cross-mesh relation for each of these four uses.

We work on a fixed physical domain with a sequence of finite element meshes related either by fixed-topology motion or by topology-changing remeshing. A historical field is kept as a physical function when it appears in a cross-mesh integral, but the BE/BDF2 history used as a current Galerkin unknown must belong to the current discrete space. The phase BE/BDF2 history is therefore obtained by a Ritz map in the quadratic phase-energy norm. For the kinetic term, a Jacobian-weighted isometry is used during fixed-topology ALE motion and a norm-preserving projection after remeshing. The pressure-correction method similarly transfers the pressure history in the current density-weighted gradient norm.

The Cahn--Hilliard (CH) equation uses a carrier velocity that is orthogonal to gradients in the current phase space. Together with the affine density law, its weak form determines a discrete AGG mass flux; the same flux is used in the skew momentum term. This gives phase conservation, the affine density balance, and zero kinetic self-work from one discrete mass flux. We also use an equivalent SAV--constant-scalar extension in which $r$ represents the nonlinear free energy and $q\equiv1$ expresses the capillary exchange identity. After the CH and momentum solves are separated, $q$ is obtained from a scalar quadratic equation with one positive root.

The paper makes three specific contributions. First, it defines phase and kinetic histories for both fixed-topology motion and topology-changing remeshing and proves the BE/BDF2 polarization identities across consecutive meshes. Second, it derives a density-consistent mass flux from the phase equation and uses the same flux in the affine density balance and the skew momentum term. Third, it gives saddle-point and pressure-correction BE/BDF2 schemes and proves stepwise solvability, positivity of the scalar $q$, and the associated modified energy balances. The numerical tests use a trace--logarithmic variational moving mesh \cite{WangHuangWei2026} for fixed connectivity and the MMG simplicial-remeshing library \cite{ArpaiaEtAl2022,DapognyDobrzynskiFrey2014,DobrzynskiFrey2008} when connectivity changes.

Section 2 gives the extended AGG model, the mesh-dependent finite element spaces, and the cross-mesh BE/BDF2 histories. Section 3 presents the carrier correction, mass flux, and the saddle-point and pressure-correction schemes. Section 4 proves conservation, solvability, positivity of $q$, and the modified energy balances. Section 5 describes the two mesh updates, cross-mesh integration, and the linear algebra used in the implementation. Section 6 contains the convergence and benchmark tests, followed by conclusions in Section 7.

\section{Model, mesh-dependent finite element spaces, and cross-time histories}
\subsection{Notation, nondimensionalization, and the extended AGG model}
\noindent\textbf{Notation.} Let $\Omega\subset\R^d$ ($d=2,3$) be a fixed bounded polygonal or polyhedral domain. For scalar, vector, or tensor fields $v$ and $z$, $(v,z)_D$ denotes the standard $L^2(D)$ inner product, with the Frobenius product used for tensor fields, and
\begin{equation*}
\norm{v}_D^2=(v,v)_D,\qquad |v|_{1,D}=\norm{\nabla v}_D.
\end{equation*}
The subscript $D$ is omitted when $D=\Omega$. For a nonnegative weight $c$, we use $\norm{c^{1/2}v}_D^2=(cv,v)_D$. Below, $(\cdot,\cdot)_\ell$ denotes the physical $L^2(\Omega)$ pairing at level $\ell$, whereas $\langle\cdot,\cdot\rangle_\times$ denotes the physical pairing of quantities represented on different meshes.

\noindent\textbf{Dimensionless parameters and material laws.} We use the following nondimensional variables and parameters. Let $L_r$, $u_r$, $\rho_r$, and $\eta_r$ denote the reference length, velocity, density, and dynamic viscosity, respectively; let $\epsilon$ be the interfacial thickness, $\lambda$ the phase-field energy scale, $M_r$ the reference mobility, and $g$ the gravitational acceleration. We set
\begin{equation*}
 x=L_rx^*,\qquad t=\frac{L_r}{u_r}t^*,\qquad \bu=u_r\bu^*,\qquad
 \rho=\rho_r\rho^*,\qquad \eta=\eta_r\eta^*,\qquad p=\rho_ru_r^2p^*.
\end{equation*}
The phase-field energy and chemical potential are scaled consistently with the same energy scale. After dropping the superscript $*$, the Cahn, P\'eclet, Reynolds, and Weber numbers and the squared Froude parameter are, respectively,
\begin{equation*}
 \Cn=\frac{\epsilon}{L_r},\qquad
 \Pe=\frac{2\sqrt2\,\epsilon L_ru_r}{3\lambda M_r},\qquad
 \Rey=\frac{L_r\rho_ru_r}{\eta_r},\qquad
 \We=\frac{2\sqrt2\,L_r\rho_ru_r^2}{3\lambda},\qquad
 \Fr=\frac{u_r^2}{gL_r}.
\end{equation*}
The analysis uses $\Pe>0$; its benchmark-dependent choice is specified together with the numerical setup. The nondimensional prefactor of the external force is absorbed into $\bfv$. Here $\Fr=(u_r/\sqrt{gL_r})^2$ is the square of the conventional Froude number. If gravity is retained explicitly, the right-hand side of the momentum equation is written as $\Fr^{-1}\bfv_g$. The dimensionless density and viscosity depend affinely on the phase variable,
\begin{equation}
 \rho(\phi)=\rho_0+\rho_\phi\phi,\qquad \eta(\phi)=\eta_0+\eta_\phi\phi,
\end{equation}
where
\begin{equation*}
 \rho_0=\frac{\rho_1+\rho_2}{2\rho_r},\qquad
 \rho_\phi=\frac{\rho_1-\rho_2}{2\rho_r},\qquad
 \eta_0=\frac{\eta_1+\eta_2}{2\eta_r},\qquad
 \eta_\phi=\frac{\eta_1-\eta_2}{2\eta_r}.
\end{equation*}
For a nonnegative mobility $m(\phi)\ge0$, the diffusive mass flux and the total mass flux are defined by
\begin{equation}
 \bJ=-\frac{\rho_\phi}{\Pe}m(\phi)\nabla\mu,\qquad
 \bJ_\rho=\rho(\phi)\bu+\bJ.
\end{equation}

\noindent\textbf{AGG energy identity.} In the effective-pressure formulation, the affine density law and the common mass flux $\bJ_\rho$ give the density balance and the skew form of variable-density inertia. The phase transport and capillary force exchange equal and opposite work. The auxiliary variables introduced below preserve this continuous energy law and allow the field equations to be solved linearly.

\noindent\textbf{Equivalent auxiliary-variable extension.} Decompose the bulk free energy as
\begin{equation*}
 F(\phi)=\frac{s}{2}\phi^2+G(\phi),\qquad s>0,
\end{equation*}
and choose $C_G>0$ such that
\begin{equation*}
 U(\phi)=\left(\int_\Omega G(\phi)\,dx+C_G\right)^{1/2}>0.
\end{equation*}
The continuous SAV variable is defined by $r(t)=U(\phi(t))$ \cite{ShenXuYang2018,ShenXuYang2019}. We further choose $C_A>0$ and $\alpha_q>0$. The constant $C_A$ guarantees strict positivity of the phase--SAV storage entering the subsequent equation for $q$, while $\alpha_q$ sets the positive weight of the constant-scalar term in the extended energy and scalar equation. Define
\begin{equation*}
 E_{\phi,r}(t)=\frac{1}{2\We\Cn}\bigl(\Cn^2\norm{\nabla\phi}^2+s\norm{\phi}^2\bigr)
 +\frac{1}{\We\Cn}r^2+C_A.
\end{equation*}
We start from the following continuously equivalent extended system:
\begin{subequations}\label{eq:extended}
\begin{align}
 &\partial_t\phi+\nabla\cdot(\phi\bu)=\frac1\Pe\nabla\cdot(m(\phi)\nabla\mu),\label{eq:extended-a}\\
 &\mu=-\Cn^2\Delta\phi+s\phi+\frac{r}{U(\phi)}G'(\phi),\label{eq:extended-b}\\
 &\partial_t\rho+\nabla\cdot\bJ_\rho=0,\label{eq:extended-c}\\
 &\rho^{1/2}\partial_t(\rho^{1/2}\bu)+\bJ_\rho\cdot\nabla\bu+\frac12(\nabla\cdot\bJ_\rho)\bu
 -\frac1\Rey\nabla\cdot(2\eta D(\bu))+\nabla p+\frac{q}{\We\Cn}\phi\nabla\mu=\bfv,\label{eq:extended-d}\\
 &\nabla\cdot\bu=0,\label{eq:extended-e}\\
 &\partial_t r=\frac{1}{2U(\phi)}\bigl(G'(\phi),\partial_t\phi\bigr),\label{eq:extended-f}\\
 &\left(E_{\phi,r}+\frac{q}{\alpha_q}\right)\partial_t q
 +\frac{q}{\We\Cn}\bigl[(\phi\bu,\nabla\mu)-(\phi\nabla\mu,\bu)\bigr]=0.\label{eq:extended-g}
\end{align}
\end{subequations}
Here $D(\bu)=\frac12(\nabla\bu+\nabla\bu^T)$ and the effective pressure is
\begin{equation*}
 p=p_{\rm hyd}+\frac{\Cn}{2\We}|\nabla\phi|^2+\frac{1}{\We\Cn}\bigl(F(\phi)-\mu\phi\bigr).
\end{equation*}
We prescribe $r(0)=U(\phi(0))$ and $q(0)=1$. For the velocity boundary condition, let
$\partial\Omega=\overline{\Gamma_D}\cup\overline{\Gamma_S}$ with disjoint relative interiors. We impose
\begin{equation}
 \bu=0\quad\text{on }\Gamma_D,\qquad
 \bu\cdot\boldsymbol n=0,\quad
 \bigl(2\eta D(\bu)\boldsymbol n\bigr)\cdot\boldsymbol t=0
 \quad\text{on }\Gamma_S,
 \qquad
 \partial_n\phi=\partial_n\mu=0\quad\text{on }\partial\Omega,
\end{equation}
where $\boldsymbol t$ denotes any tangential direction. The pure no-slip case corresponds to
$\Gamma_D=\partial\Omega$. In either case the velocity is impermeable, the viscous boundary work vanishes,
and $\bJ_\rho\cdot\boldsymbol n=0$ on $\partial\Omega$; hence the boundary terms used in the mass and energy
identities below vanish.
With $r(0)=U(\phi(0))$, equation \eqref{eq:extended-f} gives $r(t)=U(\phi(t))$. Moreover,
\begin{equation*}
 (\phi\bu,\nabla\mu)=(\phi\nabla\mu,\bu),
\end{equation*}
so that $q(t)\equiv1$ on the branch initialized by $q(0)=1$. Hence the extended system is equivalent to the standard effective-pressure AGG formulation on its consistent branch: $r$ represents the nonlinear free energy, whereas $q$ records the capillary-exchange identity.

Define the extended energy
\begin{equation}
 E_{\rm ext}(\phi,\bu,r,q)=\frac12\norm{\rho(\phi)^{1/2}\bu}^2+qE_{\phi,r}+\frac{1}{2\alpha_q}q^2.
\end{equation}

\begin{theorem}[Energy balance of the continuous extended AGG system]\label{thm:continuous-energy}
Let $(\phi,\mu,\rho,\bu,p,r,q)$ be a sufficiently smooth solution of \eqref{eq:extended}--(2.4) with $q(t)>0$. Then
\begin{equation}
 \frac{d}{dt}E_{\rm ext}+\frac{q}{\Pe\We\Cn}\norm{m(\phi)^{1/2}\nabla\mu}^2
 +\frac{2}{\Rey}\norm{\eta(\phi)^{1/2}D(\bu)}^2=(\bfv,\bu).
\end{equation}
If $r(0)=U(\phi(0))$ and $q(0)=1$, then $r(t)=U(\phi(t))$ and $q(t)=1$. In this case, define the standard AGG physical energy by
\begin{equation*}
 E_{\rm AGG}(\phi,\bu)=\frac12\norm{\rho(\phi)^{1/2}\bu}^2
 +\frac{1}{\We\Cn}\left[\frac{\Cn^2}{2}\norm{\nabla\phi}^2+\int_\Omega F(\phi)\,dx\right].
\end{equation*}
Then $E_{\rm ext}=E_{\rm AGG}+C_G/(\We\Cn)+C_A+1/(2\alpha_q)$; that is, the extended and physical AGG energies differ only by a constant.
\end{theorem}

\begin{proof}
We first derive the phase--SAV balance. Testing \eqref{eq:extended-a} by $\mu/(\We\Cn)$ and integrating by parts gives
\begin{equation*}
 \frac{1}{\We\Cn}(\partial_t\phi,\mu)-\frac{1}{\We\Cn}(\phi\bu,\nabla\mu)
 +\frac{1}{\Pe\We\Cn}\norm{m(\phi)^{1/2}\nabla\mu}^2=0.
\end{equation*}
On the other hand, testing \eqref{eq:extended-b} by $\partial_t\phi/(\We\Cn)$ yields
\begin{equation*}
 \frac{1}{\We\Cn}(\mu,\partial_t\phi)=\frac{1}{\We\Cn}
 \left[\Cn^2(\nabla\phi,\nabla\partial_t\phi)+s(\phi,\partial_t\phi)
 +\frac{r}{U(\phi)}(G'(\phi),\partial_t\phi)\right].
\end{equation*}
By \eqref{eq:extended-f},
\begin{equation*}
 \frac{r}{U(\phi)}(G'(\phi),\partial_t\phi)=2r\,\partial_t r,
\end{equation*}
and therefore
\begin{equation*}
 \frac{1}{\We\Cn}(\mu,\partial_t\phi)=\frac{d}{dt}E_{\phi,r}.
\end{equation*}
Substitution into the phase equation gives the continuous phase--SAV balance
\begin{equation*}
 \frac{d}{dt}E_{\phi,r}-\frac{1}{\We\Cn}(\phi\bu,\nabla\mu)
 +\frac{1}{\Pe\We\Cn}\norm{m(\phi)^{1/2}\nabla\mu}^2=0.
\end{equation*}
Multiply this identity by $q$. At the same time, write \eqref{eq:extended-g} in the form
\begin{equation*}
 E_{\phi,r}\partial_tq+\frac{q}{\alpha_q}\partial_tq
 +\frac{q}{\We\Cn}\bigl[(\phi\bu,\nabla\mu)-(\phi\nabla\mu,\bu)\bigr]=0.
\end{equation*}
Adding the two relations cancels the two occurrences of $(\phi\bu,\nabla\mu)$, and the product rule yields
\begin{equation*}
 \frac{d}{dt}\left(qE_{\phi,r}+\frac{q^2}{2\alpha_q}\right)
 +\frac{q}{\Pe\We\Cn}\norm{m^{1/2}\nabla\mu}^2
 -\frac{q}{\We\Cn}(\phi\nabla\mu,\bu)=0.
\end{equation*}
Next, test \eqref{eq:extended-d} by $\bu$. Since
\begin{equation*}
 (\bJ_\rho\cdot\nabla\bu+\tfrac12(\nabla\cdot\bJ_\rho)\bu,\bu)=0,
 \qquad (p,\nabla\cdot\bu)=0,
\end{equation*}
the momentum balance is
\begin{equation*}
 (\bfv,\bu)=\frac12\frac{d}{dt}\norm{\rho^{1/2}\bu}^2
 +\frac{2}{\Rey}\norm{\eta^{1/2}D(\bu)}^2
 +\frac{q}{\We\Cn}(\phi\nabla\mu,\bu).
\end{equation*}
Adding the last two identities cancels the capillary-exchange terms exactly and gives (2.6). On the consistent branch, $r^2=\int_\Omega G(\phi)\,dx+C_G$ and $q=1$; substituting these identities into $E_{\phi,r}$ and (2.5) gives the stated constant offset.
\end{proof}

\subsection{Mesh-induced finite element spaces and BE/BDF2 notation}
For each accepted time level $t^\ell$, let $\Th^\ell$ be a conforming triangulation of the fixed physical domain $\Omega$. Define
\begin{equation*}
 \mathbf V(\Omega)=\left\{\boldsymbol v\in[H^1(\Omega)]^d:
 \boldsymbol v=0\ \text{on }\Gamma_D,\
 \boldsymbol v\cdot\boldsymbol n=0\ \text{on }\Gamma_S\right\}.
\end{equation*}
The tangential free-slip condition on $\Gamma_S$ is natural in the viscous weak form. We take
\begin{equation*}
 \Sh^\ell\subset H^1(\Omega),\qquad \Vh^\ell\subset\mathbf V(\Omega),\qquad \Ph^\ell\subset L^2(\Omega).
\end{equation*}
We assume $1\in\Sh^\ell\cap\Ph^\ell$ and set $S_{h,0}^\ell=\Sh^\ell\cap L_0^2(\Omega)$ and $P_{h,0}^\ell=\Ph^\ell\cap L_0^2(\Omega)$. For quantities represented on two accepted meshes, define the physical cross-space pairing
\begin{equation}
 \langle v_h^m,z_h^n\rangle_\times:=\int_\Omega v_h^m(x)z_h^n(x)\,dx.
\end{equation}
The same notation is used for vector, gradient, and tensor pairings. If both factors are represented on the current mesh, it coincides with the current physical inner product. Unless stated otherwise, $(\cdot,\cdot)_\ell$ and $\norm{\cdot}_\ell$ denote the physical $L^2(\Omega)$ inner product and norm at level $\ell$.

The mesh sequence may be generated by continuous fixed-topology motion or by instantaneous topology-changing remeshing. Throughout this work, the finite element families and polynomial degrees are fixed; the changes in $\Sh^\ell$, $\Vh^\ell$, and $\Ph^\ell$ are induced by these mesh operations. All accepted meshes preserve the boundary partition $\Gamma_D\cup\Gamma_S$, are shape regular, and the velocity--pressure pairs satisfy a uniform discrete inf--sup condition \cite{BoffiBrezziFortin2013}. On fixed-topology intervals, the reference maps introduced in Section 2.4 and their inverses are uniformly admissible.

Fix $k\in\{1,2\}$. For histories viewed as physical functions on the fixed domain, define
\begin{equation}
 z_H^{n+1}=\begin{cases}z^n,&k=1,\\(4z^n-z^{n-1})/3,&k=2,\end{cases}
 \qquad
 z_E^{n+1}=\begin{cases}z^n,&k=1,\\2z^n-z^{n-1},&k=2,\end{cases}
\end{equation}
and set $\gamma_1=1$, $\gamma_2=3/2$,
\begin{equation}
 D_\tau^{(k)}z^{n+1}=\frac{\gamma_k}{\tau}(z^{n+1}-z_H^{n+1}).
\end{equation}
Extrapolated histories such as $\phi_E^{n+1}$, $G'(\phi_E^{n+1})$, and $m_E^{n+1}$ are evaluated as physical functions in the cross-space pairings below.

\subsection{Phase-field BE/BDF2 history}
With the splitting $F(\phi)=\frac{s}{2}\phi^2+G(\phi)$ from Section 2.1, define the quadratic phase-energy norm
\begin{equation}
 \norm{v}_{\Cn,s}^2:=\Cn^2\norm{\nabla v}^2+s\norm{v}^2.
\end{equation}
The phase BE/BDF2 history is read through this quadratic part of the chemical-potential relation. Its current representative $\phi_{H,h}^{n+1}\in\Sh^{n+1}$ is defined by
\begin{equation}\label{eq:phase-history}
 \Cn^2(\nabla\phi_{H,h}^{n+1},\nabla\chi_h)_{n+1}+s(\phi_{H,h}^{n+1},\chi_h)_{n+1}
 =\Cn^2\langle\nabla\phi_H^{n+1},\nabla\chi_h\rangle_\times+s\langle\phi_H^{n+1},\chi_h\rangle_\times,
 \quad\forall\chi_h\in\Sh^{n+1}.
\end{equation}
Since $s>0$, this problem is coercive and uniquely solvable. Choosing $\chi_h=1$ gives
\begin{equation*}
 (\phi_{H,h}^{n+1},1)_{n+1}=\langle\phi_H^{n+1},1\rangle_\times.
\end{equation*}

\subsection{Kinetic BE/BDF2 histories}
\noindent\textbf{Kinetic BE/BDF2 storage.} Define the square-root kinetic-energy variable
\begin{equation}
 \bzeta_h^\ell=(\rho_h^\ell)^{1/2}\bu_h^\ell.
\end{equation}
On a fixed space, the BDF2 kinetic storage reads the two quadratic states $\norm{\bzeta_h^\ell}^2$ and $\norm{2\bzeta_h^\ell-\bzeta_h^{\ell-1}}^2$. Accordingly, if $\bg_{0,H}^\ell$ denotes the current representative of the previously accepted $\bg_{0,h}^{\ell-1}$, define
\begin{equation}
 \bg_{0,h}^\ell=\bzeta_h^\ell,\qquad \bg_{1,h}^\ell=2\bzeta_h^\ell-\bg_{0,H}^\ell\quad(k=2).
\end{equation}
After the BE start-up, $\bg_{1,h}^1$ is initialized by the same relation.

\noindent\textbf{Fixed-topology ALE transfer.} On a fixed-topology interval, let $\bX_h^\ell:\widehat\Omega\to\Omega$ map a common reference triangulation $\That$ onto $\Th^\ell$, and set
\begin{equation}
 \Th^\ell=\bX_h^\ell(\That),\qquad \bF_h^\ell=\widehat\nabla \bX_h^\ell,\qquad J_h^\ell=\det \bF_h^\ell>0.
\end{equation}
The reference kinetic variable is
\begin{equation}
 \widehat{\bzeta}_h^\ell=(J_h^\ell)^{1/2}(\bzeta_h^\ell\circ \bX_h^\ell)
 =\bigl[J_h^\ell(\rho_h^\ell\circ \bX_h^\ell)\bigr]^{1/2}(\bu_h^\ell\circ \bX_h^\ell),
 \qquad \norm{\widehat{\bzeta}_h^\ell}_{\widehat\Omega}^2=\norm{\bzeta_h^\ell}_\ell^2.
\end{equation}
With $\bw=\partial_t\bX\circ\bX^{-1}$, the identities $\partial_tJ=J\nabla\cdot\bw$ and the chain rule give
\begin{equation*}
 J^{-1/2}\partial_t\widehat{\bzeta}\circ\bX^{-1}=\partial_t\bzeta+\bw\cdot\nabla\bzeta+\frac12(\nabla\cdot\bw)\bzeta.
\end{equation*}
\begin{subequations}\label{eq:kinetic-transfer}
The corresponding isometric history connection is
\begin{equation}\label{eq:kinetic-transfer-fixed}
 \bg_{j,H}^{n+1}=(J_h^{n+1})^{-1/2}\Bigl[(J_h^n)^{1/2}(\bg_{j,h}^n\circ \bX_h^n)\Bigr]\circ(\bX_h^{n+1})^{-1}.
\end{equation}

\noindent\textbf{Topology-changing transfer.} Let $\mathcal P$ be the $L^2$-orthogonal projection onto the chosen piecewise-polynomial history space on the current mesh. Define
\begin{equation}\label{eq:kinetic-transfer-constrained}
 \bg_{j,H}^{n+1}=\arg\min_{z_h}\frac12\norm{z_h-\bg_{j,h}^n}^2
 \quad\text{s.t.}\quad \norm{z_h}_{n+1}=\norm{\bg_{j,h}^n}_n.
\end{equation}
For $\mathcal P \bg_{j,h}^n\ne0$, the nearest positive branch is
\begin{equation}\label{eq:kinetic-transfer-normalized}
 \bg_{j,H}^{n+1}=\frac{\norm{\bg_{j,h}^n}_n}{\norm{\mathcal P \bg_{j,h}^n}_{n+1}}\,\mathcal P \bg_{j,h}^n,
 \qquad \norm{\bg_{j,H}^{n+1}}_{n+1}=\norm{\bg_{j,h}^n}_n.
\end{equation}
\end{subequations}
We assume $\mathcal P \bg_{j,h}^n\ne0$ whenever $\bg_{j,h}^n\ne0$, and set $\bg_{j,H}^{n+1}=0$ for a zero history. Orthogonality of $\mathcal P$ gives
\begin{equation*}
 \norm{\bg_{j,H}^{n+1}-\bg_{j,h}^n}\le\sqrt2\,\norm{(I-\mathcal P)\bg_{j,h}^n},
\end{equation*}
so the radial normalization preserves the approximation order of the underlying projection.

\noindent\textbf{Common BE/BDF2 reconstruction.} Both transfers use the same history combination
\begin{equation}\label{eq:kinetic-reconstruction}
 \bzeta_{H,h}^{n+1}=\begin{cases}
 \bg_{0,H}^{n+1},&k=1,\\[2pt]
 \dfrac{2\bg_{0,H}^{n+1}+\bg_{1,H}^{n+1}}{3},&k=2,
 \end{cases}
 \qquad
 \begin{cases}
 \bg_{0,h}^{n+1}=\bzeta_h^{n+1},\\
 \bg_{1,h}^{n+1}=2\bzeta_h^{n+1}-\bg_{0,H}^{n+1},&k=2.
 \end{cases}
\end{equation}

\subsection{Cross-mesh BE/BDF2 polarization}
\begin{proposition}[Cross-mesh BE/BDF2 polarization identities]\label{prop:bdf-pol}
The phase history \eqref{eq:phase-history} and kinetic histories \eqref{eq:kinetic-transfer-fixed}--\eqref{eq:kinetic-reconstruction} satisfy the following polarization identities.
\begin{subequations}\label{eq:bdf-polarization}
For $k=1$,
\begin{align*}
 &\tau\left[\Cn^2\left(\nabla\phi_h^{n+1},\nabla\frac{\gamma_1}{\tau}(\phi_h^{n+1}-\phi_{H,h}^{n+1})\right)_{n+1}
 +s\left(\phi_h^{n+1},\frac{\gamma_1}{\tau}(\phi_h^{n+1}-\phi_{H,h}^{n+1})\right)_{n+1}\right]\notag\\
 &\qquad=\frac12\left[\norm{\phi_h^{n+1}}_{\Cn,s}^2-\norm{\phi_h^n}_{\Cn,s}^2
 +\norm{\phi_h^{n+1}-\phi_h^n}_{\Cn,s}^2\right],\qquad k=1.
\end{align*}
and, for $k=2$,
\begin{align}
 &\tau\left[\Cn^2\left(\nabla\phi_h^{n+1},\nabla\frac{\gamma_2}{\tau}(\phi_h^{n+1}-\phi_{H,h}^{n+1})\right)_{n+1}
 +s\left(\phi_h^{n+1},\frac{\gamma_2}{\tau}(\phi_h^{n+1}-\phi_{H,h}^{n+1})\right)_{n+1}\right]\notag\\
 &\quad=\frac14\Bigl[\norm{\phi_h^{n+1}}_{\Cn,s}^2+\norm{2\phi_h^{n+1}-\phi_h^n}_{\Cn,s}^2
 -\norm{\phi_h^n}_{\Cn,s}^2-\norm{2\phi_h^n-\phi_h^{n-1}}_{\Cn,s}^2
 +\norm{\phi_h^{n+1}-2\phi_h^n+\phi_h^{n-1}}_{\Cn,s}^2\Bigr].\label{eq:phase-polarization-bdf2}
\end{align}
Moreover,
\begin{align}
 \gamma_1(\bzeta_h^{n+1}-\bzeta_{H,h}^{n+1},\bzeta_h^{n+1})_{n+1}
 &=\frac12\Bigl[\norm{\bg_{0,h}^{n+1}}_{n+1}^2-\norm{\bg_{0,h}^{n}}_{n}^2
 +\norm{\bzeta_h^{n+1}-\bg_{0,H}^{n+1}}_{n+1}^2\Bigr],\qquad k=1,\notag\\
 \gamma_2(\bzeta_h^{n+1}-\bzeta_{H,h}^{n+1},\bzeta_h^{n+1})_{n+1}
 &=\frac14\Bigl[\norm{\bg_{0,h}^{n+1}}_{n+1}^2+\norm{\bg_{1,h}^{n+1}}_{n+1}^2
 -\norm{\bg_{0,h}^{n}}_{n}^2-\norm{\bg_{1,h}^{n}}_{n}^2
 +\norm{\bzeta_h^{n+1}-\bg_{1,H}^{n+1}}_{n+1}^2\Bigr],\qquad k=2.\label{eq:kinetic-polarization-bdf2}
\end{align}
\end{subequations}
\end{proposition}
\begin{proof}
Choosing $\chi_h=\phi_h^{n+1}$ in \eqref{eq:phase-history} gives
\begin{equation*}
 \Cn^2(\nabla\phi_{H,h}^{n+1},\nabla\phi_h^{n+1})_{n+1}+s(\phi_{H,h}^{n+1},\phi_h^{n+1})_{n+1}
 =\Cn^2\langle\nabla\phi_H^{n+1},\nabla\phi_h^{n+1}\rangle_\times+s\langle\phi_H^{n+1},\phi_h^{n+1}\rangle_\times.
\end{equation*}
Substitution of the BE/BDF2 histories from (2.8), followed by the standard two-point and BDF2 polarization identities, gives the stated BE identity and \eqref{eq:phase-polarization-bdf2}. For the kinetic part, both \eqref{eq:kinetic-transfer-fixed} and \eqref{eq:kinetic-transfer-normalized} give
\begin{equation*}
 \norm{\bg_{j,H}^{n+1}}_{n+1}=\norm{\bg_{j,h}^{n}}_n.
\end{equation*}
The $k=1$ identity follows from two-point polarization. For $k=2$, (2.13) gives
\begin{equation*}
 \bg_{1,h}^{n+1}=2\bzeta_h^{n+1}-\bg_{0,H}^{n+1},\qquad
 \bzeta_{H,h}^{n+1}=\frac{2\bg_{0,H}^{n+1}+\bg_{1,H}^{n+1}}{3}.
\end{equation*}
Direct expansion and the two norm-preservation identities yield \eqref{eq:kinetic-polarization-bdf2}.
\end{proof}

\section{Fully discrete schemes with mesh motion and remeshing}
We advance from the previous accepted levels to the current mesh $\Th^{n+1}$. The scalar histories $r_H^{n+1}$ and $q_H^{n+1}$ are defined by (2.8), with
\begin{equation*}
 r^0=\bigl((G(\phi_h^0),1)_0+C_G\bigr)^{1/2},\qquad q^0=1.
\end{equation*}

\noindent\textbf{Momentum-transport mesh velocity.} For the ALE-relative momentum transport, set
\begin{equation}
 \bw_h^{n+1,(k)}=\begin{cases}
 [D_\tau^{(k)}\bX_h^{n+1}]\circ(\bX_h^{n+1})^{-1},&\text{fixed-topology ALE update},\\
 0,&\text{physical-time substep after remeshing}.
 \end{cases}
\end{equation}

\subsection{Common steps}
\noindent\textbf{Step 1. Cross-mesh BE/BDF2 histories.} Form $\phi_E^{n+1}$ and $\bu_{E,h}^{n+1}$ by (2.8), set $m_E^{n+1}=m(\phi_E^{n+1})$, and determine $\phi_{H,h}^{n+1}\in\Sh^{n+1}$ from
\begin{equation*}
 \Cn^2(\nabla\phi_{H,h}^{n+1},\nabla\chi_h)_{n+1}+s(\phi_{H,h}^{n+1},\chi_h)_{n+1}
 =\Cn^2\langle\nabla\phi_H^{n+1},\nabla\chi_h\rangle_\times+s\langle\phi_H^{n+1},\chi_h\rangle_\times,
 \qquad\forall\chi_h\in\Sh^{n+1}.
\end{equation*}
Construct $\bzeta_{H,h}^{n+1}$ by \eqref{eq:kinetic-transfer-fixed}--\eqref{eq:kinetic-reconstruction}.

\noindent\textbf{Step 2. Compatible carrier.} At a topology-changing step, find $\lambda_h^{n+1}\in S_{h,0}^{n+1}$ and set $\bu_{c,h}^{n+1}$ by
\begin{equation}
 (\nabla\lambda_h^{n+1},\nabla\chi_h)_{n+1}=\langle\bu_{E,h}^{n+1},\nabla\chi_h\rangle_\times,
 \qquad\forall\chi_h\in S_{h,0}^{n+1},
 \qquad \bu_{c,h}^{n+1}=\bu_{E,h}^{n+1}-\nabla\lambda_h^{n+1}.
\end{equation}
On a fixed-topology interval, use the reference-space correction and Piola map of Remark~\ref{rem:piola-carrier}. In either case,
\begin{equation*}
 \langle\bu_{c,h}^{n+1},\nabla\chi_h\rangle_\times=0,
 \qquad\forall\chi_h\in\Sh^{n+1}.
\end{equation*}

\noindent\textbf{Step 3. Phase--SAV update and AGG transport.} Set
\begin{equation}
 U_E^{n+1}=\left(\langle G(\phi_E^{n+1}),1\rangle_\times+C_G\right)^{1/2}>0.
\end{equation}
Find $(\phi_h^{n+1},\mu_h^{n+1},r^{n+1})\in\Sh^{n+1}\times\Sh^{n+1}\times\R$ such that, for all $\omega_h,\chi_h\in\Sh^{n+1}$,
\begin{subequations}\label{eq:phase-sav}
\begin{align}
 &\left(\frac{\gamma_k}{\tau}(\phi_h^{n+1}-\phi_{H,h}^{n+1}),\omega_h\right)_{n+1}
 -\langle\phi_E^{n+1}\bu_{c,h}^{n+1},\nabla\omega_h\rangle_\times
 +\frac1\Pe\langle m_E^{n+1}\nabla\mu_h^{n+1},\nabla\omega_h\rangle_\times=0,\label{eq:phase-sav-a}\\
 &(\mu_h^{n+1},\chi_h)_{n+1}=\Cn^2(\nabla\phi_h^{n+1},\nabla\chi_h)_{n+1}
 +s(\phi_h^{n+1},\chi_h)_{n+1}
 +\frac{r^{n+1}}{U_E^{n+1}}\langle G'(\phi_E^{n+1}),\chi_h\rangle_\times,\label{eq:phase-sav-b}\\
 &D_\tau^{(k)}r^{n+1}=\frac{1}{2U_E^{n+1}}
 \left\langle G'(\phi_E^{n+1}),\frac{\gamma_k}{\tau}(\phi_h^{n+1}-\phi_{H,h}^{n+1})\right\rangle_\times.\label{eq:phase-sav-c}
\end{align}
\end{subequations}
Set
\begin{equation}
 \rho_h^{n+1}=\rho_0+\rho_\phi\phi_h^{n+1},\qquad
 \eta_h^{n+1}=\eta_0+\eta_\phi\phi_h^{n+1},
\end{equation}
\begin{equation}
 \bJ_{\rho,h}^{n+1}=\rho_E^{n+1}\bu_{c,h}^{n+1}-\frac{\rho_\phi}{\Pe}m_E^{n+1}\nabla\mu_h^{n+1},
 \qquad \rho_E^{n+1}=\rho_0+\rho_\phi\phi_E^{n+1},
\end{equation}
and
\begin{align}
 C_{\rho,h}^{n+1}(\bv_h,\boldsymbol z_h)
 =\frac12\Bigl[{}&\left\langle(\bJ_{\rho,h}^{n+1}-\rho_h^{n+1}\bw_h^{n+1,(k)})\cdot\nabla\bv_h,\boldsymbol z_h\right\rangle_\times\notag\\
 &-\left\langle(\bJ_{\rho,h}^{n+1}-\rho_h^{n+1}\bw_h^{n+1,(k)})\cdot\nabla\boldsymbol z_h,\bv_h\right\rangle_\times\Bigr],
 \qquad \bv_h,\boldsymbol z_h\in\Vh^{n+1}.
\end{align}
Define
\begin{equation}
 E_{\phi,r,k}^{\ell}=\begin{cases}
 \dfrac{1}{2\We\Cn}\norm{\phi_h^\ell}_{\Cn,s}^2+\dfrac{1}{\We\Cn}(r^\ell)^2+C_A,&k=1,\\[6pt]
 \dfrac{1}{4\We\Cn}\left[\norm{\phi_h^\ell}_{\Cn,s}^2+\norm{2\phi_h^\ell-\phi_h^{\ell-1}}_{\Cn,s}^2\right]
 +\dfrac{1}{2\We\Cn}\left[(r^\ell)^2+(2r^\ell-r^{\ell-1})^2\right]+C_A,&k=2.
 \end{cases}
\end{equation}

\subsection{Saddle-point scheme}
\noindent\textbf{Step 4. Momentum--pressure solve and determination of $q$.} Find $(\bu_h^{n+1},p_h^{n+1},q^{n+1})\in\Vh^{n+1}\times P_{h,0}^{n+1}\times\R_{>0}$ such that, for all $\bv_h\in\Vh^{n+1}$ and $\alpha_h\in P_{h,0}^{n+1}$,
\begin{subequations}\label{eq:saddle}
\begin{align}
 &\frac{\gamma_k}{\tau}\left((\rho_h^{n+1})^{1/2}\bu_h^{n+1}-\bzeta_{H,h}^{n+1},(\rho_h^{n+1})^{1/2}\bv_h\right)_{n+1}
 +C_{\rho,h}^{n+1}(\bu_h^{n+1},\bv_h)
 +\frac{2}{\Rey}(\eta_h^{n+1}D(\bu_h^{n+1}),D(\bv_h))_{n+1}\notag\\
 &\hspace{5em}-(p_h^{n+1},\nabla\cdot\bv_h)_{n+1}
 =(\bfv^{n+1},\bv_h)_{n+1}-\frac{q^{n+1}}{\We\Cn}(\phi_h^{n+1}\nabla\mu_h^{n+1},\bv_h)_{n+1},\label{eq:saddle-a}\\
 &(\nabla\cdot\bu_h^{n+1},\alpha_h)_{n+1}=0,\label{eq:saddle-b}\\
 &\frac{\gamma_kq^{n+1}}{\alpha_q\tau}(q^{n+1}-q_H^{n+1})
 +\frac{q^{n+1}-q^n}{\tau}E_{\phi,r,k}^{n}
 +\frac{q^{n+1}}{\We\Cn}\langle\phi_E^{n+1}\bu_{c,h}^{n+1},\nabla\mu_h^{n+1}\rangle_\times
 -\frac{q^{n+1}}{\We\Cn}(\phi_h^{n+1}\nabla\mu_h^{n+1},\bu_h^{n+1})_{n+1}=0.\label{eq:saddle-c}
\end{align}
\end{subequations}

\noindent\textbf{Step 5. Kinetic-history update.} Set
\begin{equation*}
 \bg_{0,h}^{n+1}=\bzeta_h^{n+1}.
\end{equation*}
When the BE start-up is followed by BDF2, initialize
\begin{equation*}
 \bg_{1,h}^{1}=2\bzeta_h^1-\bg_{0,H}^{1};
\end{equation*}
at every BDF2 level set
\begin{equation*}
 \bg_{1,h}^{n+1}=2\bzeta_h^{n+1}-\bg_{0,H}^{n+1}.
\end{equation*}

\subsection{Pressure-correction scheme}
For projection and pressure-correction methods for incompressible and variable-density flows, see \cite{GuermondMinevShen2006,GuermondQuartapelle2000,GuermondSalgado2009}. Steps 1--3 are unchanged. Prescribe $\pi_h^0\in P_{h,0}^0$. For this branch, $\nabla_{\rho,h}^{n+1}$ denotes the discrete pressure gradient determined by
\begin{equation}
 (\rho_h^{n+1}\nabla_{\rho,h}^{n+1}\pi_h,\bv_h)_{n+1}=-(\pi_h,\nabla\cdot\bv_h)_{n+1},
 \qquad \pi_h\in P_{h,0}^{n+1},\ \bv_h\in\Vh^{n+1}.
\end{equation}

\noindent\textbf{Step 4. Pressure-history representation.} Set
\begin{equation}
 \boldsymbol\xi_{p,h}^{n+1}=\begin{cases}
 (J_h^{n+1})^{-1/2}(J_h^n)^{1/2}(\rho_h^n)^{1/2}\nabla_{\rho,h}^n\pi_h^n\circ \bX_h^n\circ(\bX_h^{n+1})^{-1},&\text{fixed topology},\\
 (\rho_h^n)^{1/2}\nabla_{\rho,h}^n\pi_h^n,&\text{topology change}.
 \end{cases}
\end{equation}
Find $\pi_{\rm tr,h}^{n+1},\bar\pi_h^{n+1}\in P_{h,0}^{n+1}$ such that, for all $\alpha_h\in P_{h,0}^{n+1}$,
\begin{subequations}
\begin{align}
 &(\rho_h^{n+1}\nabla_{\rho,h}^{n+1}\pi_{\rm tr,h}^{n+1},\nabla_{\rho,h}^{n+1}\alpha_h)_{n+1}
 =\left\langle\boldsymbol\xi_{p,h}^{n+1},(\rho_h^{n+1})^{1/2}\nabla_{\rho,h}^{n+1}\alpha_h\right\rangle_\times,\label{eq:p-hist-a}\\
 &(\rho_h^{n+1}\nabla_{\rho,h}^{n+1}\bar\pi_h^{n+1},\nabla_{\rho,h}^{n+1}\alpha_h)_{n+1}
 =\left\langle\boldsymbol\xi_{p,h}^{n+1}+\frac{\gamma_k}{\tau}\bzeta_{H,h}^{n+1},(\rho_h^{n+1})^{1/2}\nabla_{\rho,h}^{n+1}\alpha_h\right\rangle_\times.\label{eq:p-hist-b}
\end{align}
\end{subequations}

\noindent\textbf{Step 5. Tentative velocity and determination of $q$.} Find the tentative response $\bu_{*,h}^{n+1}(q)\in\Vh^{n+1}$ and determine $q^{n+1}>0$ from the unified system
\begin{subequations}\label{eq:pc-tentative}
\begin{align}
 &\frac{\gamma_k}{\tau}\left((\rho_h^{n+1})^{1/2}\bu_{*,h}^{n+1}(q)-\bzeta_{H,h}^{n+1},(\rho_h^{n+1})^{1/2}\bv_h\right)_{n+1}
 +C_{\rho,h}^{n+1}(\bu_{*,h}^{n+1}(q),\bv_h)
 +\frac{2}{\Rey}(\eta_h^{n+1}D(\bu_{*,h}^{n+1}(q)),D(\bv_h))_{n+1}\notag\\
 &\qquad=(\bar\pi_h^{n+1},\nabla\cdot\bv_h)_{n+1}+(\bfv^{n+1},\bv_h)_{n+1}
 -\frac{q}{\We\Cn}(\phi_h^{n+1}\nabla\mu_h^{n+1},\bv_h)_{n+1},
 \qquad\forall\bv_h\in\Vh^{n+1},\label{eq:pc-tentative-a}\\
 &\frac{\gamma_kq^{n+1}}{\alpha_q\tau}(q^{n+1}-q_H^{n+1})
 +\frac{q^{n+1}-q^n}{\tau}E_{\phi,r,k}^{n}
 +\frac{q^{n+1}}{\We\Cn}\langle\phi_E^{n+1}\bu_{c,h}^{n+1},\nabla\mu_h^{n+1}\rangle_\times\notag\\
 &\qquad-\frac{q^{n+1}}{\We\Cn}(\phi_h^{n+1}\nabla\mu_h^{n+1},\bu_{*,h}^{n+1}(q^{n+1}))_{n+1}=0.\label{eq:pc-tentative-b}
\end{align}
\end{subequations}
Set $\bu_{*,h}^{n+1}=\bu_{*,h}^{n+1}(q^{n+1})$.

\noindent\textbf{Step 6. Pressure--velocity correction and history update.} Find $\delta\pi_h^{n+1}\in P_{h,0}^{n+1}$ and set $\bu_h^{n+1},\pi_h^{n+1}$ by
\begin{subequations}
\begin{align}
 &(\rho_h^{n+1}\nabla_{\rho,h}^{n+1}\delta\pi_h^{n+1},\nabla_{\rho,h}^{n+1}\alpha_h)_{n+1}
 =-\frac{\gamma_k}{\tau}(\nabla\cdot\bu_{*,h}^{n+1},\alpha_h)_{n+1},
 \qquad\forall\alpha_h\in P_{h,0}^{n+1},\label{eq:pc-corr-a}\\
 &\bu_h^{n+1}=\bu_{*,h}^{n+1}-\frac{\tau}{\gamma_k}\nabla_{\rho,h}^{n+1}\delta\pi_h^{n+1},
 \qquad \pi_h^{n+1}=\pi_{\rm tr,h}^{n+1}+\delta\pi_h^{n+1}.\label{eq:pc-corr-b}
\end{align}
\end{subequations}
Finally, set $\bg_{0,h}^{n+1}=\bzeta_h^{n+1}$. When the BE start-up is followed by BDF2, initialize $\bg_{1,h}^{1}=2\bzeta_h^1-\bg_{0,H}^{1}$; at every BDF2 level set $\bg_{1,h}^{n+1}=2\bzeta_h^{n+1}-\bg_{0,H}^{n+1}$.

\begin{remark}[Fixed-topology reference correction and Piola map]\label{rem:piola-carrier}
On a fixed-topology interval, let $\widehat S_h$ be the scalar finite element space on the common reference triangulation $\widehat{\mathcal T}_h$. Pull each accepted velocity back by the inverse contravariant Piola map,
\begin{equation*}
 \widehat{\bu}_h^m=J_h^m(\bF_h^m)^{-1}(\bu_h^m\circ\bX_h^m),
 \qquad m=n\ \text{and, for BDF2, }m=n-1,
\end{equation*}
and form $\widehat{\bu}_E^{\,n+1}$ by the same BE/BDF2 extrapolation. Find $\widehat\lambda_h^{\,n+1}\in\widehat S_{h,0}$ from
\begin{equation*}
 (\widehat\nabla\widehat\lambda_h^{\,n+1},\widehat\nabla\widehat\chi_h)_{\widehat\Omega}
 =(\widehat{\bu}_E^{\,n+1},\widehat\nabla\widehat\chi_h)_{\widehat\Omega},
 \qquad\forall\widehat\chi_h\in\widehat S_{h,0},
\end{equation*}
and push the corrected carrier to the current mesh:
\begin{equation}
 \bu_{c,h}^{n+1}\circ\bX_h^{n+1}
 =\frac1{J_h^{n+1}}\bF_h^{n+1}
 \bigl(\widehat{\bu}_E^{\,n+1}-\widehat\nabla\widehat\lambda_h^{\,n+1}\bigr).
\end{equation}
For $\widehat\chi_h=\chi_h\circ\bX_h^{n+1}\in\widehat S_h$, using its zero-mean representative in the correction equation, change of variables and the Piola identity give
\begin{equation*}
 (\bu_{c,h}^{n+1},\nabla\chi_h)_{n+1}
 =(\widehat{\bu}_E^{\,n+1}-\widehat\nabla\widehat\lambda_h^{\,n+1},\widehat\nabla\widehat\chi_h)_{\widehat\Omega}=0.
\end{equation*}
Thus the required compatibility is enforced directly in the phase space, with no inclusion relation required between $S_h^\ell$ and $P_h^\ell$. For smooth fixed-topology maps and histories, the reference-space BE/BDF2 extrapolation retains first- and second-order temporal consistency. After topology-changing remeshing, Step 2 uses the current-space correction instead.
\end{remark}

\section{Structural properties and modified energy analysis}
This section proves discrete phase conservation, the affine-density balance, stepwise unique solvability, positivity of $q$, and the modified energy balances for the saddle-point and pressure-correction schemes. It also records the scalar-defect identity and the local consistency with the continuous state $q=1$. Throughout the analysis, accepted states satisfy
\begin{equation*}
 0<\rho_*\le\rho(\phi_h^\ell)\le\rho^*,\qquad
 0<\eta_*\le\eta(\phi_h^\ell)\le\eta^*.
\end{equation*}

\subsection{Discrete phase conservation and affine-density balance}
\begin{theorem}[Discrete mass conservation]\label{thm:mass}
Let $\phi_{H,h}^{n+1}$ be defined by (2.11), let $\bu_{c,h}^{n+1}$ satisfy the carrier compatibility of Step 2, and define $\bJ_{\rho,h}^{n+1}$ by (3.6). This includes the fixed-topology reference correction and Piola map of Remark~\ref{rem:piola-carrier}. Then
\begin{equation}
 \left(\frac{\gamma_k}{\tau}(\phi_h^{n+1}-\phi_{H,h}^{n+1}),1\right)_{n+1}=0,
\end{equation}
and, with the BE start-up, $(\phi_h^{n+1},1)_{n+1}=(\phi_h^n,1)_n$ at every accepted time level. Since $\rho_h^\ell=\rho_0+\rho_\phi\phi_h^\ell$ on the fixed physical domain, the total discrete density is conserved as well,
\begin{equation*}
 (\rho_h^{n+1},1)_{n+1}=(\rho_h^n,1)_n.
\end{equation*}
Moreover, for every $\omega_h\in\Sh^{n+1}$,
\begin{equation}
 \left(\frac{\gamma_k}{\tau}[\rho_h^{n+1}-(\rho_0+\rho_\phi\phi_{H,h}^{n+1})],\omega_h\right)_{n+1}
 -\langle\bJ_{\rho,h}^{n+1},\nabla\omega_h\rangle_\times=0.
\end{equation}
\end{theorem}
\begin{proof}
By the carrier compatibility of Step 2, with the fixed-topology case justified by Remark~\ref{rem:piola-carrier}, taking $\omega_h=1$ in (3.4a) gives (4.1). By the mean-preserving Ritz history (2.11),
\begin{equation*}
 (\phi_{H,h}^{n+1},1)_{n+1}=\begin{cases}
 (\phi_h^n,1)_n,&k=1,\\[2pt]
 \dfrac{4(\phi_h^n,1)_n-(\phi_h^{n-1},1)_{n-1}}{3},&k=2,
 \end{cases}
\end{equation*}
so that
\begin{equation*}
 \begin{cases}
 (\phi_h^{n+1},1)_{n+1}-(\phi_h^n,1)_n=0,&k=1,\\
 3(\phi_h^{n+1},1)_{n+1}-4(\phi_h^n,1)_n+(\phi_h^{n-1},1)_{n-1}=0,&k=2.
 \end{cases}
\end{equation*}
For arbitrary $\omega_h\in\Sh^{n+1}$, multiply (3.4a) by $\rho_\phi$ and use
\begin{equation*}
 \rho_h^{n+1}-(\rho_0+\rho_\phi\phi_{H,h}^{n+1})=\rho_\phi(\phi_h^{n+1}-\phi_{H,h}^{n+1}),
 \qquad \rho_\phi\phi_E^{n+1}\bu_{c,h}^{n+1}=\rho_E^{n+1}\bu_{c,h}^{n+1}-\rho_0\bu_{c,h}^{n+1},
\end{equation*}
\begin{equation*}
 \bJ_h^{n+1}=-\frac{\rho_\phi}{\Pe}m_E^{n+1}\nabla\mu_h^{n+1},
 \qquad \rho_0\langle\bu_{c,h}^{n+1},\nabla\omega_h\rangle_\times=0.
\end{equation*}
Therefore,
\begin{align*}
 0={}&\left(\frac{\gamma_k}{\tau}[\rho_h^{n+1}-(\rho_0+\rho_\phi\phi_{H,h}^{n+1})],\omega_h\right)_{n+1}
 -\langle\rho_E^{n+1}\bu_{c,h}^{n+1},\nabla\omega_h\rangle_\times
 +\frac{\rho_\phi}{\Pe}\langle m_E^{n+1}\nabla\mu_h^{n+1},\nabla\omega_h\rangle_\times\\
={}&\left(\frac{\gamma_k}{\tau}[\rho_h^{n+1}-(\rho_0+\rho_\phi\phi_{H,h}^{n+1})],\omega_h\right)_{n+1}
 -\langle\rho_E^{n+1}\bu_{c,h}^{n+1}+\bJ_h^{n+1},\nabla\omega_h\rangle_\times\\
={}&\left(\frac{\gamma_k}{\tau}[\rho_h^{n+1}-(\rho_0+\rho_\phi\phi_{H,h}^{n+1})],\omega_h\right)_{n+1}
 -\langle\bJ_{\rho,h}^{n+1},\nabla\omega_h\rangle_\times,
\end{align*}
which is (4.2).
\end{proof}

\begin{remarkstar}[Momentum-transport neutrality]
The same $\bJ_{\rho,h}^{n+1}$ is used in (3.7); therefore
\begin{equation}
 C_{\rho,h}^{n+1}(\bv_h,\bv_h)=0,\qquad\forall\bv_h\in\Vh^{n+1},
\end{equation}
by direct cancellation of the two exchanged terms. Thus the mass-consistent flux produces no kinetic self-work.
\end{remarkstar}

\subsection{Stepwise unique solvability and positivity of \texorpdfstring{$q$}{q}}
For a phase state obtained from the CH--SAV update, the saddle-point momentum solution depends affinely on $q$,
\begin{equation}\label{eq:q-affine-response}
 \bu_h(q)=\bu_{0,h}^{n+1}+q\bu_{1,h}^{n+1},\qquad
 p_h(q)=p_{0,h}^{n+1}+qp_{1,h}^{n+1},
\end{equation}
and the tentative velocity has the analogous representation
\begin{equation*}
 \bu_{*,h}(q)=\bu_{*,0,h}^{n+1}+q\bu_{*,1,h}^{n+1}.
\end{equation*}
In both branches, the unit capillary response satisfies
\begin{equation}\label{eq:q-unit-response}
 \frac{\gamma_k}{\tau}\norm{(\rho_h^{n+1})^{1/2}\bu_{1,h}^{n+1}}_{n+1}^2
 +\frac{2}{\Rey}\norm{(\eta_h^{n+1})^{1/2}D(\bu_{1,h}^{n+1})}_{n+1}^2
 =-\frac1{\We\Cn}(\phi_h^{n+1}\nabla\mu_h^{n+1},\bu_{1,h}^{n+1})_{n+1}.
\end{equation}
In the pressure-correction branch, $\bu_{1,h}^{n+1}$ denotes the unit capillary response of the tentative velocity. Substitution of the affine response and \eqref{eq:q-unit-response} into the corresponding scalar equation gives
\begin{equation}\label{eq:q-quadratic}
 A_q(q^{n+1})^2+B_qq^{n+1}-C_q=0.
\end{equation}
For the saddle-point branch,
\begin{equation}\label{eq:q-coefficients}
 \begin{aligned}
 A_q={}&\frac{\gamma_k}{\alpha_q}
 +\gamma_k\norm{(\rho_h^{n+1})^{1/2}\bu_{1,h}^{n+1}}_{n+1}^2
 +\frac{2\tau}{\Rey}\norm{(\eta_h^{n+1})^{1/2}D(\bu_{1,h}^{n+1})}_{n+1}^2,\\
 B_q={}&-\frac{\gamma_k}{\alpha_q}q_H^{n+1}+E_{\phi,r,k}^{n}
 +\frac{\tau}{\We\Cn}\langle\phi_E^{n+1}\bu_{c,h}^{n+1},\nabla\mu_h^{n+1}\rangle_\times
 -\frac{\tau}{\We\Cn}(\phi_h^{n+1}\nabla\mu_h^{n+1},\bu_{0,h}^{n+1})_{n+1},\\
 C_q={}&q^nE_{\phi,r,k}^{n}.
 \end{aligned}
\end{equation}
For the pressure-correction branch one replaces $\bu_{0,h}^{n+1},\bu_{1,h}^{n+1}$ by the corresponding tentative-velocity responses; $A_q>0$ and $C_q=q^nE_{\phi,r,k}^{n}$ are unchanged.

\begin{assumption}[Stepwise admissibility]\label{ass:stepwise}
The current accepted step satisfies: (i) the mesh and history transfers of Sections 2.2--2.4 are admissible, $0<\rho_*\le\rho_h^{n+1}\le\rho^*$, $0<\eta_*\le\eta_h^{n+1}\le\eta^*$, $s>0$, $m_E^{n+1}\ge0$, $C_A>0$, and $q^n>0$; (ii) the saddle-point pair satisfies the discrete inf--sup condition; (iii) the pressure-correction density-weighted pressure form is positive definite on $P_{h,0}^{n+1}$.
\end{assumption}

\begin{theorem}[Stepwise unique solvability and positivity of $q$]\label{thm:solvability}
Under Assumption~\ref{ass:stepwise}, one BE/BDF2 step of either formulation has a unique discrete solution, and \eqref{eq:q-quadratic} has exactly one positive root $q^{n+1}$.
\end{theorem}
\begin{proof}
The phase-energy Ritz history is unique because $s>0$, and the zero-mean carrier correction is unique by scalar stiffness coercivity, either on the current mesh after remeshing or on the fixed reference mesh before the Piola pushforward. Let $(\delta\phi_h,\delta\mu_h,\delta r)$ be the difference of two CH--SAV solutions. From the homogeneous form of (3.4c),
\begin{equation*}
 \delta r=\frac1{2U_E^{n+1}}\langle G'(\phi_E^{n+1}),\delta\phi_h\rangle_\times.
\end{equation*}
Testing the homogeneous forms of (3.4a)--(3.4b) by $\delta\mu_h$ and $\delta\phi_h$ and combining with this identity gives
\begin{equation*}
 \frac{\gamma_k}{\tau}\norm{\delta\phi_h}_{\Cn,s}^2
 +\frac{2\gamma_k}{\tau}|\delta r|^2
 +\frac1\Pe\langle m_E^{n+1}\nabla\delta\mu_h,\nabla\delta\mu_h\rangle_\times=0,
\end{equation*}
Since $s>0$ and $m_E^{n+1}\ge0$, this identity first gives $\delta\phi_h=0$ and $\delta r=0$. Returning to the homogeneous chemical-potential equation (3.4b) then yields
\begin{equation*}
 (\delta\mu_h,\chi_h)_{n+1}=0,\qquad\forall\chi_h\in\Sh^{n+1},
\end{equation*}
so $\delta\mu_h=0$. This final step does not require a positive lower bound for $m_E^{n+1}$; hence the argument remains valid for degenerate mobilities.

For the homogeneous saddle-point response,
\begin{equation*}
 \frac{\gamma_k}{\tau}\norm{(\rho_h^{n+1})^{1/2}\delta\bu_h}_{n+1}^2
 +\frac{2}{\Rey}\norm{(\eta_h^{n+1})^{1/2}D(\delta\bu_h)}_{n+1}^2=0,
\end{equation*}
so $\delta\bu_h=0$, and the discrete inf--sup condition yields $\delta p_h=0$. The tentative-velocity response has the same coercive mass--viscous--skew identity, while (3.12) and (3.14) are unique by positive definiteness of the current pressure form.

Finally, $E_{\phi,r,k}^{n}\ge C_A>0$ and $q^n>0$ imply
\begin{equation*}
 A_q>0,\qquad C_q=q^nE_{\phi,r,k}^{n}>0,\qquad B_q^2+4A_qC_q>0,
 \qquad q_+q_-=-\frac{C_q}{A_q}<0.
\end{equation*}
Hence \eqref{eq:q-quadratic} has exactly one positive root,
\begin{equation}\label{eq:q-positive-root}
 q^{n+1}=\frac{-B_q+\sqrt{B_q^2+4A_qC_q}}{2A_q}>0.
\end{equation}
Once the positive root is fixed, \eqref{eq:q-affine-response}, or the corresponding affine tentative-velocity response, uniquely reconstructs the velocity. The pressure-correction branch then completes the pressure update and the discretely divergence-free velocity correction through (3.14). Thus the entire time step is uniquely solvable. Starting from $q^0=1$, induction gives
\begin{equation}\label{eq:q-positive-induction}
 q^n>0,\qquad 0\le n\le N,
\end{equation}
for every finite accepted time level $N$.
\end{proof}

\begin{proposition}[Scalar-defect identity and local consistency]\label{prop:q-consistency}
Let $q^{n+1}>0$ be the positive scalar produced by Theorem~\ref{thm:solvability}. For the saddle-point branch define
\begin{equation*}
 R_q^{n+1}:=\frac1{\We\Cn}\left[
 \langle\phi_E^{n+1}\bu_{c,h}^{n+1},\nabla\mu_h^{n+1}\rangle_\times
 -(\phi_h^{n+1}\nabla\mu_h^{n+1},\bu_h^{n+1})_{n+1}\right],
\end{equation*}
and for the pressure-correction branch replace the terminal velocity by $\bu_{*,h}^{n+1}$. Then the scalar update satisfies the exact defect identity
\begin{equation}
 \left(E_{\phi,r,k}^{n}+\frac{\gamma_k}{\alpha_q}q^{n+1}\right)(q^{n+1}-1)
 =E_{\phi,r,k}^{n}(q^n-1)+\frac{\gamma_k}{\alpha_q}q^{n+1}(q_H^{n+1}-1)-\tau q^{n+1}R_q^{n+1}.
\label{eq:q-defect}
\end{equation}
Consequently, if the incoming scalar histories are consistent, $q^n=q_H^{n+1}=1$, then
\begin{equation}
 |q^{n+1}-1|\le\frac{\alpha_q}{\gamma_k}\tau|R_q^{n+1}|.
\label{eq:q-local-bound}
\end{equation}
In particular, under consistent scalar histories, $R_q^{n+1}=0$ implies $q^{n+1}=1$.

In the purely temporal consistency setting, suppose the spatial and cross-mesh pairings are exact, the exact solution is smooth, and the phase extrapolation and compatible carrier satisfy
\begin{equation*}
 \|\phi_E^{n+1}-\phi^{n+1}\|_{L^\infty(\Omega)}
 +\|\bu_c^{n+1}-\bu^{n+1}\|_{L^2(\Omega)}=O(\tau^k).
\end{equation*}
For the pressure-correction branch, assume in addition that the tentative velocity satisfies
\begin{equation*}
 \|\bu_*^{n+1}-\bu^{n+1}\|_{L^2(\Omega)}=O(\tau^k).
\end{equation*}
Assume also
\begin{equation*}
 \|\bu^{n+1}\|_{L^2(\Omega)}+\|\phi_E^{n+1}\|_{L^\infty(\Omega)}
 +\|\nabla\mu^{n+1}\|_{L^2(\Omega)}\le C
\end{equation*}
uniformly in $\tau$. Then $R_q^{n+1}=O(\tau^k)$ and, for consistent incoming scalar histories,
\begin{equation}
 q^{n+1}-1=O(\tau^{k+1}).
\label{eq:q-local-consistency}
\end{equation}
Thus the scalar equation is locally consistent with the continuous branch $q\equiv1$; no global fully discrete error estimate is asserted here.
\end{proposition}
\begin{proof}
Multiplying the scalar equation by $\tau$ and rearranging gives \eqref{eq:q-defect}. If $q^n=q_H^{n+1}=1$, division by the positive scalar $q^{n+1}$ yields
\begin{equation*}
 \left(\frac{E_{\phi,r,k}^{n}}{q^{n+1}}+\frac{\gamma_k}{\alpha_q}\right)(q^{n+1}-1)=-\tau R_q^{n+1}.
\end{equation*}
Since $E_{\phi,r,k}^{n}>0$, estimate \eqref{eq:q-local-bound} follows. In particular, $R_q^{n+1}=0$ forces $q^{n+1}=1$.

For the local temporal statement, evaluated on the smooth exact state, the saddle-point exchange defect is
\begin{equation*}
 \We\Cn R_q^{n+1}
 =\left((\phi_E^{n+1}-\phi^{n+1})\bu^{n+1}
 +\phi_E^{n+1}(\bu_c^{n+1}-\bu^{n+1}),\nabla\mu^{n+1}\right).
\end{equation*}
For the pressure-correction branch, the corresponding identity contains the additional term
\begin{equation*}
 \left(\phi^{n+1}(\bu^{n+1}-\bu_*^{n+1}),\nabla\mu^{n+1}\right).
\end{equation*}
The stated consistency and boundedness assumptions, together with the tentative-velocity assumption in the pressure-correction branch, therefore give $R_q^{n+1}=O(\tau^k)$, and \eqref{eq:q-local-bound} yields \eqref{eq:q-local-consistency}.
\end{proof}

\subsection{Modified energy law of the saddle-point scheme}
Define the total modified energy and the polarization remainders. For the BE start-up and BDF2 main step, respectively, set
\begin{subequations}\label{eq:modified-energy}
\begin{align}
 E_{h,1}^{\ell}&=\frac12\norm{\bg_{0,h}^{\ell}}_\ell^2+q^\ell E_{\phi,r,1}^{\ell}+\frac1{2\alpha_q}(q^\ell)^2,\label{eq:modified-energy-a}\\
 E_{h,2}^{\ell}&=\frac14\left[\norm{\bg_{0,h}^{\ell}}_\ell^2+\norm{\bg_{1,h}^{\ell}}_\ell^2\right]
 +q^\ell E_{\phi,r,2}^{\ell}+\frac1{4\alpha_q}\left[(q^\ell)^2+(2q^\ell-q^{\ell-1})^2\right].\label{eq:modified-energy-b}
\end{align}
The corresponding polarization remainders are
\begin{align}
 N_{h,1}^{\ell}&=\frac12\norm{\bzeta_h^\ell-\bg_{0,H}^{\ell}}_\ell^2
 +q^\ell\left[\frac1{2\We\Cn}\norm{\phi_h^\ell-\phi_h^{\ell-1}}_{\Cn,s}^2
 +\frac1{\We\Cn}(r^\ell-r^{\ell-1})^2\right]
 +\frac1{2\alpha_q}(q^\ell-q^{\ell-1})^2,\label{eq:Nh1}\\
 N_{h,2}^{\ell}&=\frac14\norm{\bzeta_h^\ell-\bg_{1,H}^{\ell}}_\ell^2
 +q^\ell\left[\frac1{4\We\Cn}\norm{\phi_h^\ell-2\phi_h^{\ell-1}+\phi_h^{\ell-2}}_{\Cn,s}^2
 +\frac1{2\We\Cn}(r^\ell-2r^{\ell-1}+r^{\ell-2})^2\right]\notag\\
 &\qquad+\frac1{4\alpha_q}(q^\ell-2q^{\ell-1}+q^{\ell-2})^2.\label{eq:Nh2}
\end{align}
\end{subequations}

\begin{theorem}[Modified energy balance of the saddle-point scheme]\label{thm:saddle-energy}
Assume that the current mesh satisfies the assumptions of Section 2.2, the accepted-state bounds above hold, and $q^{n+1}>0$. Then the scheme (3.4)--(3.9) satisfies
\begin{align}\label{eq:saddle-energy-law}
 &E_{h,k}^{n+1}-E_{h,k}^{n}+N_{h,k}^{n+1}
 +\frac{2\tau}{\Rey}\norm{(\eta_h^{n+1})^{1/2}D(\bu_h^{n+1})}_{n+1}^2\notag\\
 &\qquad+\frac{\tau q^{n+1}}{\Pe\We\Cn}\langle m_E^{n+1}\nabla\mu_h^{n+1},\nabla\mu_h^{n+1}\rangle_\times
 =\tau(\bfv^{n+1},\bu_h^{n+1})_{n+1},\qquad k=1,2.
\end{align}
\end{theorem}
\begin{proof}
We present the details for $k=2$ and set $\gamma_2=3/2$. We first derive the phase--SAV identity, add the scalar equation, and then combine the result with the momentum equation.
The case $k=1$ follows by replacing each BDF2 polarization below with the corresponding two-point BE polarization.
\begin{subequations}\label{eq:saddle-phase-sav-chain}
Choose $\omega_h=\mu_h^{n+1}$ in (3.4a) to obtain
\begin{equation}\label{eq:saddle-phase-sav-a}
 (D_\tau^{(2)}\phi_h^{n+1},\mu_h^{n+1})_{n+1}
 -\langle\phi_E^{n+1}\bu_{c,h}^{n+1},\nabla\mu_h^{n+1}\rangle_\times
 +\frac1\Pe\langle m_E^{n+1}\nabla\mu_h^{n+1},\nabla\mu_h^{n+1}\rangle_\times=0.
\end{equation}
Here $D_\tau^{(2)}\phi_h^{n+1}=\frac{\gamma_2}{\tau}(\phi_h^{n+1}-\phi_{H,h}^{n+1})$. On the other hand, choose $\chi_h=D_\tau^{(2)}\phi_h^{n+1}$ in (3.4b) and use (3.4c) to obtain
\begin{align}\label{eq:saddle-phase-sav-b}
 (\mu_h^{n+1},D_\tau^{(2)}\phi_h^{n+1})_{n+1}
 ={}&\Cn^2(\nabla\phi_h^{n+1},\nabla D_\tau^{(2)}\phi_h^{n+1})_{n+1}
 +s(\phi_h^{n+1},D_\tau^{(2)}\phi_h^{n+1})_{n+1}\notag+\frac{r^{n+1}}{U_E^{n+1}}\langle G'(\phi_E^{n+1}),D_\tau^{(2)}\phi_h^{n+1}\rangle_\times\notag\\
 ={}&\Cn^2(\nabla\phi_h^{n+1},\nabla D_\tau^{(2)}\phi_h^{n+1})_{n+1}
 +s(\phi_h^{n+1},D_\tau^{(2)}\phi_h^{n+1})_{n+1}
 +2r^{n+1}D_\tau^{(2)}r^{n+1}.
\end{align}
Substituting \eqref{eq:saddle-phase-sav-b} into \eqref{eq:saddle-phase-sav-a} gives the complete phase--SAV balance
\begin{align}\label{eq:saddle-phase-sav-c}
 &\Cn^2(\nabla\phi_h^{n+1},\nabla D_\tau^{(2)}\phi_h^{n+1})_{n+1}
 +s(\phi_h^{n+1},D_\tau^{(2)}\phi_h^{n+1})_{n+1}
 +2r^{n+1}D_\tau^{(2)}r^{n+1}\notag\\
 &+\frac1\Pe\langle m_E^{n+1}\nabla\mu_h^{n+1},\nabla\mu_h^{n+1}\rangle_\times
 -\langle\phi_E^{n+1}\bu_{c,h}^{n+1},\nabla\mu_h^{n+1}\rangle_\times = 0.
\end{align}
\end{subequations}
Next multiply \eqref{eq:saddle-phase-sav-c} by $\tau q^{n+1}/(\We\Cn)$ and add $\tau$ times (3.9c). The carrier-exchange terms cancel term by term in the same physical pairing, giving the following identity before the BDF2 polarization is applied:
\begin{subequations}\label{eq:saddle-phase-scalar-chain}
\begin{align}\label{eq:saddle-phase-scalar-a}
 &\frac{\tau q^{n+1}}{\We\Cn}\Bigl[\Cn^2(\nabla\phi_h^{n+1},\nabla D_\tau^{(2)}\phi_h^{n+1})_{n+1}
 +s(\phi_h^{n+1},D_\tau^{(2)}\phi_h^{n+1})_{n+1}\Bigr]
 +\frac{2\tau q^{n+1}}{\We\Cn}r^{n+1}D_\tau^{(2)}r^{n+1}\notag\\
 &+\frac{\gamma_2q^{n+1}}{\alpha_q}(q^{n+1}-q_H^{n+1})+(q^{n+1}-q^n)E_{\phi,r,2}^{n}
 +\frac{\tau q^{n+1}}{\Pe\We\Cn}\langle m_E^{n+1}\nabla\mu_h^{n+1},\nabla\mu_h^{n+1}\rangle_\times\notag\\
 &-\frac{\tau q^{n+1}}{\We\Cn}(\phi_h^{n+1}\nabla\mu_h^{n+1},\bu_h^{n+1})_{n+1}=0.
\end{align}
Apply the phase-field polarization \eqref{eq:phase-polarization-bdf2} from Proposition~\ref{prop:bdf-pol} and the BDF2 polarization identities for $r$ and $q$ to the first three time terms. Using
\begin{equation*}
 q^{n+1}(E^{n+1}-E^n)+(q^{n+1}-q^n)E^n=q^{n+1}E^{n+1}-q^nE^n,
\end{equation*}
equation \eqref{eq:saddle-phase-scalar-a} becomes
\begin{align}\label{eq:saddle-phase-scalar-b}
 &q^{n+1}E_{\phi,r,2}^{n+1}-q^nE_{\phi,r,2}^{n}
 +\frac1{4\alpha_q}\Bigl[(q^{n+1})^2+(2q^{n+1}-q^n)^2-(q^n)^2-(2q^n-q^{n-1})^2\Bigr]\notag\\
 &+\frac{q^{n+1}}{4\We\Cn}\norm{\phi_h^{n+1}-2\phi_h^n+\phi_h^{n-1}}_{\Cn,s}^2
 +\frac{q^{n+1}}{2\We\Cn}(r^{n+1}-2r^n+r^{n-1})^2
 +\frac1{4\alpha_q}(q^{n+1}-2q^n+q^{n-1})^2\notag\\
 &+\frac{\tau q^{n+1}}{\Pe\We\Cn}\langle m_E^{n+1}\nabla\mu_h^{n+1},\nabla\mu_h^{n+1}\rangle_\times
 -\frac{\tau q^{n+1}}{\We\Cn}(\phi_h^{n+1}\nabla\mu_h^{n+1},\bu_h^{n+1})_{n+1}=0.
\end{align}
\end{subequations}
Next, test the momentum equation. Choose $\bv_h=\bu_h^{n+1}$ in (3.9a), multiply by $\tau$, and choose $\alpha_h=p_h^{n+1}$ in (3.9b). Since
\begin{equation*}
 C_{\rho,h}^{n+1}(\bu_h^{n+1},\bu_h^{n+1})=0,
 \qquad (p_h^{n+1},\nabla\cdot\bu_h^{n+1})_{n+1}=0,
\end{equation*}
the unpolarized momentum identity, followed by the kinetic BDF2 polarization \eqref{eq:kinetic-polarization-bdf2}, gives
\begin{align}\label{eq:saddle-momentum}
 \tau(\bfv^{n+1},\bu_h^{n+1})_{n+1}
 ={}&\gamma_2(\bzeta_h^{n+1}-\bzeta_{H,h}^{n+1},\bzeta_h^{n+1})_{n+1}
 +\frac{2\tau}{\Rey}\norm{(\eta_h^{n+1})^{1/2}D(\bu_h^{n+1})}_{n+1}^2
 +\frac{\tau q^{n+1}}{\We\Cn}(\phi_h^{n+1}\nabla\mu_h^{n+1},\bu_h^{n+1})_{n+1}\notag\\
 ={}&\frac14\Bigl[\norm{\bg_{0,h}^{n+1}}_{n+1}^2+\norm{\bg_{1,h}^{n+1}}_{n+1}^2
 -\norm{\bg_{0,h}^{n}}_{n}^2-\norm{\bg_{1,h}^{n}}_{n}^2
 +\norm{\bzeta_h^{n+1}-\bg_{1,H}^{n+1}}_{n+1}^2\Bigr]\notag\\
 &+\frac{2\tau}{\Rey}\norm{(\eta_h^{n+1})^{1/2}D(\bu_h^{n+1})}_{n+1}^2
 +\frac{\tau q^{n+1}}{\We\Cn}(\phi_h^{n+1}\nabla\mu_h^{n+1},\bu_h^{n+1})_{n+1}.
\end{align}
The only remaining pair of reversible exchange terms in \eqref{eq:saddle-phase-scalar-b} and \eqref{eq:saddle-momentum} is
\begin{equation*}
 -\frac{\tau q^{n+1}}{\We\Cn}(\phi_h^{n+1}\nabla\mu_h^{n+1},\bu_h^{n+1})_{n+1}
 +\frac{\tau q^{n+1}}{\We\Cn}(\phi_h^{n+1}\nabla\mu_h^{n+1},\bu_h^{n+1})_{n+1}=0.
\end{equation*}
Therefore, adding \eqref{eq:saddle-phase-scalar-b} and \eqref{eq:saddle-momentum} cancels the capillary work exactly. Collecting the storage and numerical-dissipation terms according to \eqref{eq:modified-energy-b} and \eqref{eq:Nh2} yields
\begin{align}\label{eq:saddle-energy-bdf2}
 & E_{h,2}^{n+1}-E_{h,2}^{n}+N_{h,2}^{n+1}
 +\frac{2\tau}{\Rey}\norm{(\eta_h^{n+1})^{1/2}D(\bu_h^{n+1})}_{n+1}^2\notag\\
 &{}+\frac{\tau q^{n+1}}{\Pe\We\Cn}\langle m_E^{n+1}\nabla\mu_h^{n+1},\nabla\mu_h^{n+1}\rangle_\times
 -\tau(\bfv^{n+1},\bu_h^{n+1})_{n+1} = 0,
\end{align}
which is \eqref{eq:saddle-energy-law} for $k=2$. The case $k=1$ follows along the same chain of identities by replacing the BDF2 polarizations of $\phi_h$, $r$, $q$, and the kinetic histories with their BE counterparts.

The first time step is computed by BE and is used to initialize the BDF2 kinetic histories according to (2.13) and (2.17). The BDF2 modified-energy sequence is defined by \eqref{eq:modified-energy}--\eqref{eq:saddle-energy-law} from the first $k=2$ step onward; the connection with the BE start-up is stated after Theorem~\ref{thm:saddle-energy}. Equation~\eqref{eq:saddle-energy-law} is a discrete balance for the SAV--$q$ modified energy. If $\bfv=0$, then $q^{n+1}>0$ from Theorem~\ref{thm:solvability}, $m_E^{n+1}\ge0$, and the accepted-state bounds imply that every dissipation term in \eqref{eq:saddle-energy-law} is nonnegative. No time-step restriction enters this estimate. Hence, under the stated admissibility assumptions, the modified energy is nonincreasing for arbitrary $\tau>0$.
\end{proof}

\subsection{Modified energy law of the pressure-correction scheme}
By (3.11), the fixed-topology isometry and the unchanged physical field across a topology change both give
\begin{equation*}
 \norm{\boldsymbol\xi_{p,h}^{n+1}}^2=\norm{(\rho_h^n)^{1/2}\nabla_{\rho,h}^n\pi_h^n}_n^2.
\end{equation*}
Equation (3.12a) is then the $L^2$-orthogonal projection of this pressure-history field onto the current density-weighted gradient image space. Hence the Pythagorean decomposition gives
\begin{equation}\label{eq:pc-pythag}
 \norm{\boldsymbol\xi_{p,h}^{n+1}}^2
 =\norm{(\rho_h^{n+1})^{1/2}\nabla_{\rho,h}^{n+1}\pi_{\rm tr,h}^{n+1}}_{n+1}^2
 +\norm{\boldsymbol\xi_{p,h}^{n+1}-(\rho_h^{n+1})^{1/2}\nabla_{\rho,h}^{n+1}\pi_{\rm tr,h}^{n+1}}^2.
\end{equation}
The second term is a nonnegative projection remainder in the algorithmic pressure storage. Stepwise solvability of (3.12) and (3.14) has already been included in Theorem~\ref{thm:solvability}; here we use only the resulting orthogonality to track the pressure-gradient term. Define the pressure-correction energy and the additional remainder by
\begin{subequations}
\begin{align}
 E_{h,k}^{\ell,\rm pc}&=E_{h,k}^{\ell}+\frac{\tau^2}{2\gamma_k}\norm{(\rho_h^\ell)^{1/2}\nabla_{\rho,h}^{\ell}\pi_h^\ell}_\ell^2,\label{eq:pc-energy-a}\\
 N_{p,h,k}^{n+1}&=\frac{\tau^2}{2\gamma_k}\norm{(\rho_h^{n+1})^{1/2}\nabla_{\rho,h}^{n+1}\delta\pi_h^{n+1}}_{n+1}^2
 +\frac{\tau^2}{2\gamma_k}\norm{\boldsymbol\xi_{p,h}^{n+1}-(\rho_h^{n+1})^{1/2}\nabla_{\rho,h}^{n+1}\pi_{\rm tr,h}^{n+1}}^2.\label{eq:pc-energy-b}
\end{align}
\end{subequations}

\begin{theorem}[Modified energy balance of the pressure-correction scheme]\label{thm:pc-energy}
Assume that the current mesh satisfies the assumptions of Section 2.2, the accepted-state bounds above hold, and $q^{n+1}>0$. Then the pressure-correction scheme satisfies
\begin{align}\label{eq:pc-energy-law}
 &E_{h,k}^{n+1,\rm pc}-E_{h,k}^{n,\rm pc}+N_{h,k}^{n+1}+N_{p,h,k}^{n+1}
 +\frac{2\tau}{\Rey}\norm{(\eta_h^{n+1})^{1/2}D(\bu_{*,h}^{n+1})}_{n+1}^2\notag\\
 &\qquad+\frac{\tau q^{n+1}}{\Pe\We\Cn}\langle m_E^{n+1}\nabla\mu_h^{n+1},\nabla\mu_h^{n+1}\rangle_\times
 =\tau(\bfv^{n+1},\bu_{*,h}^{n+1})_{n+1},\qquad k=1,2.
\end{align}
\end{theorem}
\begin{proof}
Again we present the case $k=2$ and set $\gamma_2=3/2$. The phase--SAV testing is identical to that of the saddle-point branch, so we start from the common identity \eqref{eq:saddle-phase-sav-c}. The only difference is that the capillary work in the scalar equation is evaluated with the tentative velocity $\bu_{*,h}^{n+1}$ and the momentum identity also contains the pressure-correction term. Multiply \eqref{eq:saddle-phase-sav-c} by $\tau q^{n+1}/(\We\Cn)$ and write $\tau$ times (3.13b) as
\begin{equation*}
\frac{\gamma_2q^{n+1}}{\alpha_q}(q^{n+1}-q_H^{n+1})+(q^{n+1}-q^n)E_{\phi,r,2}^{n}
 +\frac{\tau q^{n+1}}{\We\Cn}\langle\phi_E^{n+1}\bu_{c,h}^{n+1},\nabla\mu_h^{n+1}\rangle_\times
 -\frac{\tau q^{n+1}}{\We\Cn}(\phi_h^{n+1}\nabla\mu_h^{n+1},\bu_{*,h}^{n+1})_{n+1}=0.
\end{equation*}
Adding this relation to the weighted form of \eqref{eq:saddle-phase-sav-c} gives the analogue of \eqref{eq:saddle-phase-scalar-a} with $\bu_h^{n+1}$ replaced by $\bu_{*,h}^{n+1}$. The same BDF2 polarizations therefore yield
\begin{align}\label{eq:pc-phase-balance}
 &q^{n+1}E_{\phi,r,2}^{n+1}-q^nE_{\phi,r,2}^{n}
 +\frac1{4\alpha_q}\Bigl[(q^{n+1})^2+(2q^{n+1}-q^n)^2-(q^n)^2-(2q^n-q^{n-1})^2\Bigr]\notag\\
 &+\frac{q^{n+1}}{4\We\Cn}\norm{\phi_h^{n+1}-2\phi_h^n+\phi_h^{n-1}}_{\Cn,s}^2
 +\frac{q^{n+1}}{2\We\Cn}(r^{n+1}-2r^n+r^{n-1})^2
 +\frac1{4\alpha_q}(q^{n+1}-2q^n+q^{n-1})^2\notag\\
 &+\frac{\tau q^{n+1}}{\Pe\We\Cn}\langle m_E^{n+1}\nabla\mu_h^{n+1},\nabla\mu_h^{n+1}\rangle_\times
 -\frac{\tau q^{n+1}}{\We\Cn}(\phi_h^{n+1}\nabla\mu_h^{n+1},\bu_{*,h}^{n+1})_{n+1}=0.
\end{align}
We next treat the tentative velocity and pressure terms specific to the pressure-correction branch. Combining the pressure-increment equation (3.14a), the velocity correction (3.14b), and the density-weighted discrete-gradient definition (3.10) gives
\begin{equation*}
 (\nabla\cdot\bu_h^{n+1},\alpha_h)_{n+1}=0\qquad\forall\alpha_h\in P_{h,0}^{n+1},
\end{equation*}
and therefore
\begin{equation*}
 (\rho_h^{n+1}\bu_h^{n+1},\nabla_{\rho,h}^{n+1}\alpha_h)_{n+1}=0
 \qquad\forall\alpha_h\in P_{h,0}^{n+1}.
\end{equation*}
The velocity correction (3.14b) then implies
\begin{equation*}
 (\rho_h^{n+1})^{1/2}\bu_{*,h}^{n+1}=(\rho_h^{n+1})^{1/2}\bu_h^{n+1}
 +\frac{\tau}{\gamma_2}(\rho_h^{n+1})^{1/2}\nabla_{\rho,h}^{n+1}\delta\pi_h^{n+1},
\end{equation*}
\begin{equation*}
 \bigl((\rho_h^{n+1})^{1/2}\bu_h^{n+1},(\rho_h^{n+1})^{1/2}\nabla_{\rho,h}^{n+1}\delta\pi_h^{n+1}\bigr)_{n+1}=0.
\end{equation*}
Subtracting (3.12a) from (3.12b) and choosing $\alpha_h=\delta\pi_h^{n+1}$ gives
\begin{equation}\label{eq:pc-pressure-cross}
 \left((\rho_h^{n+1})^{1/2}\nabla_{\rho,h}^{n+1}(\bar\pi_h^{n+1}-\pi_{\rm tr,h}^{n+1}),
 (\rho_h^{n+1})^{1/2}\nabla_{\rho,h}^{n+1}\delta\pi_h^{n+1}\right)_{n+1}
 =\frac{\gamma_2}{\tau}\left\langle\bzeta_{H,h}^{n+1},(\rho_h^{n+1})^{1/2}\nabla_{\rho,h}^{n+1}\delta\pi_h^{n+1}\right\rangle_\times.
\end{equation}
Moreover, $\pi_h^{n+1}=\pi_{\rm tr,h}^{n+1}+\delta\pi_h^{n+1}$ yields
\begin{align}\label{eq:pc-pressure-polarization}
 &2(\rho_h^{n+1}\nabla_{\rho,h}^{n+1}\pi_{\rm tr,h}^{n+1},\nabla_{\rho,h}^{n+1}\delta\pi_h^{n+1})_{n+1}
 +2\norm{(\rho_h^{n+1})^{1/2}\nabla_{\rho,h}^{n+1}\delta\pi_h^{n+1}}_{n+1}^2\notag\\
 &\qquad=\norm{(\rho_h^{n+1})^{1/2}\nabla_{\rho,h}^{n+1}\pi_h^{n+1}}_{n+1}^2
 -\norm{(\rho_h^{n+1})^{1/2}\nabla_{\rho,h}^{n+1}\pi_{\rm tr,h}^{n+1}}_{n+1}^2
 +\norm{(\rho_h^{n+1})^{1/2}\nabla_{\rho,h}^{n+1}\delta\pi_h^{n+1}}_{n+1}^2.
\end{align}
\begin{subequations}\label{eq:pc-momentum-chain}
Now choose $\bv_h=\bu_{*,h}^{n+1}$ in (3.13a) and multiply by $\tau$. The skew-symmetric transport vanishes, giving first the tentative-velocity identity before reorganizing the pressure terms,
\begin{align}\label{eq:pc-momentum-a}
 \tau(\bfv^{n+1},\bu_{*,h}^{n+1})_{n+1}
 ={}&\gamma_2(\bzeta_{*,h}^{n+1}-\bzeta_{H,h}^{n+1},\bzeta_{*,h}^{n+1})_{n+1}
 -\tau(\bar\pi_h^{n+1},\nabla\cdot\bu_{*,h}^{n+1})_{n+1}
 +\frac{2\tau}{\Rey}\norm{(\eta_h^{n+1})^{1/2}D(\bu_{*,h}^{n+1})}_{n+1}^2\notag\\
 &+\frac{\tau q^{n+1}}{\We\Cn}(\phi_h^{n+1}\nabla\mu_h^{n+1},\bu_{*,h}^{n+1})_{n+1},
\end{align}
where $\bzeta_{*,h}^{n+1}=(\rho_h^{n+1})^{1/2}\bu_{*,h}^{n+1}$. By the velocity correction and discrete gradient orthogonality, the time term expands as
\begin{align*}
 \gamma_2(\bzeta_{*,h}^{n+1}-\bzeta_{H,h}^{n+1},\bzeta_{*,h}^{n+1})_{n+1}
 ={}&\gamma_2(\bzeta_h^{n+1}-\bzeta_{H,h}^{n+1},\bzeta_h^{n+1})_{n+1}
 -\tau\left\langle\bzeta_{H,h}^{n+1},(\rho_h^{n+1})^{1/2}\nabla_{\rho,h}^{n+1}\delta\pi_h^{n+1}\right\rangle_\times\\
 &+\frac{\tau^2}{\gamma_2}\norm{(\rho_h^{n+1})^{1/2}\nabla_{\rho,h}^{n+1}\delta\pi_h^{n+1}}_{n+1}^2.
\end{align*}
On the other hand, since the corrected velocity $\bu_h^{n+1}$ is orthogonal to the current discrete gradient space, the density-weighted discrete-gradient definition (3.10) and the velocity correction imply
\begin{equation*}
 -\tau(\bar\pi_h^{n+1},\nabla\cdot\bu_{*,h}^{n+1})_{n+1}
 =\frac{\tau^2}{\gamma_2}\left((\rho_h^{n+1})^{1/2}\nabla_{\rho,h}^{n+1}\bar\pi_h^{n+1},
 (\rho_h^{n+1})^{1/2}\nabla_{\rho,h}^{n+1}\delta\pi_h^{n+1}\right)_{n+1}.
\end{equation*}
Adding these two expansions and using \eqref{eq:pc-pressure-cross} to eliminate the cross pairing between $\bzeta_{H,h}^{n+1}$ and the pressure-increment gradient gives
\begin{align*}
 &\gamma_2(\bzeta_{*,h}^{n+1}-\bzeta_{H,h}^{n+1},\bzeta_{*,h}^{n+1})_{n+1}
 -\tau(\bar\pi_h^{n+1},\nabla\cdot\bu_{*,h}^{n+1})_{n+1}\\
 &\quad=\gamma_2(\bzeta_h^{n+1}-\bzeta_{H,h}^{n+1},\bzeta_h^{n+1})_{n+1}
 +\frac{\tau^2}{\gamma_2}(\rho_h^{n+1}\nabla_{\rho,h}^{n+1}\pi_{\rm tr,h}^{n+1},\nabla_{\rho,h}^{n+1}\delta\pi_h^{n+1})_{n+1}
 +\frac{\tau^2}{\gamma_2}\norm{(\rho_h^{n+1})^{1/2}\nabla_{\rho,h}^{n+1}\delta\pi_h^{n+1}}_{n+1}^2.
\end{align*}
Polarizing the last two pressure terms by \eqref{eq:pc-pressure-polarization}, then applying the kinetic BDF2 polarization \eqref{eq:kinetic-polarization-bdf2}, the Pythagorean decomposition \eqref{eq:pc-pythag}, and the pressure-history norm preservation gives the chained identity
\begin{equation}\label{eq:pc-momentum-b}
\begin{aligned}
 \tau(\bfv^{n+1},\bu_{*,h}^{n+1})_{n+1}
 ={}&\gamma_2(\bzeta_h^{n+1}-\bzeta_{H,h}^{n+1},\bzeta_h^{n+1})_{n+1}
 +\frac{2\tau}{\Rey}\norm{(\eta_h^{n+1})^{1/2}D(\bu_{*,h}^{n+1})}_{n+1}^2\\
 &+\frac{\tau q^{n+1}}{\We\Cn}(\phi_h^{n+1}\nabla\mu_h^{n+1},\bu_{*,h}^{n+1})_{n+1}
 +\frac{\tau^2}{2\gamma_2}\Bigl[
 \norm{(\rho_h^{n+1})^{1/2}\nabla_{\rho,h}^{n+1}\pi_h^{n+1}}_{n+1}^2\\
 &-\norm{(\rho_h^{n+1})^{1/2}\nabla_{\rho,h}^{n+1}\pi_{\rm tr,h}^{n+1}}_{n+1}^2
 +\norm{(\rho_h^{n+1})^{1/2}\nabla_{\rho,h}^{n+1}\delta\pi_h^{n+1}}_{n+1}^2\Bigr]\\
 ={}&\frac14\Bigl[\norm{\bg_{0,h}^{n+1}}_{n+1}^2+\norm{\bg_{1,h}^{n+1}}_{n+1}^2
 -\norm{\bg_{0,h}^{n}}_{n}^2-\norm{\bg_{1,h}^{n}}_{n}^2
 +\norm{\bzeta_h^{n+1}-\bg_{1,H}^{n+1}}_{n+1}^2\Bigr]\\
 &+\frac{\tau^2}{2\gamma_2}\Bigl[\norm{(\rho_h^{n+1})^{1/2}\nabla_{\rho,h}^{n+1}\pi_h^{n+1}}_{n+1}^2
 -\norm{(\rho_h^{n})^{1/2}\nabla_{\rho,h}^{n}\pi_h^{n}}_{n}^2\Bigr]
 +N_{p,h,2}^{n+1}\\
 &+\frac{2\tau}{\Rey}\norm{(\eta_h^{n+1})^{1/2}D(\bu_{*,h}^{n+1})}_{n+1}^2
 +\frac{\tau q^{n+1}}{\We\Cn}(\phi_h^{n+1}\nabla\mu_h^{n+1},\bu_{*,h}^{n+1})_{n+1}.
\end{aligned}
\end{equation}
\end{subequations}
The tentative-velocity capillary terms in \eqref{eq:pc-phase-balance} and \eqref{eq:pc-momentum-b} have equal magnitude and opposite signs, and therefore cancel exactly. Adding the two identities gives
\begin{align*}
 &E_{h,2}^{n+1,\rm pc}-E_{h,2}^{n,\rm pc}+N_{h,2}^{n+1}+N_{p,h,2}^{n+1}
 +\frac{2\tau}{\Rey}\norm{(\eta_h^{n+1})^{1/2}D(\bu_{*,h}^{n+1})}_{n+1}^2\\
 &\qquad+\frac{\tau q^{n+1}}{\Pe\We\Cn}\langle m_E^{n+1}\nabla\mu_h^{n+1},\nabla\mu_h^{n+1}\rangle_\times
 -\tau(\bfv^{n+1},\bu_{*,h}^{n+1})_{n+1}=0,
\end{align*}
which is \eqref{eq:pc-energy-law} for $k=2$. The case $k=1$ follows by the same argument, with every BDF2 polarization replaced by its BE counterpart.

The first step is computed by BE, using the prescribed $\pi_h^0$ to form the first algorithmic pressure-gradient history. For $\bfv=0$, Theorem~\ref{thm:solvability}, $m_E^{n+1}\ge0$, and the accepted-state bounds make every remainder and dissipation term in \eqref{eq:pc-energy-law} nonnegative, so the pressure-correction modified energy is nonincreasing at each BDF2 step. The BDF2 pressure-correction energy law applies from the first $k=2$ step onward. BE and BDF2 use their respective coefficient $1/(2\gamma_k)$ in the algorithmic pressure storage. The pressure-gradient term is an algorithmic component of the modified energy, and the orthogonal defect of a topology-changing pressure projection appears as the nonnegative remainder in \eqref{eq:pc-energy-b}.
\end{proof}

\noindent\textbf{BE--BDF2 start-up connection.} The first-order start-up storage controls the initial BDF2 storage. Indeed, the Hilbert-space inequality
\begin{equation*}
 \frac14\left(\norm{x^1}^2+\norm{2x^1-x^0}^2\right)
 \le\frac34\norm{x^1}^2+\frac12\norm{x^1-x^0}^2
\end{equation*}
applied componentwise to the kinetic, phase, SAV, and scalar storages, together with $q^1>0$, gives
\begin{equation*}
 E_{h,2}^1\le\frac32E_{h,1}^1+N_{h,1}^1,
 \qquad E_{h,2}^{1,\rm pc}\le\frac32E_{h,1}^{1,\rm pc}+N_{h,1}^1.
\end{equation*}
Consequently, when $\bfv=0$, the BE energy balances imply
\begin{equation*}
 E_{h,2}^1\le\frac32E_{h,1}^0,
 \qquad E_{h,2}^{1,\rm pc}\le\frac32E_{h,1}^{0,\rm pc}.
\end{equation*}
Combining these bounds with \eqref{eq:saddle-energy-law} and \eqref{eq:pc-energy-law} yields, for all subsequent BDF2 levels,
\begin{equation*}
 E_{h,2}^n\le\frac32E_{h,1}^0,
 \qquad E_{h,2}^{n,\rm pc}\le\frac32E_{h,1}^{0,\rm pc}.
\end{equation*}
Thus the BE start-up and the BDF2 energy sequence are controlled by the initial data without any additional time-step restriction beyond the individual BE and BDF2 balances.

\section{Computational implementation}
This section describes the mesh, cross-space integration, and linear algebra used to implement the discretizations of Sections 2--4. The analysis itself only requires the accepted-mesh and finite element admissibility conditions of Section 2.2 and is independent of the particular metric generator. For the physical-interface tests of Section 6, both mesh branches use anisotropic metrics constructed from the accepted phase field $\phi_h^n$: the fixed-topology branch employs a trace--logarithmic variational moving mesh, and the topology-changing branch uses MMG-type metric-based $h$-adaptation. In the manufactured temporal-convergence test, an exact-Hessian metric is used for the topology-changing branch to isolate temporal error. On nonmatching meshes, $\langle\cdot,\cdot\rangle_\times$ is evaluated by overlap/supermesh integration or an equivalent piecewise integration on the active history meshes. In Sections 2--4 these symbols denote the corresponding physical cross-mesh integrals; the implementation evaluates them with the finite element quadrature accuracy required by the chosen polynomial and coefficient representations. For BE the active meshes are $\Th^{n+1}$ and $\Th^n$; for BDF2 they are $\Th^{n+1}$, $\Th^n$, and $\Th^{n-1}$.

Linear cross-mesh pairings are assembled on the corresponding source--target intersections. For coefficients formed from the BDF2 extrapolation, all three active meshes are represented on a common local partition. For example,
\begin{equation*}
\int_\Omega F(2\phi_h^n-\phi_h^{n-1})z_h^{n+1}\,dx
=\sum_{K^{n+1}\in\Th^{n+1}}\sum_{K^n\in\Th^n}\sum_{\substack{K^{n-1}\in\Th^{n-1}\\ |K^{n+1}\cap K^n\cap K^{n-1}|>0}}
\int_{K^{n+1}\cap K^n\cap K^{n-1}}F(2\phi_h^n-\phi_h^{n-1})z_h^{n+1}\,dx.
\end{equation*}
This common-intersection assembly supplies the BDF2 extrapolated mobility, nonlinear free-energy coefficient, and all other nonlinear history-dependent coefficients used in the current equations.

\subsection{Phase-driven metric and the two mesh updates}
The analysis in Sections 2--4 applies to accepted meshes satisfying the shape and finite element stability assumptions of Section 2.2. A representative phase-driven choice used in the physical-interface tests is a gradient-based metric constructed from Huang's anisotropic metric-tensor framework \cite{Huang2005}. Given the accepted state $\phi_h^n$ on $\Th^n$, choose $\nu_M\in[0,1]$ and $r_M\ge1$, and set
\begin{subequations}\label{eq:metric-defs}
\begin{equation}\label{eq:metric-alpha}
 \alpha_M^n=\max\left\{\alpha_{\min},\left[\frac1{|\Omega|}\int_\Omega
 \norm{\nabla\phi_h^n}^{1/(d/r_M+1-\nu_M)}\,dx\right]^{2(d/r_M+1-\nu_M)}\right\}.
\end{equation}
If the number of elements is fixed as $N_h=\#\Th^n$, take
\begin{align}\label{eq:metric-field}
 \mathbf M_{f,h}^n(x)={}&\left[\frac{N_h}{\displaystyle\int_\Omega
 \left(1+\left(\norm{\nabla\phi_h^n}^2\right)/\left(\alpha_M^n\right)\right)^{1/[2(d/r_M+1-\nu_M)]}\,dx}\right]^{2/d}\notag\\
 &\times\left(1+\left(\norm{\nabla\phi_h^n}^2\right)/\left(\alpha_M^n\right)\right)^{\frac1d\left(\frac1{d/r_M+1-\nu_M}-1\right)}
 \left[\mathbf I+\left(\nabla\phi_h^n(\nabla\phi_h^n)^T\right)/\left(\alpha_M^n\right)\right].
\end{align}
For each $K\in\Th^n$, take the elementwise average and apply the prescribed smoothing and spectral truncation,
\begin{equation}\label{eq:metric-cell}
 \mathbf M_K^n=\frac1{|K|}\int_K \mathbf M_{f,h}^n(x)\,dx,
 \qquad 0<\lambda_M^-\mathbf I\preceq \mathbf M_K^n\preceq\lambda_M^+\mathbf I.
\end{equation}
\end{subequations}
The $N_h$ normalization in \eqref{eq:metric-field} gives the metric scale for $r$-movement with a fixed number of degrees of freedom. In $h$-adaptation, $\mathbf M_{f,h}^n$ supplies the anisotropic directions and relative scales, and its global scaling is normalized to the prescribed mesh complexity.

\noindent\textbf{Fixed-topology trace--logarithmic variational mesh update.} With the subscript ${\rm TL}$ denoting the trace--logarithmic functional, fix one metric epoch and define
\begin{equation*}
 p_{\rm TL}=\frac{d\gamma_{\rm TL}}2,\qquad \gamma_{\rm TL}>1,
 \qquad
 \theta_h^n=\left(\left(\displaystyle\sum_{K\in\Th^n}|K|\sqrt{\det \mathbf M_K^n}\right)/\left(|\widehat\Omega|\right)\right)^{-2/d}.
\end{equation*}
For a physical element $K$ and its computational counterpart, let
\begin{align}
 \mathbf E_K^n&=[\boldsymbol x_1^K-\boldsymbol x_0^K,\ldots,\boldsymbol x_d^K-\boldsymbol x_0^K],\qquad
 \mathbf E_{c,K}=[\boldsymbol\xi_1^K-\boldsymbol\xi_0^K,\ldots,\boldsymbol\xi_d^K-\boldsymbol\xi_0^K],\notag\\
 \mathbf P_K^n&=\mathbf E_{c,K}(\mathbf E_K^n)^{-1},\qquad
 \mathbf A_K^n=\mathbf P_K^n(\mathbf M_K^n)^{-1}(\mathbf P_K^n)^T.
\end{align}
Within this epoch, $\mathbf M_K^n$ and the physical element $K$ are fixed while the computational coordinates $\boldsymbol\xi$ evolve in artificial time. The discrete trace--logarithmic functional is
\begin{equation}
 I_h^n(\xi_h)=\sum_{K\in\Th^n}|K|\sqrt{\det \mathbf M_K^n}\left\{[\operatorname{tr}(\mathbf A_K^n)]^{p_{\rm TL}}
 -p_{\rm TL}d^{p_{\rm TL}-1}(\theta_h^n)^{p_{\rm TL}}\log\det \mathbf A_K^n\right\}.
\end{equation}
Its computational-coordinate residual uses
\begin{subequations}\label{eq:tl-residual}
\begin{equation}\label{eq:tl-Q}
 \mathbf Q_K=p_{\rm TL}[\operatorname{tr}(\mathbf A_K^n)]^{p_{\rm TL}-1}\mathbf A_K^n
 -p_{\rm TL}d^{p_{\rm TL}-1}(\theta_h^n)^{p_{\rm TL}}\mathbf I.
\end{equation}
Let $\mathbf R=[-1\ \mathbf I_d]^T\in\R^{(d+1)\times d}$. Applying the geometric discretization of \cite{WangHuangWei2026} in computational coordinates gives
\begin{equation}\label{eq:tl-residual-gradient}
 \frac{\partial I_h^n}{\partial\boldsymbol\xi_i}
 =\sum_{K\in S_i}2|K|\sqrt{\det \mathbf M_K^n}
 \left[\mathbf R(\mathbf E_{c,K})^{-1}\mathbf Q_K\right]_{\ell(i,K),:}.
\end{equation}
\end{subequations}
During a metric epoch, $\Th^n$, $\mathbf M_K^n$, and the global metric volume remain fixed; each residual evaluation updates $\mathbf A_K^n$ and $\mathbf Q_K$ through the current $\mathbf E_{c,K}$. Let $\mathbf M_i^n$ denote the nodal metric obtained from surrounding element metrics. With the dimensionless nodal scaling used here, the free computational nodes satisfy
\begin{equation}
 \frac{d\boldsymbol\xi_i}{d\varsigma}
 =-\frac1{\tau_{\rm mesh}}
 \left[(\theta_h^n)^{-1/2}(\det \mathbf M_i^n)^{1/d}\right]^{-d/4}
 \frac{\partial I_h^n}{\partial\boldsymbol\xi_i},\qquad i\in\mathcal N_{\rm free}.
\end{equation}
Artificial time is integrated with the backward differentiation formula strategy used in \cite{WangHuangWei2026}. Once the computational mesh has been obtained, the nodal correspondence with the current physical mesh defines a piecewise-affine inverse map $\boldsymbol\Psi_h^n$. The next physical mesh is recovered at fixed reference nodes by
\begin{equation}
 \boldsymbol\Psi_h^n(\boldsymbol\xi_i^*)=\boldsymbol x_i^n,\qquad \boldsymbol x_i^{n+1}=\boldsymbol\Psi_h^n(\widehat{\boldsymbol x}_i),\qquad
 \Th^{n+1}=\bX_h^{n+1}(\widehat{\mathcal T}_h).
\end{equation}
The candidate mesh enters the CHNS update after satisfying the boundary correspondence, positive orientation, and fixed-topology mesh assumptions of Sections 2.2 and 2.4.

\noindent\textbf{MMG topology-changing update.} The topology-changing branch constructs a new conforming simplicial mesh in the same symmetric positive-definite metric field. First rescale the metric uniformly to match the target mesh complexity,
\begin{equation*}
 \int_\Omega\sqrt{\det \mathbf M^n(x)}\,dx\simeq N_{\rm tar},
\end{equation*}
where $N_{\rm tar}$ denotes the target element complexity. For a new element $K$, let $\mathbf M_K^n$ denote the local metric average. The output is accepted when the quasi-$M$-uniform conditions
\begin{subequations}\label{eq:mmg-conditions}
\begin{equation}\label{eq:mmg-equidistribution}
 0<c_{\rm eq}\le
 \frac{N_{\rm tar}|K|\sqrt{\det \mathbf M_K^n}}{\displaystyle\int_\Omega\sqrt{\det \mathbf M^n(x)}\,dx}
 \le C_{\rm eq}<\infty,\qquad K\in\Th^{n+1},
\end{equation}
and
\begin{equation}\label{eq:mmg-alignment}
 1\le\frac{\operatorname{tr}((\mathbf E_K)^T\mathbf M_K^n\mathbf E_K)}{d\,\det((\mathbf E_K)^T\mathbf M_K^n\mathbf E_K)^{1/d}}
 \le C_{\rm ali},\qquad \det \mathbf E_K>0,
\end{equation}
\end{subequations}
hold with mesh-independent positive constants. Condition (5.7a) controls quasi-equidistribution of metric volume, while (5.7b) controls the deviation of each element from isotropy in the local metric. Boundary labels and prescribed internal geometric constraints are preserved. The MMG operations of point insertion/removal, connectivity changes, and vertex relocation are driven by the prescribed metric \cite{ArpaiaEtAl2022,DapognyDobrzynskiFrey2014,DobrzynskiFrey2008}. We accept the remeshed grid only when $\Th^{n+1}$ satisfies \eqref{eq:mmg-conditions} and the finite element admissibility conditions of Section 2.2; the overlap partition is then constructed as described at the beginning of this section and reused by all cross-mesh histories.

The phase BE/BDF2 history is represented by the current phase-energy Ritz map (2.11), the carrier history by (3.2), the kinetic BE/BDF2 histories by the norm-constrained nearest projection \eqref{eq:kinetic-transfer-constrained}--\eqref{eq:kinetic-transfer-normalized}, and the pressure-gradient history by (3.11)--(3.12). On refinement steps for which the broken target space contains the restrictions of the old piecewise polynomials, the local projection acts as the identity; coarsening and general nonnested remeshing use the same overlap-defined projection and history construction.

\subsection{Linear algebra and operator reuse}
All field subproblems in Section 3 are linear. Once the mesh at $t^{n+1}$ has been accepted, the implementation uses cross-mesh preparation, two affine-response condensations for $r^{n+1}$ and $q^{n+1}$, and, in the pressure-correction branch, repeated application of one density-weighted pressure operator. The following algebra shows which matrices can be reused.

\subsubsection{Cross-mesh preparation}
Let $\{\varphi_i\}$ be a basis of $S_h^{n+1}$ and set
\begin{equation}
 M_{ij}=(\varphi_j,\varphi_i)_{n+1},\qquad
 K_{ij}=(\nabla\varphi_j,\nabla\varphi_i)_{n+1},\qquad
 \mathbf A_\phi=\Cn^2\mathbf K+s\mathbf M.
\end{equation}
The same matrix $\mathbf A_\phi$ is the current phase-energy Ritz operator and the phase-energy block in the CH--SAV solve. The phase history and the topology-changing carrier correction are prepared from
\begin{equation}
 \mathbf A_\phi\boldsymbol\Phi_H=\boldsymbol b_H,\qquad \mathbf K\boldsymbol\Lambda=\boldsymbol b_\lambda,
\end{equation}
where
\begin{equation*}
 (\boldsymbol b_H)_i=\Cn^2\langle\nabla\phi_H^{n+1},\nabla\varphi_i\rangle_\times
 +s\langle\phi_H^{n+1},\varphi_i\rangle_\times,
 \qquad
 (\boldsymbol b_\lambda)_i=\langle\bu_{E,h}^{n+1},\nabla\varphi_i\rangle_\times,
\end{equation*}
and the second system is understood on the zero-mean subspace. On a fixed-topology ALE interval, the reference correction uses the fixed stiffness matrix $\widehat{\mathbf K}$ and is pushed forward by (3.15), so the reference factorization is reusable over the interval. At a topology-changing step, the overlap partition is assembled once and supplies the right-hand sides above, the kinetic-history projection data, and the pressure-history data. If the kinetic-history target in \eqref{eq:kinetic-transfer-constrained} is the broken polynomial space on the current mesh, its $L^2$ projection is elementwise and \eqref{eq:kinetic-transfer-normalized} completes the history by a global norm rescaling.

\subsubsection{Phase responses and SAV condensation}
Define
\begin{equation*}
 (K_m)_{ij}=\langle m_E^{n+1}\nabla\varphi_j,\nabla\varphi_i\rangle_\times,
\end{equation*}
\begin{equation*}
 (\boldsymbol f_\phi)_i=\frac{\gamma_k}{\tau}(\phi_{H,h}^{n+1},\varphi_i)_{n+1}
 +\langle\phi_E^{n+1}\bu_{c,h}^{n+1},\nabla\varphi_i\rangle_\times,
 \qquad
 (\boldsymbol b_G)_i=\frac1{U_E^{n+1}}\langle G'(\phi_E^{n+1}),\varphi_i\rangle_\times.
\end{equation*}
The two affine field responses are obtained simultaneously from the same block operator,
\begin{equation}
 \begin{bmatrix}
 \dfrac{\gamma_k}{\tau}\mathbf M&\dfrac1\Pe \mathbf K_m\\[2pt]
 -\mathbf A_\phi&\mathbf M
 \end{bmatrix}
 \begin{bmatrix}\boldsymbol\Phi_0&\boldsymbol\Phi_1\\ \boldsymbol\Psi_0&\boldsymbol\Psi_1\end{bmatrix}
 =\begin{bmatrix}\boldsymbol f_\phi&0\\0&\boldsymbol b_G\end{bmatrix},
 \qquad
 \boldsymbol\Phi=\boldsymbol\Phi_0+r^{n+1}\boldsymbol\Phi_1,\quad \boldsymbol\Psi=\boldsymbol\Psi_0+r^{n+1}\boldsymbol\Psi_1.
\end{equation}
Substitution into (3.4c) gives
\begin{equation}
 \left[1-\frac1{2U_E^{n+1}}\langle G'(\phi_E^{n+1}),\phi_{1,h}^{n+1}\rangle_\times\right]r^{n+1}
 =r_H^{n+1}+\frac1{2U_E^{n+1}}\langle G'(\phi_E^{n+1}),\phi_{0,h}^{n+1}-\phi_{H,h}^{n+1}\rangle_\times.
\end{equation}
The unit response satisfies
\begin{equation*}
 \frac{\gamma_k}{\tau}(\phi_{1,h}^{n+1},\mu_{1,h}^{n+1})_{n+1}
 +\frac1\Pe\langle m_E^{n+1}\nabla\mu_{1,h}^{n+1},\nabla\mu_{1,h}^{n+1}\rangle_\times=0,
\end{equation*}
\begin{equation*}
 (\mu_{1,h}^{n+1},\phi_{1,h}^{n+1})_{n+1}
 =\norm{\phi_{1,h}^{n+1}}_{\Cn,s}^2
 +\frac1{U_E^{n+1}}\langle G'(\phi_E^{n+1}),\phi_{1,h}^{n+1}\rangle_\times,
\end{equation*}
and therefore
\begin{equation*}
 \frac1{U_E^{n+1}}\langle G'(\phi_E^{n+1}),\phi_{1,h}^{n+1}\rangle_\times
 =-\norm{\phi_{1,h}^{n+1}}_{\Cn,s}^2
 -\frac{\tau}{\gamma_k\Pe}\langle m_E^{n+1}\nabla\mu_{1,h}^{n+1},\nabla\mu_{1,h}^{n+1}\rangle_\times\le0.
\end{equation*}
Thus the coefficient in brackets in (5.11) is at least one, and the scalar reconstruction is uniquely determined.

For an iterative solver, eliminating the phase variable through the existing $\mathbf A_\phi$ solve yields
\begin{equation}
 \left(\frac{\gamma_k}{\tau}\mathbf M\mathbf A_\phi^{-1}\mathbf M+\frac1\Pe \mathbf K_m\right)\boldsymbol\Psi
 =\boldsymbol f_\phi+\frac{\gamma_k}{\tau}r^{n+1}\mathbf M\mathbf A_\phi^{-1}\boldsymbol b_G.
\end{equation}
This operator is symmetric positive definite for $m_E^{n+1}\ge0$, and its action reuses the $\mathbf A_\phi$ solver from the phase-energy Ritz history.

\subsubsection{Momentum responses and determination of \texorpdfstring{$q$}{q}}
Let $\{\bv_i\}$ and $\{\psi_a\}$ be bases of $V_h^{n+1}$ and $P_{h,0}^{n+1}$, and define
\begin{align}
 (M_\rho)_{ij}&=(\rho_h^{n+1}\bv_j,\bv_i)_{n+1},\qquad
 (C_\rho)_{ij}=C_{\rho,h}^{n+1}(\bv_j,\bv_i),\notag\\
 (K_\eta)_{ij}&=(\eta_h^{n+1}D\bv_j,D\bv_i)_{n+1},\qquad
 B_{ai}=(\nabla\cdot\bv_i,\psi_a)_{n+1},\notag\\
 \mathbf H&=\frac{\gamma_k}{\tau}\mathbf M_\rho+\mathbf C_\rho+\frac2\Rey\mathbf K_\eta.
\end{align}
The history--forcing and unit-capillary loads are
\begin{equation*}
 (\boldsymbol f_u)_i=\frac{\gamma_k}{\tau}\left\langle\bzeta_{H,h}^{n+1},(\rho_h^{n+1})^{1/2}\bv_i\right\rangle_\times
 +(\bfv^{n+1},\bv_i)_{n+1},
 \qquad
 (\boldsymbol f_c)_i=\frac1{\We\Cn}(\phi_h^{n+1}\nabla\mu_h^{n+1},\bv_i)_{n+1}.
\end{equation*}
The saddle-point responses are obtained from the two-right-hand-side system
\begin{equation}
 \begin{bmatrix}\mathbf H&-\mathbf B^T\\ \mathbf B&0\end{bmatrix}
 \begin{bmatrix}\boldsymbol U_0&\boldsymbol U_1\\ \boldsymbol P_0&\boldsymbol P_1\end{bmatrix}
 =\begin{bmatrix}\boldsymbol f_u&-\boldsymbol f_c\\0&0\end{bmatrix}.
\end{equation}
For the pressure-correction branch, after $\bar P$ has been obtained, the tentative responses satisfy
\begin{equation}
 \mathbf H[\boldsymbol U_{*,0}\ \boldsymbol U_{*,1}]=[\boldsymbol f_u+\mathbf B^T\bar{\boldsymbol P}\quad -\boldsymbol f_c].
\end{equation}
The two responses determine $A_q$, $B_q$, $C_q$ through \eqref{eq:q-coefficients}, and the scalar formula reconstructs $q^{n+1}$. A numerically stable evaluation of the positive root is
\begin{equation}
 q^{n+1}=\begin{cases}
 \dfrac{2C_q}{B_q+\sqrt{B_q^2+4A_qC_q}},& B_q\ge0,\\[8pt]
 \dfrac{-B_q+\sqrt{B_q^2+4A_qC_q}}{2A_q},& B_q<0.
 \end{cases}
\end{equation}
Since $\mathbf C_\rho^T=-\mathbf C_\rho$, the quadratic form of $\mathbf H$ is carried by $\frac{\gamma_k}{\tau}\mathbf M_\rho+\frac2\Rey \mathbf K_\eta$; this symmetric coercive part supplies a natural core for velocity preconditioning.

\subsubsection{Pressure-correction operator}
Using the discrete-gradient relation (3.10), coefficient elimination gives the positive pressure operator
\begin{equation}
 \mathbf S_p=\mathbf B\mathbf M_\rho^{-1}\mathbf B^T.
\end{equation}
With $\{\psi_a\}$ the pressure basis used in (5.13), define
\begin{equation*}
 (\boldsymbol g_{\rm tr})_a=\left\langle\bxi_{p,h}^{n+1},(\rho_h^{n+1})^{1/2}\nabla_{\rho,h}^{n+1}\psi_a\right\rangle_\times,
\end{equation*}
\begin{equation*}
 (\boldsymbol g_{\rm bar})_a=\left\langle\bxi_{p,h}^{n+1}+\frac{\gamma_k}{\tau}\bzeta_{H,h}^{n+1},(\rho_h^{n+1})^{1/2}\nabla_{\rho,h}^{n+1}\psi_a\right\rangle_\times,
 \qquad
 (\boldsymbol d_*)_a=(\nabla\cdot\bu_{*,h}^{n+1},\psi_a)_{n+1}.
\end{equation*}
The two pressure-history representations and the pressure increment use the same operator,
\begin{equation}
 \mathbf S_p[\boldsymbol P_{\rm tr}\ \bar{\boldsymbol P}]=[\boldsymbol g_{\rm tr}\ \boldsymbol g_{\rm bar}],\qquad
 \mathbf S_p\,\delta\boldsymbol P=-\frac{\gamma_k}{\tau}\boldsymbol d_*.
\end{equation}
The first two right-hand sides are available before the tentative-velocity solve, and $\boldsymbol d_*$ is assembled from the resulting tentative velocity. The action $\boldsymbol x\mapsto \mathbf B\mathbf M_\rho^{-1}\mathbf B^T\boldsymbol x$ shares the weighted mass solver and the pressure preconditioner across all three applications.

\noindent\textbf{Operator reuse.} Within one time step, the CH--SAV block is used for the two $r$ responses, the saddle or tentative-velocity matrix for the two $q$ responses, and the pressure matrix for two history solves and one correction. A direct solver can reuse one factorization for these right-hand sides; an iterative solver can reuse one preconditioner. During fixed-topology ALE motion the sparsity pattern does not change, so the ordering and symbolic factorization can be kept while the coefficients are updated. After remeshing, the overlap data and algebraic graph are rebuilt.

\section{Numerical experiments}
\label{sec:numerical}

We first verify the temporal accuracy of the BDF2 discretizations and then examine the discrete structural
properties and benchmark behavior under the two mesh updates considered above. We denote the fixed-topology
$r$-adaptive moving-mesh method by RM and the topology-changing $h$-adaptive method by HM; SP and PC denote
the saddle-point and pressure-correction fluid solvers. The three tested combinations are
RM-SP, RM-PC, and HM-SP\@. The first physical time step is computed by backward Euler. Unless otherwise stated,
$p$ denotes the effective pressure defined in Section~2.1.

The computations are implemented with the Finite Element Analysis Library in Python (FEALPy) \cite{ZhengEtAl2026}. Its mesh data structures and finite element
function spaces are used throughout, and the fixed-topology moving meshes are generated with its
moving-mesh module.

RM-SP and RM-PC use the same fixed-connectivity, phase-driven trace--logarithmic variational moving-mesh
update and the same reference-corrected Piola carrier; only the velocity--pressure treatment differs. Their meshes are
pre-adapted from the initial phase field and then moved and checked at every physical time step. HM-SP is
likewise initialized from a pre-adapted mesh, with subsequent topology-changing adaptation controlled as
specified in each test. In the gravity-driven tests we take
\[
\bfv=\Fr^{-1}\rho(\phi)\bg,\qquad \bg=(0,-1).
\]

All conservation and compatibility diagnostics reported below are absolute, without normalization.
The global phase-mass drift is
\[
\delta_M^n=\bigl|(\phi_h^n,1)_n-(\phi_h^0,1)_0\bigr|.
\]
At a topology-changing step, the accuracy of the physical cross-mesh integration is monitored by
\[
\delta_{M,\times}^{n\to n+1}
=
\left|\left\langle\phi_h^n,1\right\rangle_\times-(\phi_h^n,1)_n\right|.
\]
This diagnostic measures the accuracy of the cross-mesh physical pairing; it does not introduce a separate
remapped phase field.

The accepted velocity and the compatible carrier satisfy different weak conditions in the theory and are
therefore monitored in their corresponding test spaces. Let $\boldsymbol d_{P,h}(\boldsymbol v_h)$ denote the
assembled divergence vector against the pressure basis and let $\boldsymbol d_{S,h}(\boldsymbol v)$ denote the
assembled carrier-compatibility vector against gradients of the phase basis. We report
\[
r_{\rm div}(\boldsymbol v_h)=\|\boldsymbol d_{P,h}(\boldsymbol v_h)\|_{w,P},
\qquad
r_{\rm car}(\boldsymbol v)=\|\boldsymbol d_{S,h}(\boldsymbol v)\|_{w,S}.
\]
Thus $r_{\rm div}$ measures the discrete incompressibility of the accepted velocity, whereas $r_{\rm car}$
measures the compatibility required by Step~2 of Section~3.1. For RM the latter is supplied by the
reference-space correction followed by the Piola pushforward; for HM it is supplied by the current-space correction.

For the energy tests we distinguish the physical AGG energy $E_{\rm AGG}$ of Theorem~2.1 from the discrete
modified storage. In the BDF2 regime the latter is $E_{h,2}^n$ for SP and $E_{h,2}^{n,\mathrm{pc}}$ for PC;
the BE first step uses the corresponding $k=1$ storage. Additive constants are omitted in plotted energies.

\subsection{Temporal convergence}
\label{sec:temporal-convergence}

The manufactured solution is prescribed on $\Omega=(0,1)^2$ for $0<t\le T=0.2$ by
\begin{align}
\phi_e(x,y,t)
&=\tanh\!\left(3\big[\cos(\pi x)\cos(\pi y)-0.35\sin(\pi t)\big]\right), \label{eq:ms-phi}\\
\psi_e(x,y,t)
&=0.1\sin^2(\pi x)\sin^2(\pi y)\sin(\pi t), \label{eq:ms-psi}\\
\bu_e&=(\partial_y\psi_e,-\partial_x\psi_e)^T, \label{eq:ms-u}\\
p_e &=0.1\cos(\pi x)\cos(\pi y)\cos(\pi t), \label{eq:ms-p}\\
\mu_e&=-\Cn^2\Delta\phi_e+\phi_e^3-\phi_e. \label{eq:ms-mu}
\end{align}
Thus $\nabla\!\cdot\bu_e=0$ and $\int_\Omega p_e\,dx=0$. We take
\(
\Cn=\Pe=0.14, \Rey=\We=1, m(\phi)=1,
\)
with $\rho_1=1.2$, $\rho_2=0.8$, $\eta_1=0.15$, and $\eta_2=0.05$, so that
\[
\rho(\phi)=1+0.2\phi,
\qquad
\eta(\phi)=0.1+0.05\phi.
\]
The phase source and momentum source are obtained by substituting $(\phi_e,\mu_e,\bu_e,p_e)$ into the corresponding equations of the extended AGG system on the consistent branch $q=1$; the chemical-potential equation requires no additional source. Since $\rho$ is affine in $\phi$, the manufactured phase source induces the corresponding affine-density source $\rho_\phi$ times the phase source; no separate density unknown is solved. The exact fields satisfy
\[
\bu_e=\boldsymbol0,
\qquad
\partial_n\phi_e=\partial_n\mu_e=0
\quad\text{on }\partial\Omega.
\]

We use $P_3$ elements for $\phi$ and $\mu$, $[P_3]^2$ for velocity, $P_2$ for pressure, and quadrature order $9$. For RM-SP and RM-PC, the base mesh parameter is $n_x=8$, and the phase-driven fixed-topology update uses four outer mesh steps with pseudo-time horizon $5\tau$. For HM-SP, $n_x=4$ with one initial adaptation pass, followed by common remeshing events at $t=0.04,0.08,0.12,0.16$ using the exact-Hessian metric with $h\in[0.08,0.3]$ and anisotropy bound $8$.

\begin{table}[!htbp]
\centering
\footnotesize
\setlength{\tabcolsep}{3.0pt}
\caption{Final-time temporal errors and observed orders for the three mesh/solver combinations.}
\label{tab:time-convergence}
\begin{tabular}{ccrrrrrrrr}
\toprule
& & \multicolumn{2}{c}{$\phi$} & \multicolumn{2}{c}{$\mu$} & \multicolumn{2}{c}{$\bu$} & \multicolumn{2}{c}{$p$}\\
\cmidrule(lr){3-4}\cmidrule(lr){5-6}\cmidrule(lr){7-8}\cmidrule(lr){9-10}
Method & $\tau$ & $E_\phi$ & ord. & $E_\mu$ & ord. & $E_{\bu}$ & ord. & $E_p$ & ord.\\
\midrule
RM-SP
& $1.00\!\times\!10^{-2}$ & $6.9721\!\times\!10^{-4}$ & --    & $1.1446\!\times\!10^{-3}$ & --    & $7.5545\!\times\!10^{-5}$ & --    & $3.1246\!\times\!10^{-3}$ & --    \\
& $5.00\!\times\!10^{-3}$ & $1.2979\!\times\!10^{-4}$ & 2.425 & $2.6709\!\times\!10^{-4}$ & 2.099 & $1.5582\!\times\!10^{-5}$ & 2.277 & $8.8044\!\times\!10^{-4}$ & 1.827 \\
& $2.50\!\times\!10^{-3}$ & $3.2100\!\times\!10^{-5}$ & 2.016 & $6.6982\!\times\!10^{-5}$ & 1.995 & $4.0176\!\times\!10^{-6}$ & 1.955 & $2.2079\!\times\!10^{-4}$ & 1.996 \\
& $1.25\!\times\!10^{-3}$ & $7.6739\!\times\!10^{-6}$ & 2.065 & $1.6221\!\times\!10^{-5}$ & 2.046 & $1.0345\!\times\!10^{-6}$ & 1.957 & $5.3034\!\times\!10^{-5}$ & 2.058 \\
\midrule
RM-PC
& $1.00\!\times\!10^{-2}$ & $6.9878\!\times\!10^{-4}$ & --    & $1.1454\!\times\!10^{-3}$ & --    & $8.2571\!\times\!10^{-5}$ & --    & $3.6623\!\times\!10^{-3}$ & --    \\
& $5.00\!\times\!10^{-3}$ & $1.2978\!\times\!10^{-4}$ & 2.429 & $2.6708\!\times\!10^{-4}$ & 2.101 & $1.8421\!\times\!10^{-5}$ & 2.164 & $8.9840\!\times\!10^{-4}$ & 2.027 \\
& $2.50\!\times\!10^{-3}$ & $3.2098\!\times\!10^{-5}$ & 2.016 & $6.6981\!\times\!10^{-5}$ & 1.995 & $4.6930\!\times\!10^{-6}$ & 1.973 & $2.2507\!\times\!10^{-4}$ & 1.997 \\
& $1.25\!\times\!10^{-3}$ & $7.6737\!\times\!10^{-6}$ & 2.065 & $1.6221\!\times\!10^{-5}$ & 2.046 & $1.1963\!\times\!10^{-6}$ & 1.972 & $5.4028\!\times\!10^{-5}$ & 2.059 \\
\midrule
HM-SP
& $1.00\!\times\!10^{-2}$ & $6.9757\!\times\!10^{-4}$ & --    & $1.1455\!\times\!10^{-3}$ & --    & $7.5176\!\times\!10^{-5}$ & --    & $3.1215\!\times\!10^{-3}$ & --    \\
& $5.00\!\times\!10^{-3}$ & $1.3009\!\times\!10^{-4}$ & 2.423 & $2.6727\!\times\!10^{-4}$ & 2.100 & $1.5265\!\times\!10^{-5}$ & 2.300 & $8.8183\!\times\!10^{-4}$ & 1.824 \\
& $2.50\!\times\!10^{-3}$ & $3.2422\!\times\!10^{-5}$ & 2.005 & $6.7442\!\times\!10^{-5}$ & 1.987 & $3.8893\!\times\!10^{-6}$ & 1.973 & $2.2310\!\times\!10^{-4}$ & 1.983 \\
& $1.25\!\times\!10^{-3}$ & $8.0139\!\times\!10^{-6}$ & 2.016 & $1.6764\!\times\!10^{-5}$ & 2.008 & $9.8200\!\times\!10^{-7}$ & 1.986 & $5.5539\!\times\!10^{-5}$ & 2.006 \\
\bottomrule
\end{tabular}
\end{table}

Table~\ref{tab:time-convergence} and Figure~\ref{fig:time-convergence}
show the expected asymptotic second-order temporal behavior for all four
primary variables. On the two finest refinement intervals, the observed
orders range from $1.955$ to $2.065$, while the coarsest level is not yet
fully in the asymptotic regime.

RM-SP and RM-PC produce nearly identical phase-field and
chemical-potential errors and comparable velocity and pressure errors.
HM-SP retains the same second-order behavior across the prescribed
topology-changing remeshing events. Thus, within this manufactured test, neither fixed-topology mesh motion nor the prescribed topology-changing remeshing events degrade the observed BDF2 temporal accuracy.

\begin{figure}[!htbp]
\centering
\includegraphics[width=0.33\textwidth]{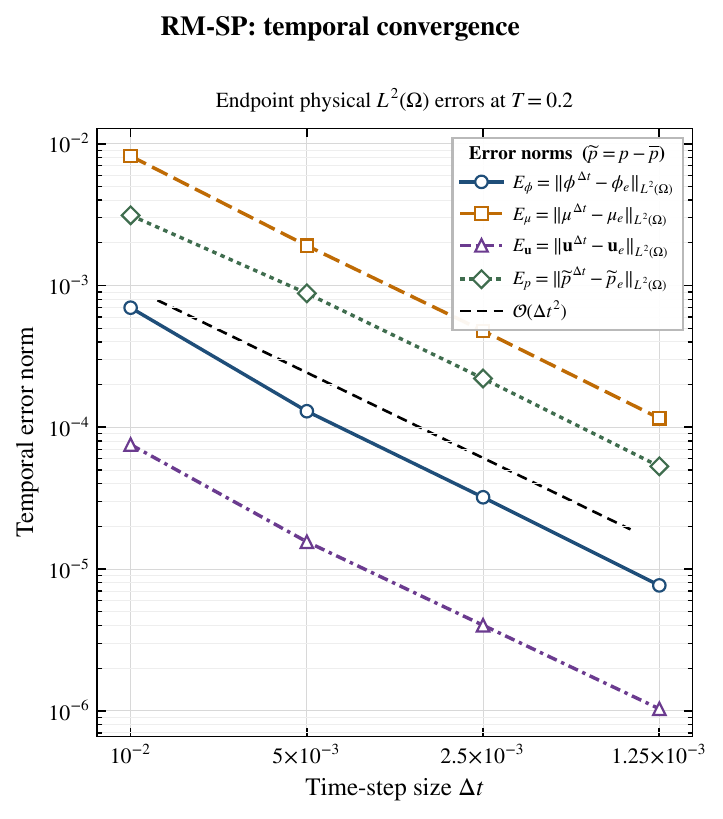}\hfill
\includegraphics[width=0.33\textwidth]{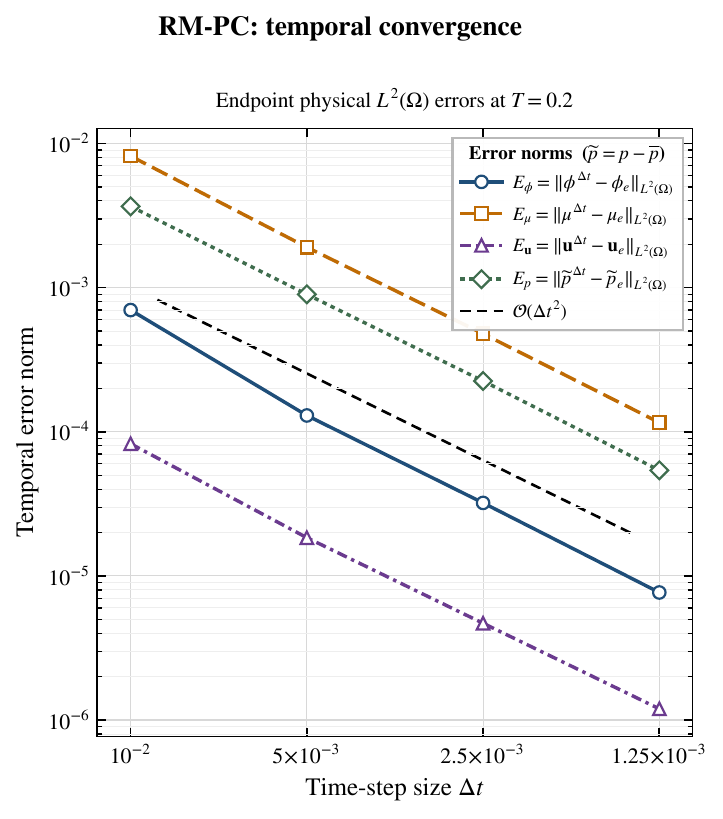}\hfill
\includegraphics[width=0.33\textwidth]{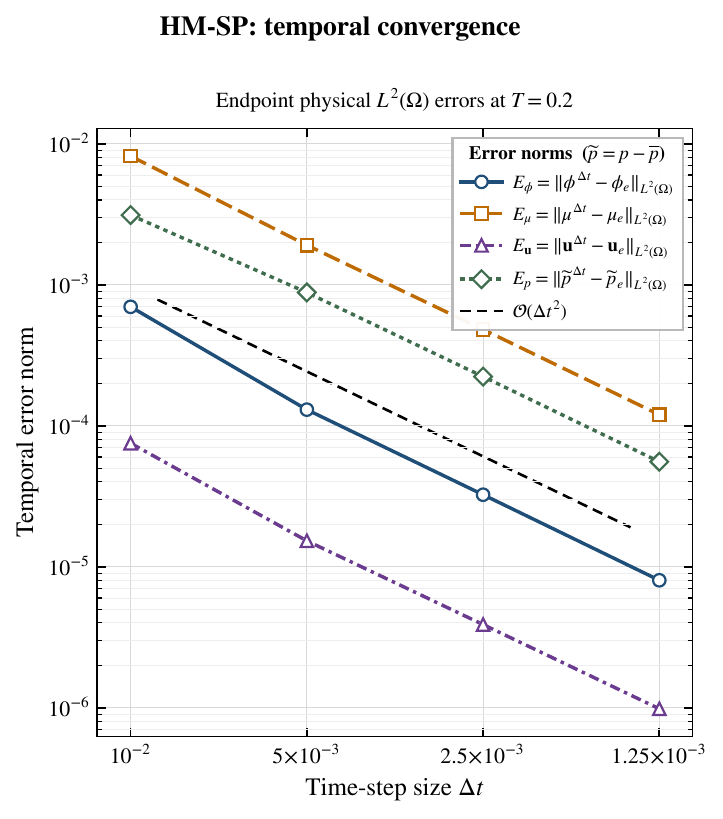}
\caption{Temporal convergence for the three methods: fixed-topology $r$-adaptive saddle-point (RM-SP, left), fixed-topology $r$-adaptive pressure-correction (RM-PC, center), and topology-changing $h$-adaptive saddle-point (HM-SP, right).}\label{fig:time-convergence}
\end{figure}

\subsection{Energy-stability tests}
\label{sec:energy-stability}

We next examine the discrete energy behavior in two unforced
matched-density tests: the relaxation of a six-lobed droplet
and a droplet-coalescence problem at $\Rey=1000$. In both tests,
\(
\rho_1=\rho_2=1,
\eta_1=\eta_2=1,
m(\phi)=1,
\)
with
\[
\bu=\boldsymbol0,\qquad
\partial_n\phi=\partial_n\mu=0
\quad\text{on }\partial\Omega.
\]
We use $P_2$ elements for $\phi$, $\mu$, and $\bu$, $P_1$ elements
for the pressure, and BDF2 time integration with a backward-Euler
first step. RM-SP and RM-PC start from the same $64\times64$ cross mesh
with $16384$ triangles and retain its connectivity. For HM-SP,
\(
h_{\min}=5\times10^{-4},
h_{\max}=0.25,
A_{\max}=60,
h_{\rm grad}=1.8,
\)
with adaptation checked every four physical time steps and at most
three passes per event.

\subsubsection{Six-lobed relaxation}\label{sec:six-lobed}

Using the six-lobed star-shaped initial geometry of Li et al.~\cite{LiLiuShenZheng2025},
the first test is performed on $\Omega=(0,1)^2$ with
\(
\Cn=0.005,
\Pe=10,
\Rey=\We=1,
\tau=10^{-3}\) and
\(T=1.5\).
The initial velocity is $\bu_0=\boldsymbol0$, and the phase field is
\[
\phi_0(x,y)
=
\tanh\!\left(
\frac{
0.25+0.1\cos\!\left(6\operatorname{atan2}(y-0.5,x-0.5)+\pi/2\right)-r
}{
\sqrt{2}\,\Cn
}
\right),
\qquad
r=\sqrt{(x-0.5)^2+(y-0.5)^2}.
\]

The left panel of Figure~\ref{fig:six-lobed-energy} shows closely matching physical-energy
trajectories for the three methods. The RM-SP and RM-PC physical energies are nonincreasing throughout the BDF2 regime. HM-SP undergoes $374$ topology-changing events; its physical energy has small changes at remeshing times, while the cross-mesh modified energy remains nonincreasing. The phase-mass drift is below $2.0\times10^{-15}$ in all three runs, the accepted-velocity weak-divergence residual below $5.5\times10^{-16}$, and the carrier-compatibility residual below $2.7\times10^{-15}$. For HM-SP, the cross-mesh phase-mass defect is below $5.6\times10^{-16}$ at topology-changing steps.

\begin{figure}[!htbp]
\centering
\includegraphics[width=0.49\linewidth]{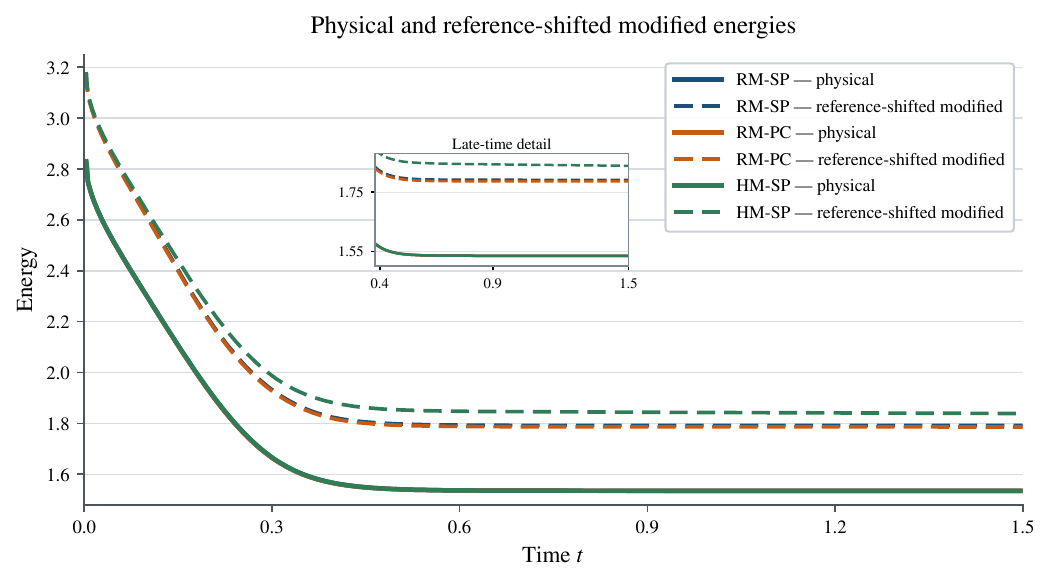}\hfill
\includegraphics[width=0.49\linewidth]{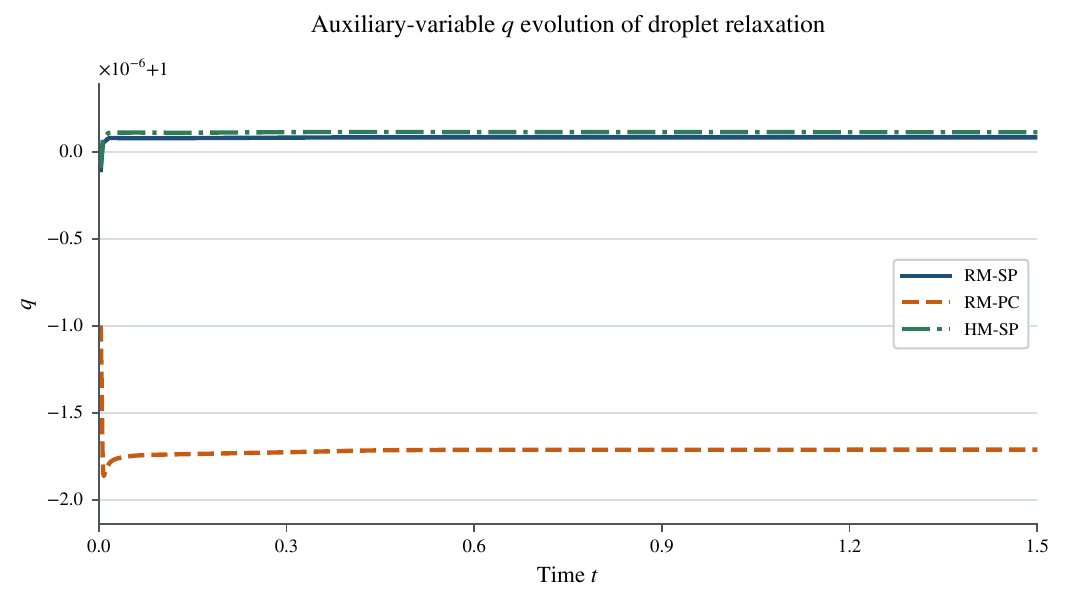}
\caption{
Energy and auxiliary-variable evolution of the six-lobed droplet relaxation for RM-SP,
RM-PC, and HM-SP\@. Left: solid curves denote the physical energy and
dashed curves the modified energy, with its additive constant omitted.
Right: the auxiliary variable $q$ remains close to $1$; the vertical axis
uses an offset of $1$ and a scale factor of $10^{-6}$.
}
\label{fig:six-lobed-energy}
\end{figure}

The right panel of Figure~\ref{fig:six-lobed-energy} shows that $q$ remains
close to $1$ throughout the relaxation. The largest visible departure is about
$1.9\times10^{-6}$ for RM-PC, with smaller departures for RM-SP and HM-SP,
including the run with repeated remeshing.

Figure~\ref{fig:six-lobed-evolution} shows the same relaxation under
the two mesh updates. The fixed-topology mesh follows the interface
through vertex redistribution, whereas HM-SP uses local refinement,
coarsening, and connectivity changes. Despite these distinct geometric
updates, both methods produce closely matching macroscopic relaxation toward an almost circular state.

\begin{figure}[!htbp]
\centering
\begin{tabular}{@{}r@{\hspace{2pt}}c@{}}
\raisebox{0.132\linewidth}{$t=0.01$} &
\includegraphics[width=0.80\linewidth,height=0.135\textheight,keepaspectratio]{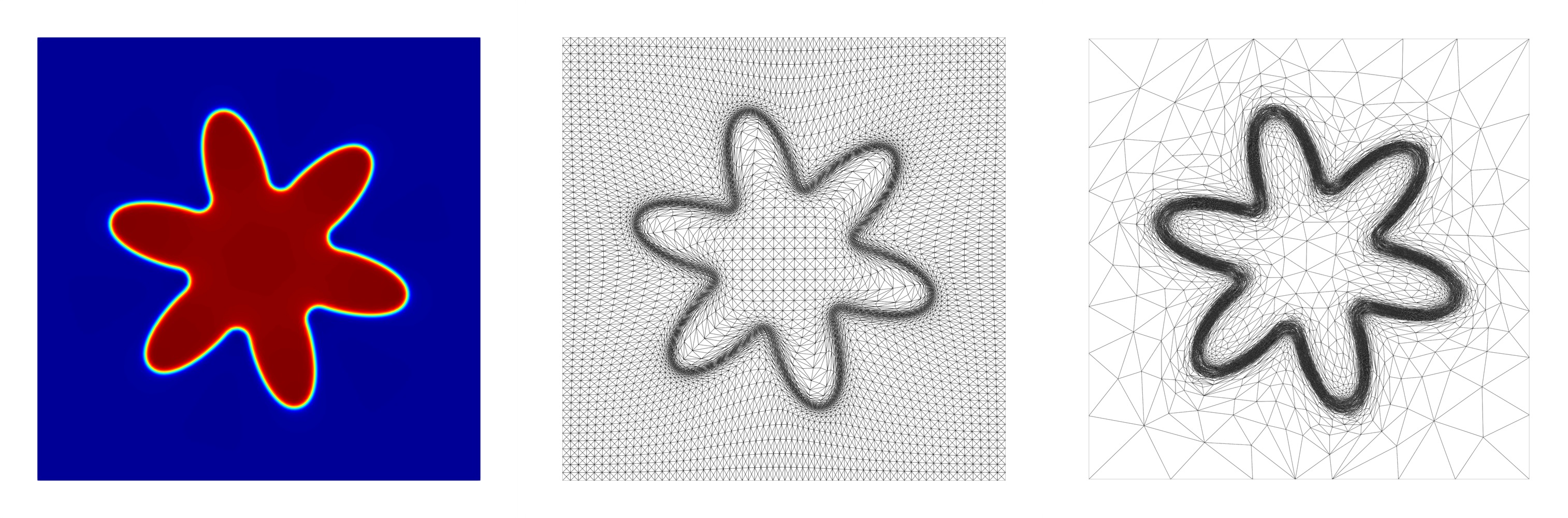} \\[-12pt]
\raisebox{0.132\linewidth}{$t=0.2$} &
\includegraphics[width=0.80\linewidth,height=0.135\textheight,keepaspectratio]{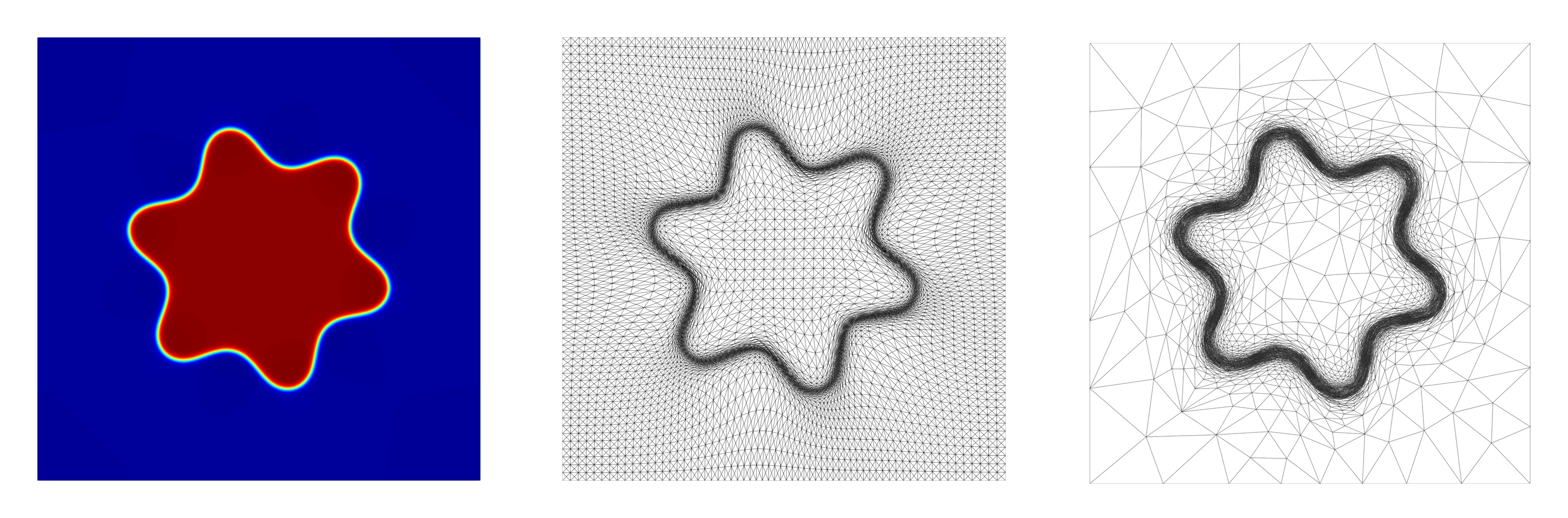} \\[-12pt]
\raisebox{0.132\linewidth}{$t=0.4$} &
\includegraphics[width=0.80\linewidth,height=0.135\textheight,keepaspectratio]{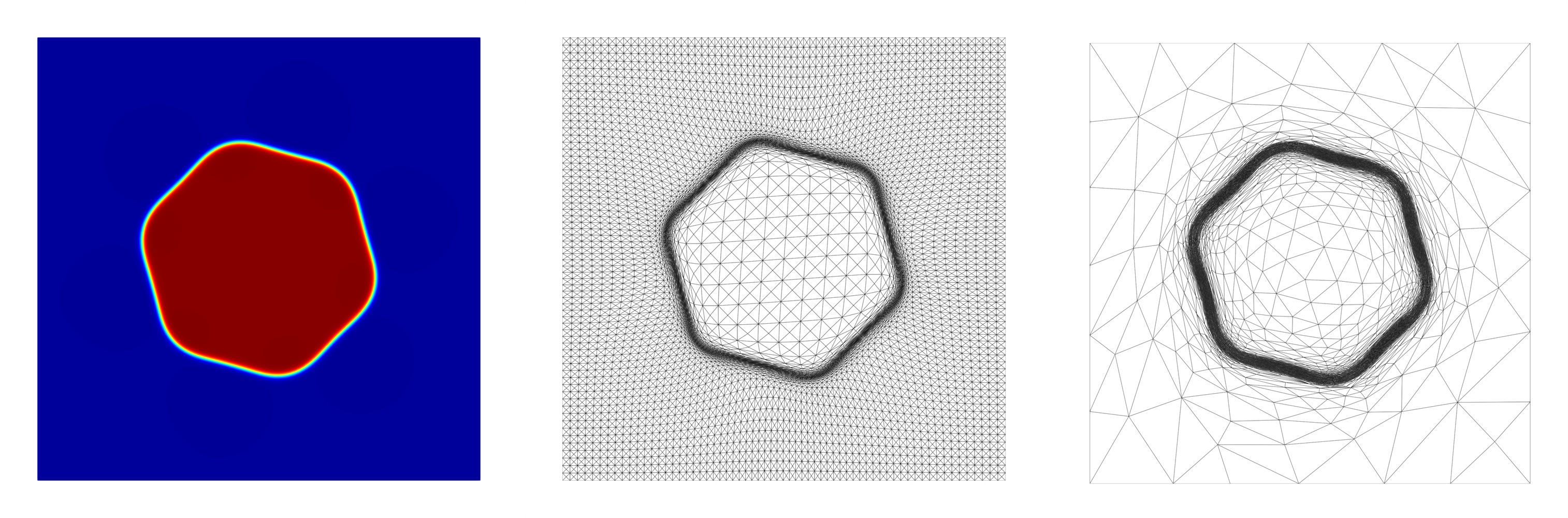} \\[-12pt]
\raisebox{0.132\linewidth}{$t=0.6$} &
\includegraphics[width=0.80\linewidth,height=0.135\textheight,keepaspectratio]{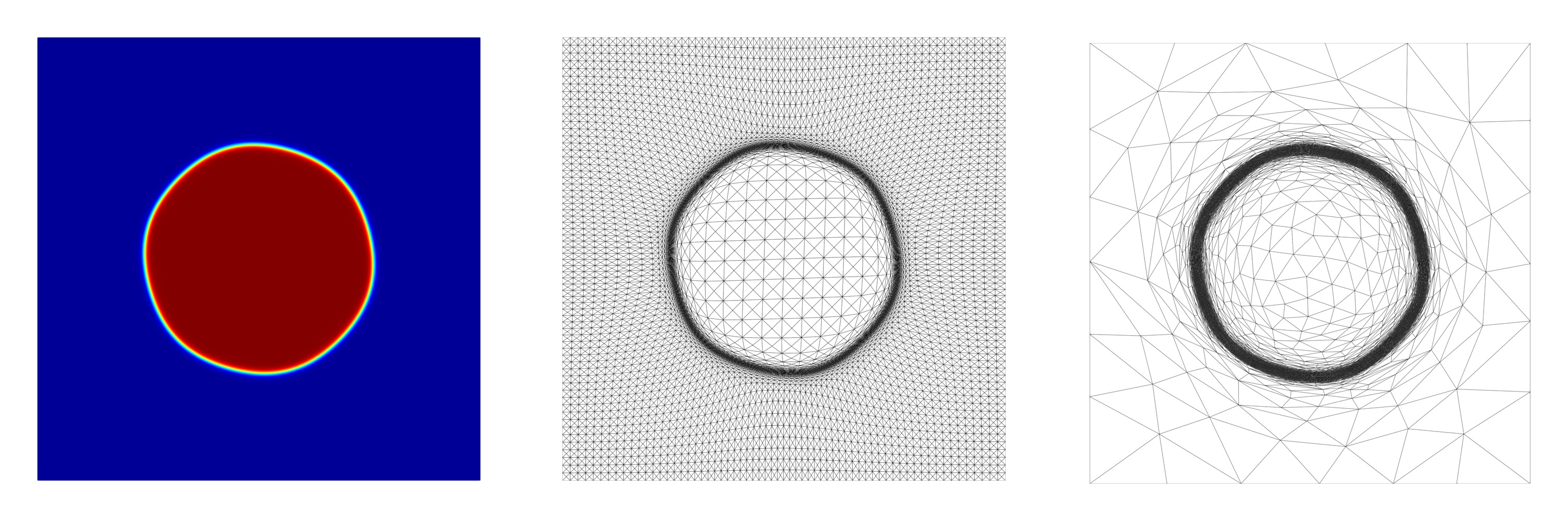} \\[-12pt]
\raisebox{0.132\linewidth}{$t=1.5$} &
\includegraphics[width=0.80\linewidth,height=0.135\textheight,keepaspectratio]{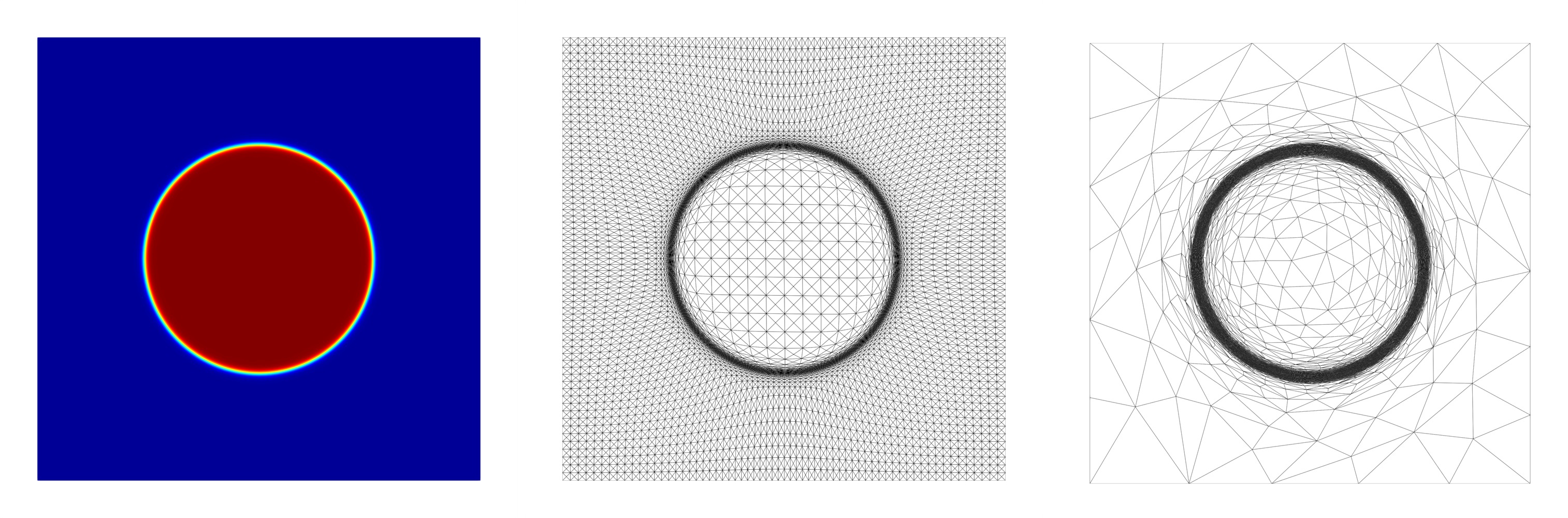}
\end{tabular}
\caption{
Evolution of the six-lobed droplet at
$t=0.01,\,0.2,\,0.4,\,0.6$ and $1.5$.
From left to right: RM-SP phase field, RM-SP fixed-topology moving mesh,
and HM-SP topology-changing $h$-adaptive mesh.
}
\label{fig:six-lobed-evolution}
\end{figure}

\subsubsection{Droplet coalescence}\label{sec:coalescence}

The second energy-stability scenario adapts the two-droplet merging configuration of Zhao and Han~\cite{ZhaoHan2021} and uses
\(
\Cn=0.005,
\Pe=\We=100,
\Rey=1000,
\tau=5\times10^{-3}\) and
\(T=10\).
Two droplets of radius $r_0=0.15$ are centered at
\[
\boldsymbol c_a=
\left(
0.5-\frac{r_0}{\sqrt2},
0.5+\frac{r_0}{\sqrt2}
\right),
\qquad
\boldsymbol c_b=
\left(
0.5+\frac{r_0}{\sqrt2},
0.5-\frac{r_0}{\sqrt2}
\right),
\]
and the initial phase field is
\[
\phi_0(x,y)
=
1-
\tanh\!\left(\frac{d_a-r_0}{\sqrt{2}\Cn}\right)
-
\tanh\!\left(\frac{d_b-r_0}{\sqrt{2}\Cn}\right),
\]
where
\(
d_a=|\boldsymbol x-\boldsymbol c_a|,
d_b=|\boldsymbol x-\boldsymbol c_b|,
\)
with $\bu_0=\boldsymbol0$.

Figure~\ref{fig:coalescence-evolution} shows the coalescence and subsequent
shape relaxation with the two mesh updates. The fixed-topology mesh
follows the evolving interface through vertex redistribution, whereas HM-SP
uses local refinement, coarsening, and connectivity changes. Despite these
distinct geometric updates, both meshes remain concentrated near the
diffuse interface, and the two methods produce closely matching macroscopic coalescence dynamics.

The left panel of Figure~\ref{fig:coalescence-energy} shows nearly coincident physical-energy
trajectories for RM-SP, RM-PC, and HM-SP; in this test the physical energies are observed to remain nonincreasing throughout the computation, while the modified storages follow the discrete dissipation law. Across the three runs, the phase-mass drift stays
below $3.0\times10^{-15}$ and the accepted-velocity weak-divergence
residual below $2.2\times10^{-14}$. For the carrier correction,
the carrier-compatibility residual remains below $3.4\times10^{-15}$. In particular,
HM-SP undergoes $499$ topology-changing events; over these remeshing steps,
the cross-mesh phase-mass defect remains below $6.7\times10^{-16}$ and the
carrier-compatibility residual below $1.9\times10^{-15}$. The corresponding phase-mass and carrier residuals therefore remain at numerical precision through repeated topology changes, while the modified energy decreases as in Section~4.

The right panel of Figure~\ref{fig:coalescence-energy} shows that $q$ also
stays very close to $1$ throughout coalescence and the subsequent relaxation.
The largest visible departure is approximately $4.2\times10^{-8}$ for RM-PC,
with still smaller deviations for RM-SP and HM-SP\@. The computed scalar therefore remains close to its continuous value even as the interface changes topology. This near-unity behavior is consistent with the continuous branch $q\equiv1$.

% Keep the two queued coalescence figures as top floats rather than allowing
% LaTeX to ship them as a float-only page; the RTI discussion can then use the
% remaining space below them.
\renewcommand{\floatpagefraction}{0.90}
\begin{figure}[!htbp]
\centering
\includegraphics[width=\textwidth]{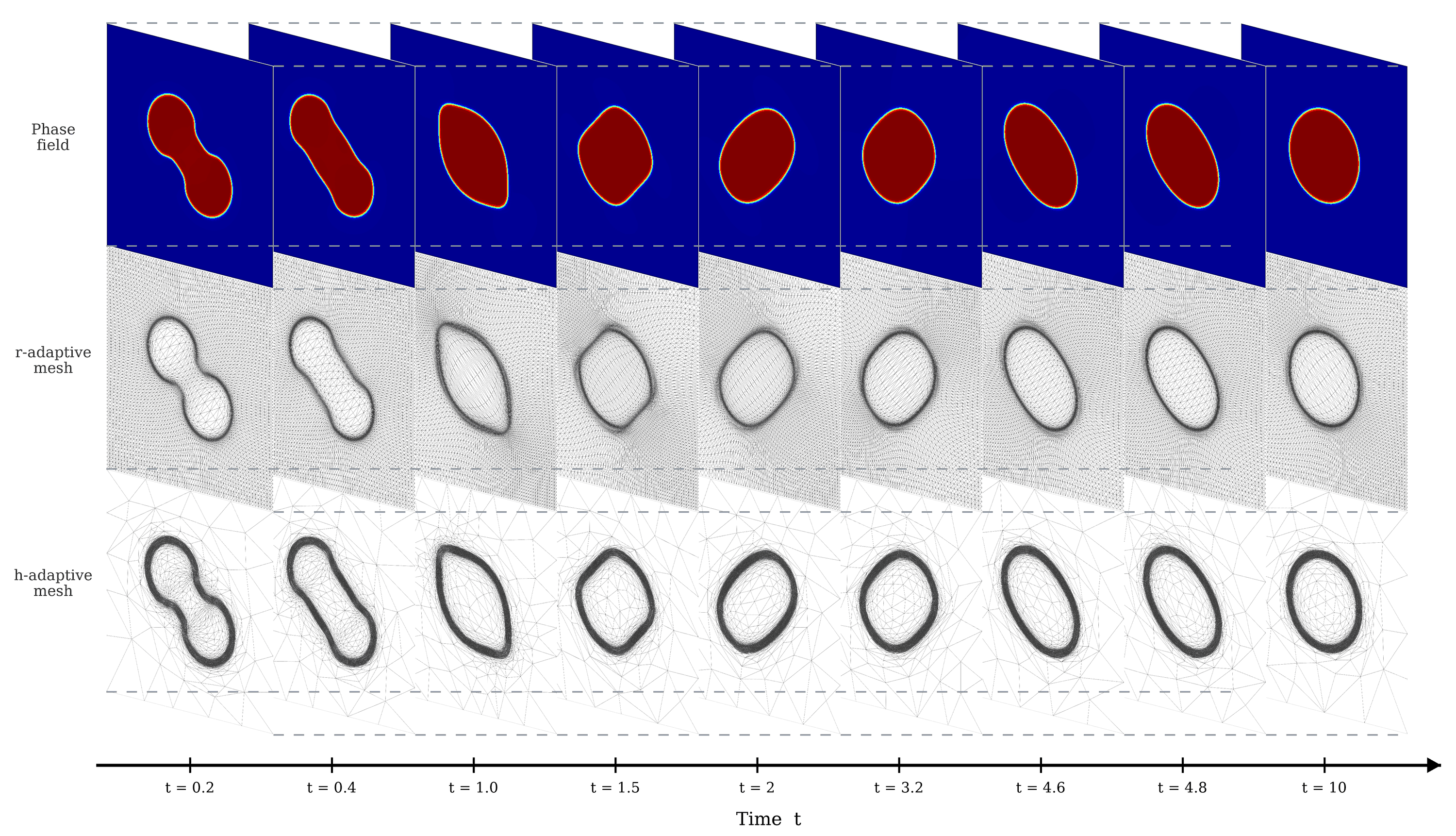}
\caption{Droplet-coalescence evolution at $t=0.2,\,0.4,\,1.0,\,1.5,\,2.0,\,3.2,\,4.6,\,4.8$, and $10$. Top: RM-SP
phase field. Middle: RM-SP fixed-topology $r$-adaptive mesh. Bottom: HM-SP topology-changing $h$-adaptive mesh.}
\label{fig:coalescence-evolution}
\end{figure}

\begin{figure}[!htbp]
\centering
\includegraphics[width=0.49\linewidth]{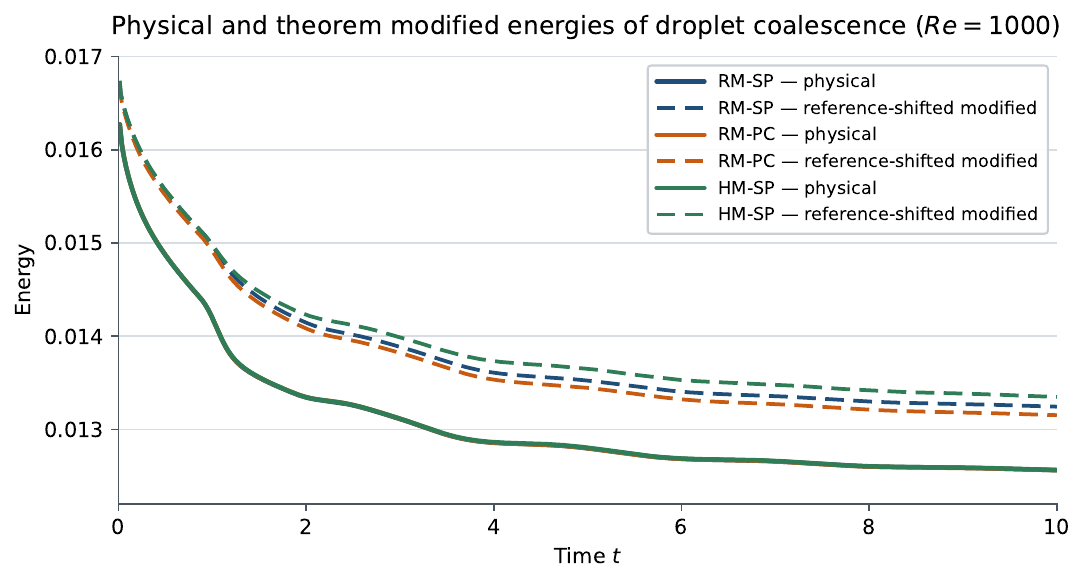}\hfill
\includegraphics[width=0.49\linewidth]{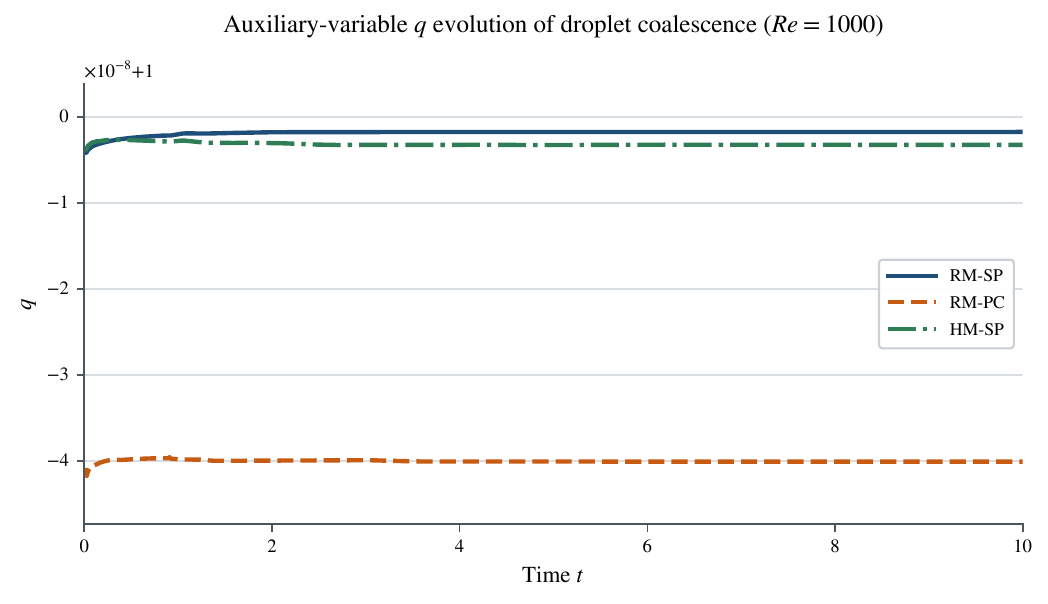}
\caption{Energy and auxiliary-variable evolution for the $\Rey=1000$ droplet-coalescence scenario
with RM-SP, RM-PC, and HM-SP\@. Left: solid curves denote the physical energy and dashed curves
the modified energy, with its additive constant omitted. Right: the auxiliary variable $q$
remains close to $1$; the vertical axis uses an offset of $1$ and a scale factor of $10^{-8}$.}
\label{fig:coalescence-energy}
\end{figure}

\subsection{Rayleigh--Taylor instability}
\label{sec:rti}

We next consider the two-dimensional single-mode Rayleigh--Taylor instability (RTI) on
\(
\Omega=(-0.5,0.5)\times(-1.5,1).
\)
The heavy fluid initially occupies the upper part of the domain and the
light fluid the lower part, with
\[
\phi_0(x,y)
=
\tanh\!\left(
\frac{y+0.1\cos(2\pi x)}
{\sqrt{2}\,\Cn}
\right),
\qquad
\bu_0=\boldsymbol0,\qquad p_0=0.
\]
No-slip conditions are imposed on the horizontal walls, free-slip
conditions on the vertical walls, and
$\partial_n\phi=\partial_n\mu=0$ on $\partial\Omega$.
We use the constant mobility $m(\phi)=1$, with
$\Cn=0.005$, $\Pe=100$, $\Fr=1$, and $\tau=10^{-3}$.
The spatial discretization uses $P_2$ elements for $\phi$, $\mu$, and
$\bu$, and $P_1$ elements for the pressure. The dimensionless parameter values for the two Atwood cases use the scaling adopted by Wang--Li--Wang~\cite{WangLiWang2024}:
\[
\begin{array}{c|cccc}
\mathrm{At} & \Rey & \We & T & \rho_H/\rho_L\\
\hline
0.50 & 3000 & 400\sqrt{2} & 2.5 & 3\\
0.75 & 7000 & 2800\sqrt{2}/3 & 2.0 & 7
\end{array}
\]
Since RM-SP and RM-PC have already exhibited nearly indistinguishable accuracy,
the RTI tests compare RM-SP and HM-SP to focus on the effect of the mesh update.

The RTI growth is quantified by the extreme positions of the reconstructed
$\phi_h=0$ interface,
\[
y_{\min}(t)=\min_{\{\phi_h=0\}}y,
\qquad
y_{\max}(t)=\max_{\{\phi_h=0\}}y,
\]
representing the descending heavy-fluid spike and rising light-fluid
bubble, respectively. We use the reference marker sets adopted by Wang--Li--Wang~\cite{WangLiWang2024}: for $\mathrm{At}=0.50$ they originate from Fu~\cite{Fu2020} and Li et al.~\cite{LiChengShenZhang2021}, while for $\mathrm{At}=0.75$ they originate from Li et al.~\cite{LiChengShenZhang2021}. These reference results were obtained with different formulations and nondimensionalizations and are used here only as trajectory references.

Figure~\ref{fig:rti-tip-trajectories} compares the computed spike and bubble
trajectories with the reference markers. For $\mathrm{At}=0.50$ ($\Rey=3000$), RM-SP and HM-SP are nearly
indistinguishable throughout the evolution. Their maximum differences
are only $1.05\times10^{-3}$ for the spike and
$2.61\times10^{-3}$ for the bubble, and the mean absolute errors
relative to the Li et al.~\cite{LiChengShenZhang2021} reference markers remain below $5.6\times10^{-3}$ and
$7.6\times10^{-3}$, respectively. The reference data of Fu~\cite{Fu2020} and Li et al.~\cite{LiChengShenZhang2021}
separate slightly at later times, while both adaptive methods follow
the same overall benchmark trajectory.

% Restore the document-wide threshold after the formerly float-only page has
% been assembled; the rising-bubble section applies its own setting below.
\renewcommand{\floatpagefraction}{0.78}

For $\mathrm{At}=0.75$ ($\Rey=7000$), the spike trajectories remain
closely superposed, with a maximum RM-SP/HM-SP difference of
$1.73\times10^{-3}$ and mean absolute errors of about
$1.1\times10^{-2}$ relative to the Li et al.~\cite{LiChengShenZhang2021} reference markers. A larger
difference appears only in the late-time bubble growth, where the maximum
inter-method separation reaches $3.58\times10^{-2}$; the corresponding
bubble errors are $5.82\times10^{-3}$ for RM-SP and
$1.71\times10^{-2}$ for HM-SP\@. Thus the descending-spike trajectory is only weakly affected by the choice between the two adaptive meshes over the tested interval, whereas the higher-Atwood bubble roll-up shows greater sensitivity to spatial adaptation. A small detached phase island is visible at late time after local separation from the rolled-up interface.

Figure~\ref{fig:rti-rmsp-snapshots} further shows the progressive
formation and roll-up of the descending spike, with the fixed-topology
mesh continuously redistributed along the increasingly deformed
interface.

\begin{figure}[!htbp]
\centering
\includegraphics[width=0.42\linewidth]{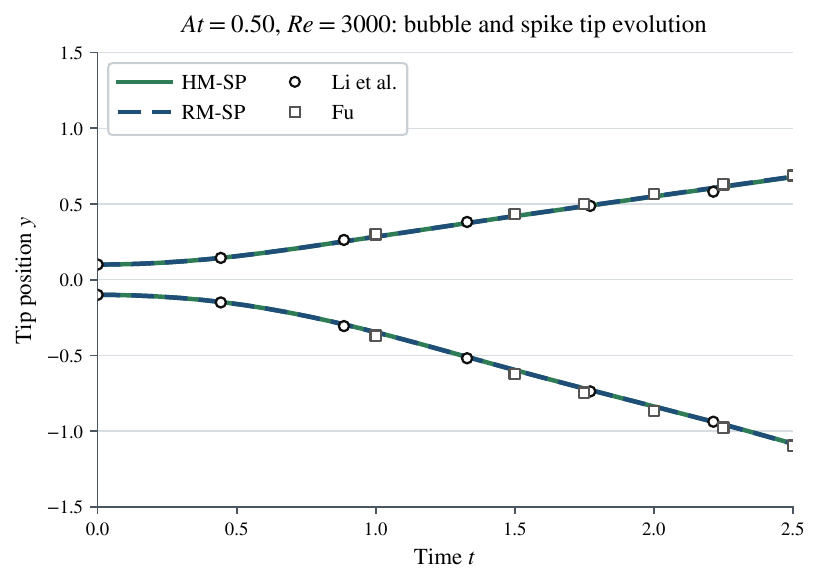}
\includegraphics[width=0.42\linewidth]{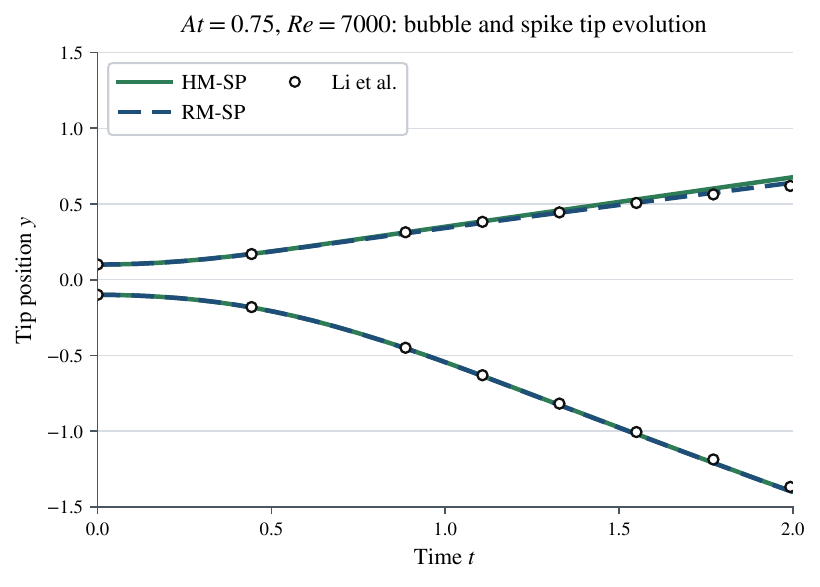}
\caption{Evolution of the RTI spike and bubble tips for
$\mathrm{At}=0.50$ (left) and $\mathrm{At}=0.75$ (right).
Solid and dashed curves denote HM-SP and RM-SP, respectively. For
$\mathrm{At}=0.50$, the markers show the reference data of Fu~\cite{Fu2020}
and Li et al.~\cite{LiChengShenZhang2021}; for $\mathrm{At}=0.75$, they show the data of
Li et al.~\cite{LiChengShenZhang2021}.}
\label{fig:rti-tip-trajectories}
\end{figure}

\begin{figure}[!htbp]
\centering
\begin{minipage}[b]{0.5\linewidth}
  \centering
  \includegraphics[width=\linewidth]{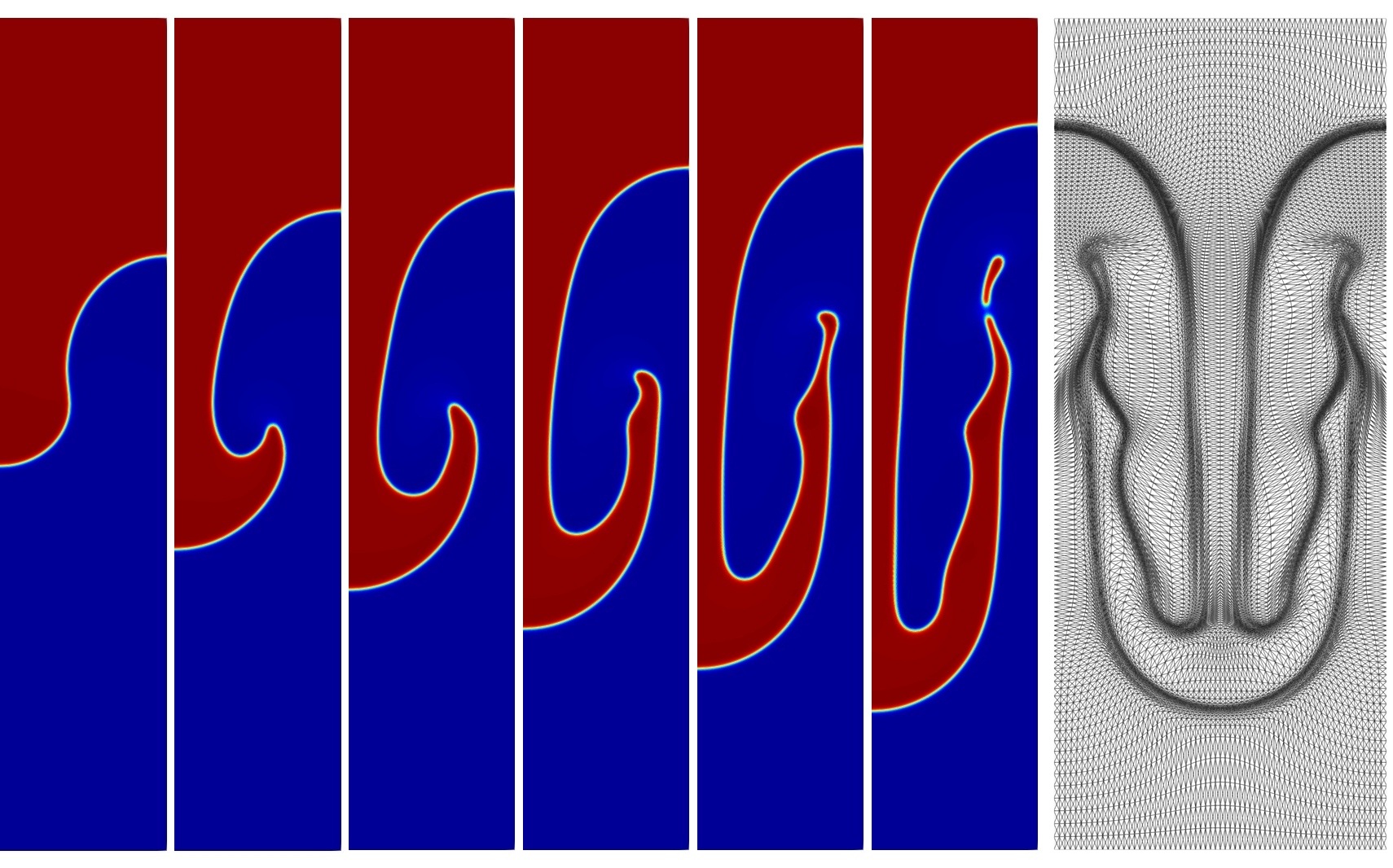}\\[-5pt]
  \small (a) $\mathrm{At}=0.50$ ($\Rey=3000$)
\end{minipage}\hfill
\begin{minipage}[b]{0.5\linewidth}
  \centering
  \includegraphics[width=\linewidth]{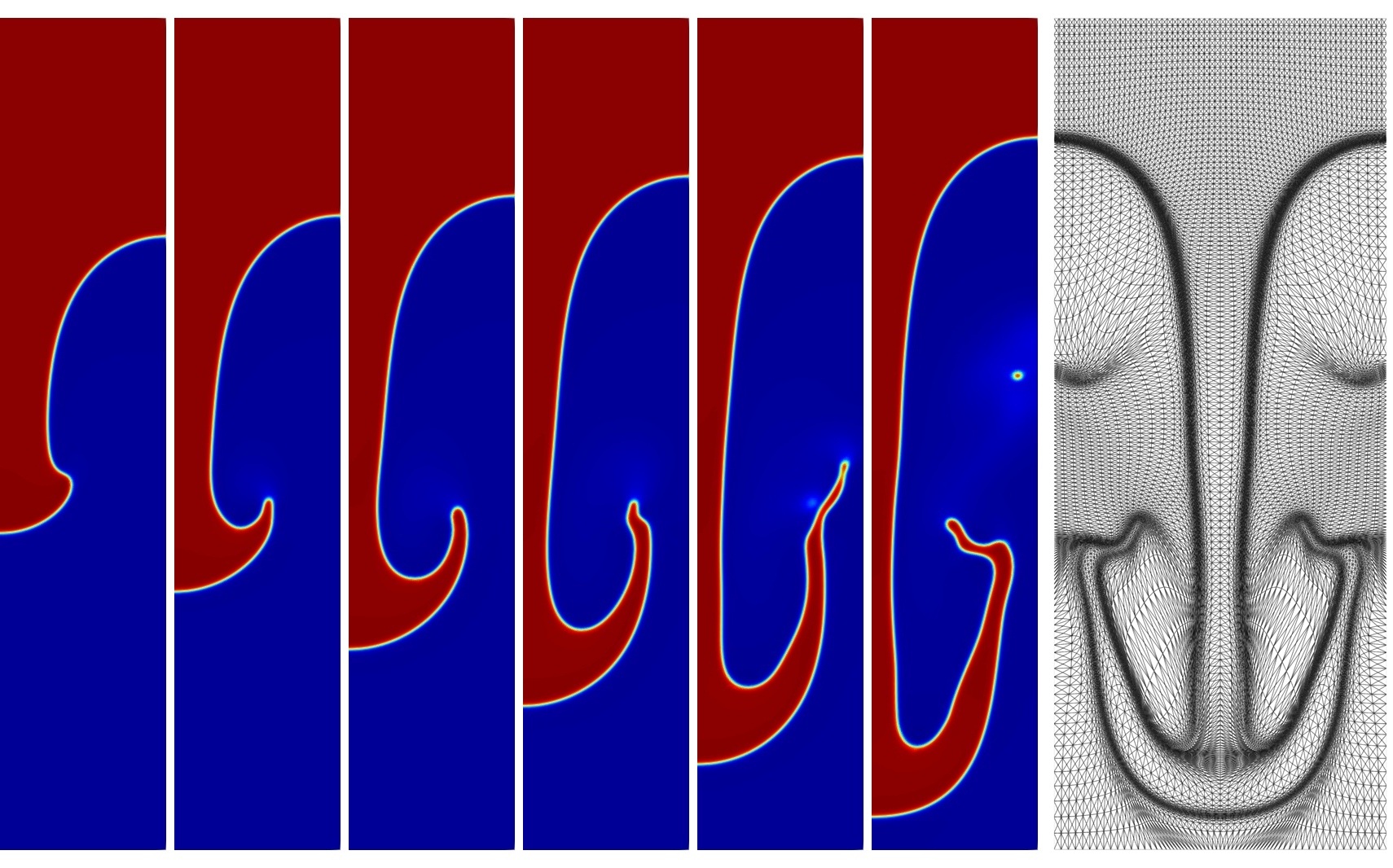}\\[-5pt]
  \small (b) $\mathrm{At}=0.75$ ($\Rey=7000$)
\end{minipage}
\caption{RM-SP snapshots for the Rayleigh--Taylor instability. In (a), the phase-field frames at $t=1.0,\,1.5,\,1.75,\,2.0,\,2.25$, and $2.5$ are followed by the RM-SP mesh at $t=2.5$. In (b), the phase-field frames at $t=1.0,\,1.2,\,1.4,\,1.6,\,1.8$, and $2.0$ are followed by the RM-SP mesh at $t=2.0$.}
\label{fig:rti-rmsp-snapshots}
\end{figure}

Across the four representative RTI runs, the phase-mass drift
remains below $3.2\times10^{-15}$ and the accepted-velocity
weak-divergence residual below $2.5\times10^{-15}$. The HM-SP runs
undergo $113$ and $565$ topology-changing adaptation events for
$\mathrm{At}=0.50$ and $0.75$, respectively; their cross-mesh
phase-mass defects remain below $7.8\times10^{-16}$, while the carrier-compatibility residual remains below $1.1\times10^{-13}$. Hence, in these runs, the large interfacial deformation and repeated topology changes leave the measured phase-conservation, accepted-velocity incompressibility, and carrier-compatibility residuals at numerical precision.

% Keep the rising-bubble figures compact and top-aligned when LaTeX creates a float-only page.
\setlength{\textfloatsep}{8pt plus 2pt minus 2pt}
\setlength{\floatsep}{5pt plus 2pt minus 2pt}
\setcounter{topnumber}{3}
\renewcommand{\topfraction}{0.96}
\renewcommand{\textfraction}{0.04}
% Keep the rising-bubble float queue compact so that the Case 2 figures
% can share a page instead of being flushed one-by-one at the final barrier.
\renewcommand{\floatpagefraction}{0.90}
\makeatletter
\setlength{\@fptop}{0pt}
\setlength{\@fpsep}{5pt}
\setlength{\@fpbot}{0pt plus 1fil}
\makeatother
\subsection{Rising-bubble}
\label{sec:rising-bubble}

We consider the two-dimensional rising-bubble benchmark of Hysing et al.~\cite{HysingEtAl2009} on
\(
\Omega=(0,1)\times(0,2), 0<t\le T=3.
\)
The initial bubble is a circle of radius $R_0=0.25$ centered at
$\boldsymbol x_0=(0.5,0.5)$. In the present diffuse-interface
formulation, the corresponding equilibrium phase profile is prescribed as
\[
\phi_0(x,y)
=
\tanh\!\left(
\frac{
\sqrt{(x-0.5)^2+(y-0.5)^2}-R_0
}{
\sqrt{2}\,\Cn
}
\right),
\]
so that $\phi<0$ identifies the lighter bubble phase, and
\(
\bu_0=\boldsymbol0, p_0=0.
\)
The phase field and chemical potential satisfy the no-flux conditions
\(\partial_n\phi=\partial_n\mu=0\) on \(\partial\Omega\).
Following the benchmark configuration, no-slip conditions are imposed
on the horizontal walls, \(\bu=\boldsymbol0\) on \(y=0,2\),
whereas the vertical walls are free slip,
\(\bu\cdot\boldsymbol n=0\) and
\(\bigl(2\eta D(\bu)\boldsymbol n\bigr)\cdot\boldsymbol t=0\)
on \(x=0,1\). The pressure is normalized by \(\int_\Omega p\,dx=0\),
and gravity uses the convention stated at the beginning of this section, with \(\Fr=1\).
The two benchmark cases differ only in their material ratios and dimensionless parameters.
For each tested interface thickness we set
\[
\Pe=\frac{\We}{\Rey\,\Cn},
\qquad
\tau=
\begin{cases}
10^{-3}, & \Cn\ge 0.005,\\
5\times10^{-4}, & \Cn<0.005,
\end{cases}
\]
for all three methods. The values of $\Rey$ and $\We$ follow the nondimensionalization of the
present AGG formulation and therefore differ numerically from the conventional dimensionless
parameters quoted for the original Hysing benchmark. The tested values are summarized in
Table~\ref{tab:rising-bubble-parameters}. RM-SP and RM-PC employ the phase-driven gradient
metric on fixed-topology moving meshes, whereas HM-SP uses the absolute-interface-gradient
metric with topology-changing adaptation.

\begin{table}[!htbp]
\centering
\footnotesize
\setlength{\tabcolsep}{4.5pt}
\caption{Dimensionless parameters used in the rising-bubble tests.}
\label{tab:rising-bubble-parameters}
\begin{tabular}{ccccccc}
\toprule
Case & $\rho_1/\rho_2$ & $\eta_1/\eta_2$ & $\Rey$ & $\We$ & $\Cn$ & $\Pe$ \\
\midrule
1 & 10 & 10 & 98.9949 & 37.7124
  & $0.010,\ 0.005,\ 0.0025$
  & $38.0952,\ 76.1905,\ 152.3810$ \\
2 & 1000 & 100 & 98.9949 & 471.4045
  & $0.010,\ 0.005,\ 0.0025,\ 0.00125$
  & $476.1905,\ 952.3810,\ 1904.7619,\ 3809.5238$ \\
\bottomrule
\end{tabular}
\end{table}

The Hysing benchmark is a sharp-interface problem; $\Cn$ and $\Pe$
enter only through the present diffuse-interface approximation. All
computations use $P_2$ elements for $\phi$, $\mu$, and $\bu$, $P_1$
elements for the pressure, and BDF2 time integration with a
backward-Euler first step. RM-SP and RM-PC use identical
fixed-topology meshes at each common resolution: $32\times64$ for
$\Cn=0.01$ and $64\times128$ for $\Cn=0.005$; the Case~2 runs at
$\Cn=0.0025$ also use $64\times128$. HM-SP employs topology-changing interface-driven
adaptation, with the case-specific controls given below.

For both benchmark cases, we use the MooNMD reference data reported by Hysing et al.~\cite{HysingEtAl2009}. Following the benchmark quantities of Hysing et al. and their phase-field implementation in Wang--Li--Wang~\cite{WangLiWang2024}, the bubble dynamics are quantified
by the centroid height, circularity, and mean rise velocity, defined respectively by
\[
y_c(t)
=\frac{\displaystyle\int_{\{\phi_h<0\}} y\,dx}
     {\displaystyle\int_{\{\phi_h<0\}} dx},
\qquad
\mathcal C(t)
=\frac{2\sqrt{\displaystyle\pi\int_{\{\phi_h<0\}}dx}}
     {\displaystyle\int_{\{\phi_h=0\}}ds},
\qquad
V_c(t)
=\frac{\displaystyle\int_{\{\phi_h<0\}} u_{h,2}\,dx}
     {\displaystyle\int_{\{\phi_h<0\}}dx}.
\]
Here $\phi_h=0$ represents the reconstructed bubble interface. For the
quantitative comparison, we report the final centroid position
$y_c(3)$, the minimum circularity
\(
\mathcal C_{\min}=\min_{0<t\le3}\mathcal C(t),
\)
and the maximum rise velocity
\(
V_{\max}=\max_{0<t\le3}V_c(t),
\)
together with the corresponding occurrence times. The complete time
histories of $y_c$, $\mathcal C$, and $V_c$ and the reconstructed
$\phi_h=0$ contour at $t=3$ are compared with the MooNMD reference data.

\subsubsection{Case 1}
\label{sec:rising-bubble-case1}

Case~1 has the moderate density and viscosity ratios
$\rho_1/\rho_2=\eta_1/\eta_2=10$. For HM-SP,
$h\in[6\times10^{-4},0.2]$, with anisotropy bound $40$ and
$h_{\rm grad}=1.8$; adaptation is checked every four time steps with
at most two passes per event.

\begin{figure}[!htbp]
\centering
\includegraphics[width=0.75\linewidth]{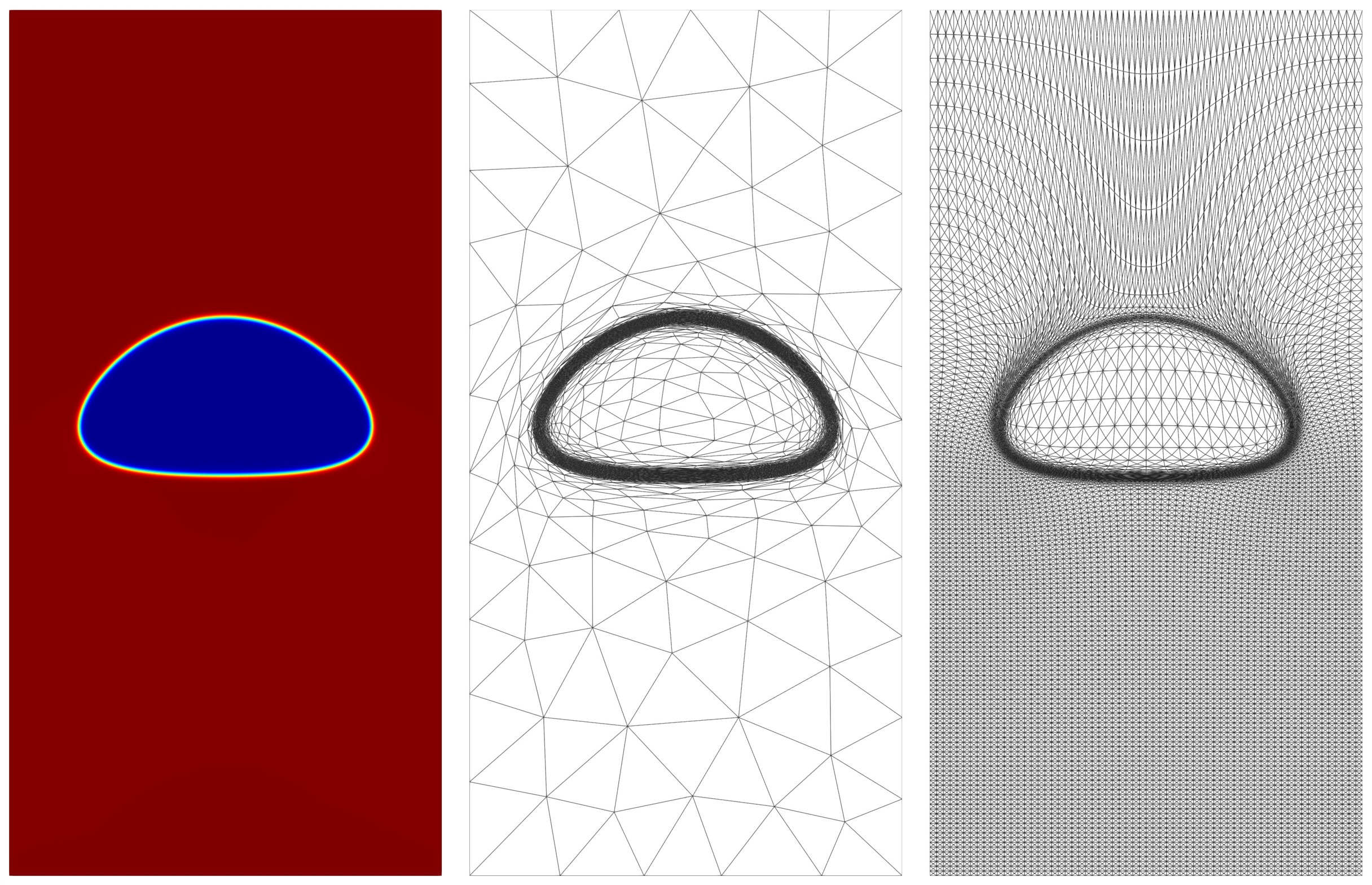}
\caption{Case~1 at $t=3$: HM-SP phase field (left), HM-SP topology-changing $h$-adaptive mesh (center),
and RM-SP fixed-topology $r$-adaptive mesh (right), all at $\Cn=0.005$.}
\label{fig:rising-bubble-case1-state}
\end{figure}

\begin{figure}[!htbp]
\centering
\setlength{\tabcolsep}{-2pt}
\begin{tabular}{@{}ccc@{}}
\includegraphics[width=0.33\linewidth]{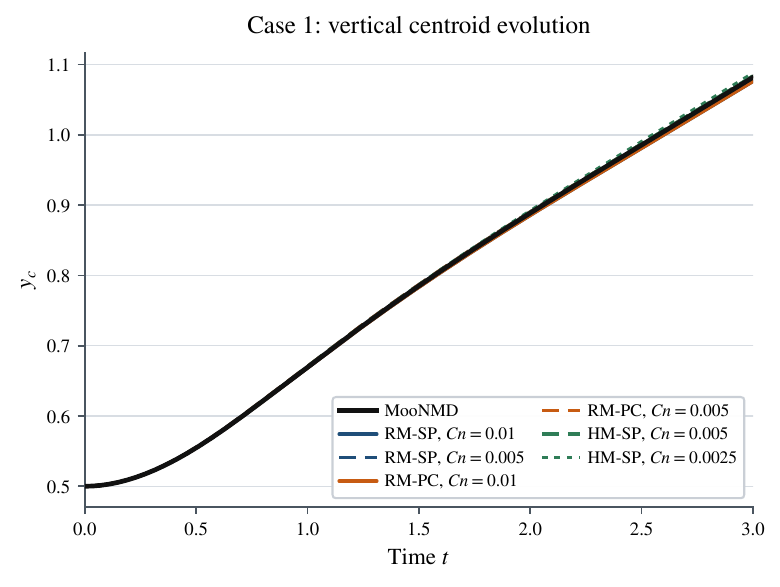} &
\includegraphics[width=0.33\linewidth]{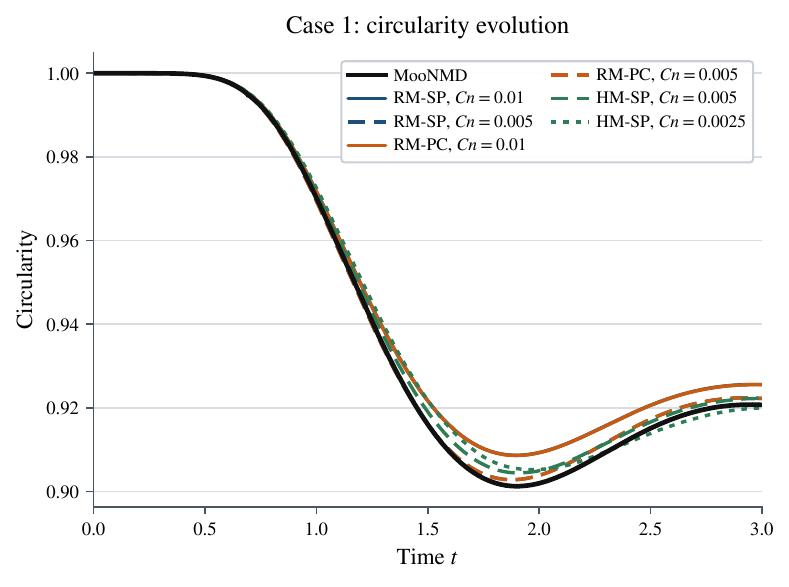} &
\includegraphics[width=0.33\linewidth]{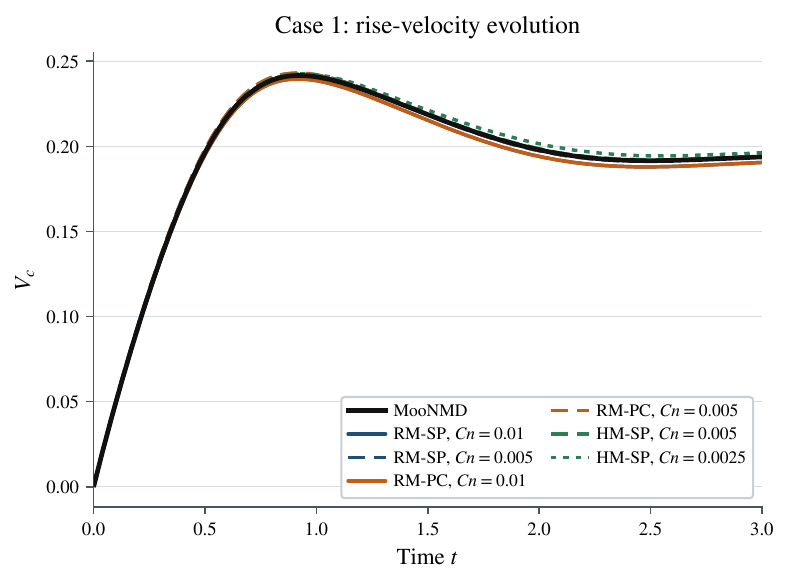}
\end{tabular}
\caption{Case~1 time histories of the bubble centroid height (left),
circularity (center), and mean rise velocity (right). Black curves
denote the MooNMD data of Hysing et al. The RM-SP and RM-PC histories are
nearly superposed.}
\label{fig:rising-bubble-case1-evolution}
\end{figure}

At $\Cn=0.005$, all three methods closely reproduce the reference
dynamics. The RM-SP and RM-PC histories are nearly superposed, as are
their terminal contours; their terminal-centroid and minimum-circularity
deviations are about $0.14\%$ and $0.17\%$, respectively. HM-SP gives
comparable centroid and circularity predictions and the closest maximum
rise velocity, with a deviation of $0.31\%$. The corresponding benchmark
quantities are listed in Table~\ref{tab:bubble-case1}. Figures~\ref{fig:rising-bubble-case1-state}
and~\ref{fig:rising-bubble-case1-evolution} therefore show essentially the
same bubble evolution under fixed-topology vertex motion and
topology-changing adaptation.

\begin{table}[!htbp]
\centering
\footnotesize
\setlength{\tabcolsep}{6.0pt}
\caption{Benchmark quantities for rising-bubble Case 1.}
\label{tab:bubble-case1}
\begin{tabular}{lccccccc}
\toprule
Method & $\Cn$
& $y_c(3)$
& $\mathcal C_{\min}$
& $t_{\mathcal C}$
& $V_{\max}$
& $t_V$ \\
\midrule
MooNMD& --& 1.081699 & 0.901252 & 1.900 & 0.241658 & 0.924\\
RM-SP& 0.01000& 1.076076 & 0.908642 & 1.897& 0.239589 & 0.917\\
RM-SP& 0.00500& 1.083260 & 0.902806 & 1.878& 0.243008 & 0.915\\
RM-PC& 0.01000& 1.076028 & 0.908615 & 1.897& 0.239577 & 0.917\\
RM-PC& 0.00500& 1.083251 & 0.902806 & 1.878& 0.243000 & 0.915\\
HM-SP& 0.00500& 1.083603 & 0.904439 & 1.904& 0.242414 & 0.932\\
HM-SP& 0.00250& 1.087693 & 0.905213 & 1.965& 0.242666 & 0.940\\
\bottomrule
\end{tabular}
\end{table}

\subsubsection{Case 2}
\label{sec:rising-bubble-case2}

Case~2 increases the density and viscosity ratios to
$\rho_1/\rho_2=1000$ and $\eta_1/\eta_2=100$, producing substantially
stronger deformation. For HM-SP, $h\in[4\times10^{-4},0.25]$, with
anisotropy bound $60$ and $h_{\rm grad}=1.8$; adaptation is checked every
five time steps with at most three passes per event.

\begin{figure}[!t]
\centering
\includegraphics[width=0.76\linewidth]{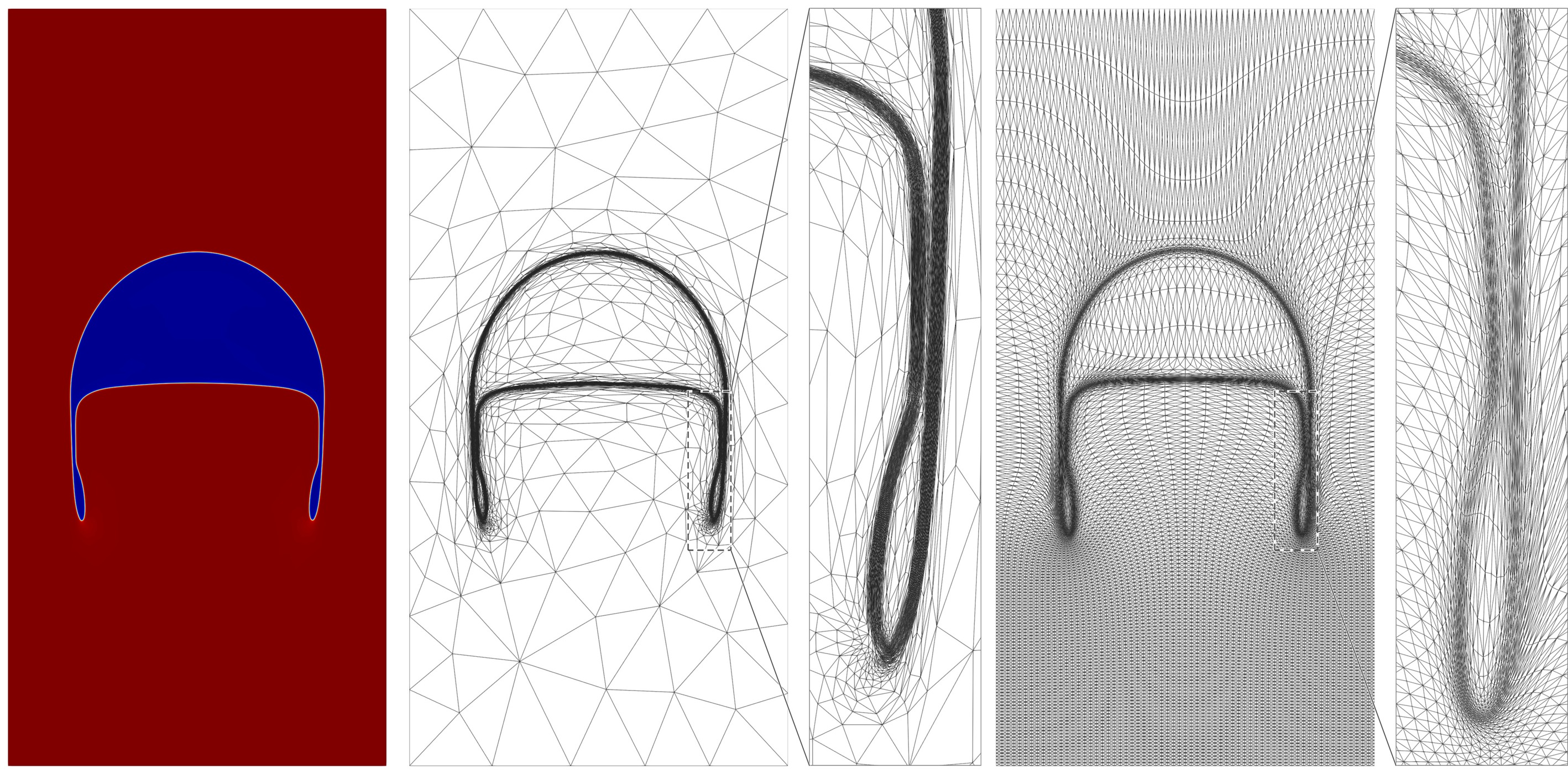}
\caption{Case~2 at $t=3$: HM-SP phase field (left), HM-SP topology-changing $h$-adaptive mesh (middle), and RM-SP
fixed-topology $r$-adaptive mesh (right). The mesh panels include their right-lower interface enlargements.}
\label{fig:rising-bubble-case2-state}
\end{figure}

\begin{figure}[!t]
\centering
\setlength{\tabcolsep}{-2pt}
\begin{tabular}{@{}ccc@{}}
\includegraphics[width=0.31\linewidth]{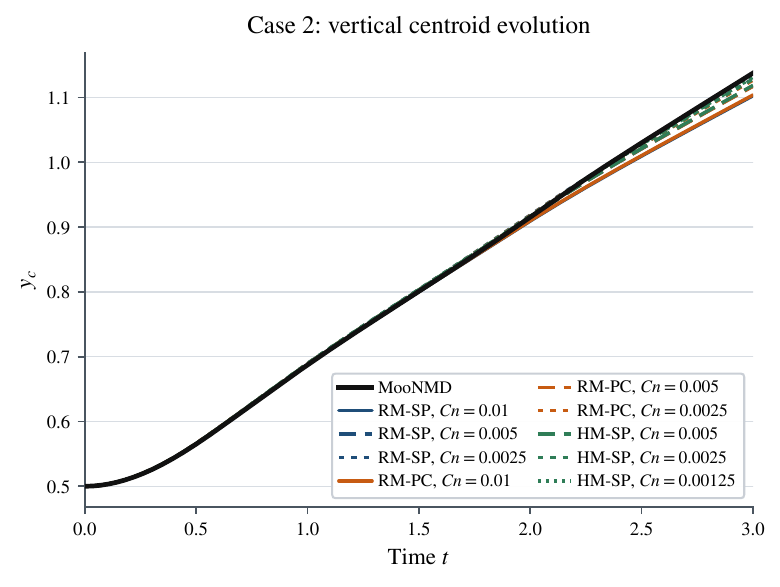} &
\includegraphics[width=0.31\linewidth]{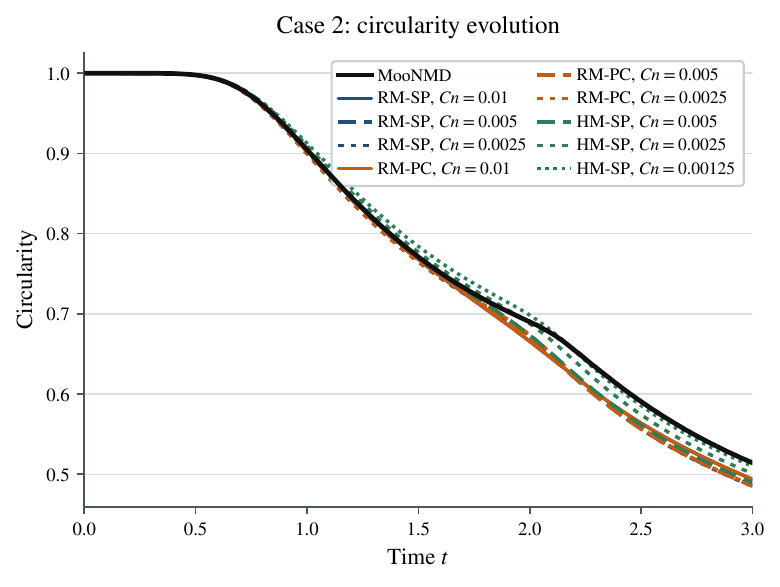} &
\includegraphics[width=0.31\linewidth]{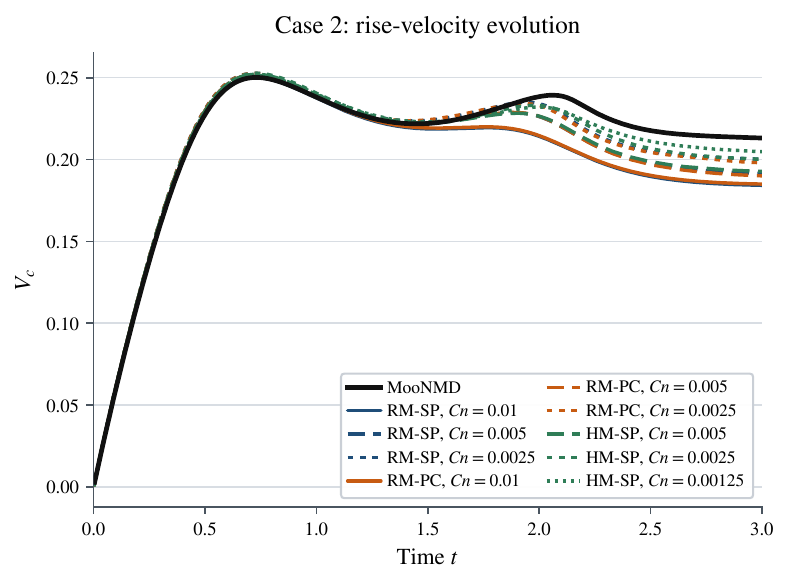}
\end{tabular}
\caption{Case~2 time histories of the bubble centroid height (left),
circularity (center), and mean rise velocity (right). Black curves
denote the MooNMD data of Hysing et al.; the RM-SP and RM-PC histories at
common $\Cn$ are nearly superposed.}
\label{fig:rising-bubble-case2-evolution}
\end{figure}

Despite the much stronger deformation, all three methods reproduce
the principal benchmark dynamics. At $\Cn=0.0025$, the RM-SP and
RM-PC histories and terminal contours remain closely superposed; their
terminal-centroid deviations are below $0.96\%$, and their maximum
rise velocities differ from the MooNMD data by about $1\%$. Figures~\ref{fig:rising-bubble-case2-state}
and~\ref{fig:rising-bubble-case2-evolution} show the corresponding terminal states and time histories.
The minimum circularity occurs at $t=3$ in all runs, consistently with
the continued deformation of the reference bubble. Table~\ref{tab:bubble-case2}
also shows a distinct resolution trend in the skirt region: the RM values
change only marginally once $\Cn\le0.005$, whereas HM-SP continues to
approach the MooNMD circularity as $\Cn$ decreases. For the present metrics and resolutions, this trend suggests that the fixed-topology $r$-adaptive mesh is more resolution-limited in the thin, highly stretched skirt. Figure~\ref{fig:rising-bubble-terminal-contours}
resolves the corresponding interface differences.

Across the six representative rising-bubble computations, the phase-mass drift remains below $2.4\times10^{-15}$ and the
accepted-velocity weak-divergence residual below $5.2\times10^{-13}$.
The HM-SP runs undergo $749$ and $1199$ topology-changing adaptation
events in Cases~1 and~2, respectively, while their carrier-compatibility residuals remain below $3.3\times10^{-14}$ and the
cross-mesh phase-mass defects below $1.6\times10^{-15}$. The close RM-SP/RM-PC agreement and the benchmark comparisons in
Figures~\ref{fig:rising-bubble-case1-state}--%
\ref{fig:rising-bubble-terminal-contours} are accompanied by phase-conservation, weak-incompressibility, and carrier residuals at the $10^{-13}$--$10^{-15}$ levels for both mesh updates.

\begin{table}[!htbp]
\centering
\footnotesize
\renewcommand{\arraystretch}{0.90}
\setlength{\tabcolsep}{5.0pt}
\caption{Benchmark quantities for rising-bubble Case 2.}
\label{tab:bubble-case2}
\begin{tabular}{lccccccc}
\toprule
Method & $\Cn$
& $y_c(3)$
& $\mathcal C_{\min}$
& $t_{\mathcal C}$
& $V_{\max}$
& $t_V$ \\
\midrule
MooNMD& --& 1.137580 & 0.514444  & 3.000& 0.250218 & 0.732\\
RM-SP& 0.01000& 1.102672 & 0.493810 & 3.000& 0.250081 & 0.729\\
RM-SP& 0.00500& 1.117772 & 0.485369 & 3.000& 0.251902 & 0.725\\
RM-SP& 0.00250& 1.128512 & 0.485054 & 3.000& 0.252717 & 0.723\\
RM-PC& 0.01000& 1.103885 & 0.493774 & 3.000& 0.250195 & 0.730\\
RM-PC& 0.00500& 1.117301 & 0.485459 & 3.000& 0.251920 & 0.725\\
RM-PC& 0.00250& 1.126765 & 0.486453 & 3.000& 0.252672 & 0.723\\
HM-SP& 0.00500& 1.118742 & 0.489709 & 3.000& 0.252159 & 0.728\\
HM-SP& 0.00250& 1.128384 & 0.500945 & 3.000& 0.252815 & 0.720\\
HM-SP& 0.00125& 1.131203 & 0.509982 & 3.000& 0.251524 & 0.732\\
\bottomrule
\end{tabular}
\end{table}

\begin{figure}[!t]
\centering
\begin{minipage}[b]{0.47\linewidth}
  \centering
  \includegraphics[width=\linewidth]{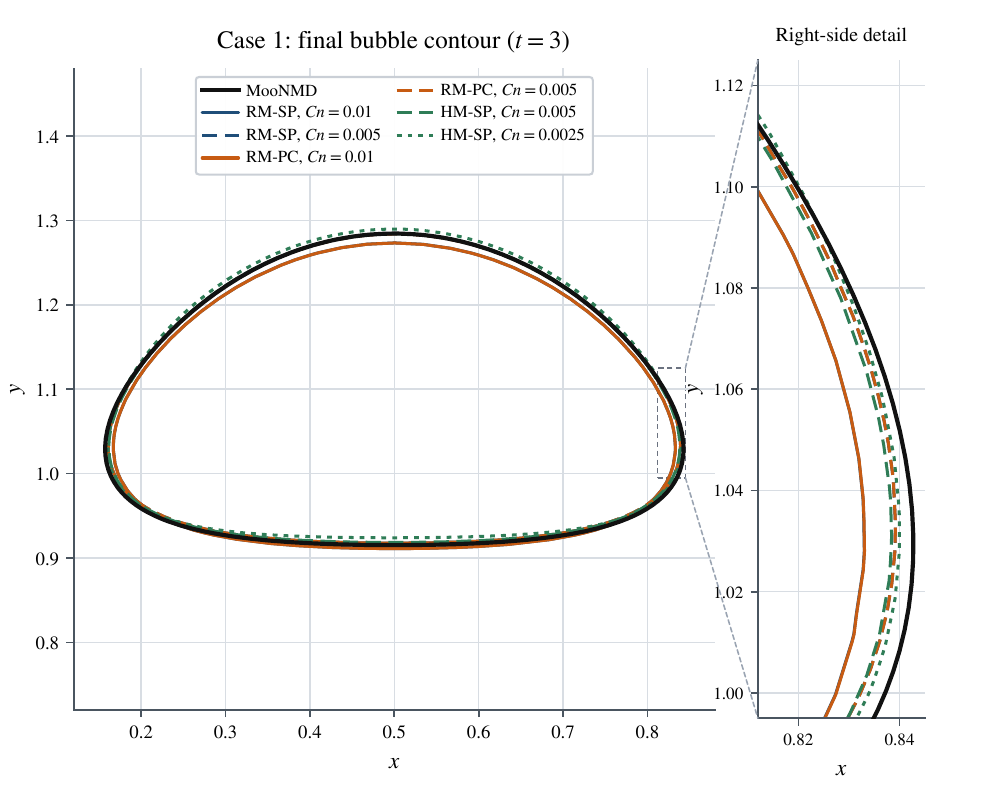}\\[-1pt]
  \small (a) Case~1
\end{minipage}\hfill
\begin{minipage}[b]{0.47\linewidth}
  \centering
  \includegraphics[width=\linewidth]{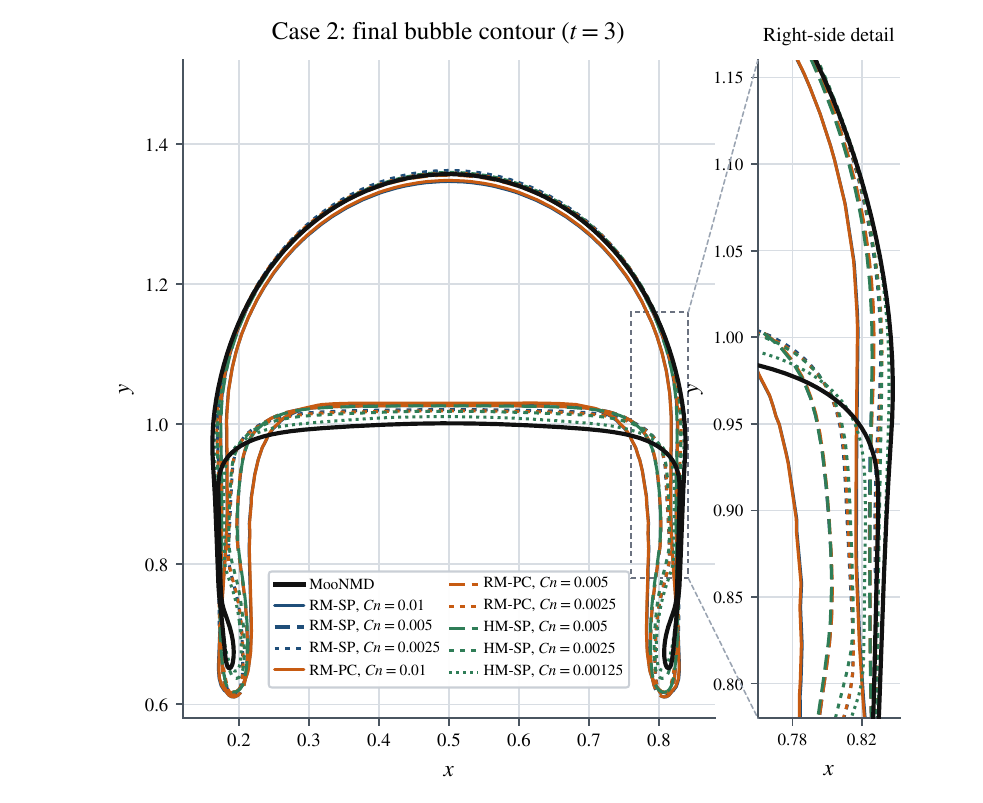}\\[-1pt]
  \small (b) Case~2
\end{minipage}
\caption{Terminal $\phi_h=0$ bubble contours at $t=3$ for the two
benchmark cases, compared with the MooNMD data of Hysing et al.\ and the selected
RM-SP, RM-PC, and HM-SP results. The RM-SP and RM-PC contours are
nearly superposed at common $\Cn$; the insets resolve the right-side
interface details.}
\label{fig:rising-bubble-terminal-contours}
\end{figure}

\FloatBarrier

\section{Conclusions}
We developed decoupled BE and BDF2 finite element schemes for the variable-density CHNS system on mesh sequences generated by fixed-topology motion or topology-changing remeshing. The phase and kinetic histories are tied to their corresponding storage norms; in the pressure-correction formulation, the pressure history is transferred in a density-weighted gradient norm. The phase update and affine density law determine the AGG-consistent mass flux used in momentum transport. Together with the scalar capillary-exchange equation, this yields linear field subproblems, a uniquely determined positive scalar, and modified energy balances for both formulations.

The numerical tests confirm second-order temporal convergence under both mesh updates. Phase mass, weak incompressibility, and carrier compatibility remain at numerical precision. In the unforced tests, the modified energies decrease as predicted, while the Rayleigh--Taylor and rising-bubble computations give comparable macroscopic dynamics for fixed-topology motion and topology-changing remeshing.

A natural next step is to develop divergence-conforming discontinuous Galerkin (DG) formulations, where exact discrete incompressibility and Piola-compatible transfer may absorb the carrier correction into the velocity space and simplify the pressure treatment. The same cross-mesh viewpoint may also be extended to wall-energy and moving-contact-line models by incorporating the corresponding boundary energy into the history construction.

\section*{Acknowledgments}
\small
\noindent
This work was supported by the Key Program of the National Natural Science Foundation of China
(Grant No. 12431014),
the National Natural Science Foundation of China
(Grant Nos.  12371410, 12571469, and 12261131501),
the 111 Project (No. D23017),
the Program for Science and Technology Innovative Research Team in Higher Educational
Institutions of Hunan Province of China,
the Scientific Research Innovation Capability Support Project for Young Faculty of China
(No. SRICSPYF-BS2025132),
and the High Performance Computing Platform of Xiangtan University.
\normalsize

\begingroup
\small
\bibliographystyle{elsarticle-num}
\bibliography{reference}

@article{AbelsGarckeGruen2012,
  author  = {Abels, Helmut and Garcke, Harald and Gr{\"u}n, G{\"u}nther},
  title   = {{Thermodynamically consistent, frame indifferent diffuse interface models for incompressible two-phase flows with different densities}},
  journal = {Mathematical Models and Methods in Applied Sciences},
  year    = {2012},
  volume  = {22},
  number  = {3},
  pages   = {1150013},
  doi     = {10.1142/S0218202511500138},
}

@incollection{ArpaiaEtAl2022,
  author    = {Arpaia, Luca and Beaugendre, H{\'e}lo{\"i}se and Cirrottola, Luca and Froehly, Algiane and Lorini, Marco and Nouveau, L{\'e}o and Ricchiuto, Mario},
  title     = {{$h$- and $r$-adaptation on simplicial meshes using MMG tools}},
  editor    = {Sevilla, Rub{\'e}n and Perotto, Simona and Morgan, Kenneth},
  booktitle = {Mesh Generation and Adaptation: Cutting-Edge Techniques},
  series    = {SEMA SIMAI Springer Series},
  volume    = {30},
  publisher = {Springer International Publishing},
  address   = {Cham},
  year      = {2022},
  pages     = {183--208},
  isbn      = {978-3-030-92540-6},
  doi       = {10.1007/978-3-030-92540-6_9},
}

@book{BoffiBrezziFortin2013,
  author    = {Boffi, Daniele and Brezzi, Franco and Fortin, Michel},
  title     = {{Mixed Finite Element Methods and Applications}},
  series    = {Springer Series in Computational Mathematics},
  volume    = {44},
  edition   = {1},
  publisher = {Springer},
  address   = {Berlin, Heidelberg},
  year      = {2013},
  isbn      = {978-3-642-36518-8},
  doi       = {10.1007/978-3-642-36519-5},
}

@article{CahnHilliard1958,
  author  = {Cahn, John W. and Hilliard, John E.},
  title   = {{Free energy of a nonuniform system. I. Interfacial free energy}},
  journal = {The Journal of Chemical Physics},
  year    = {1958},
  volume  = {28},
  number  = {2},
  pages   = {258--267},
  doi     = {10.1063/1.1744102},
  issn    = {0021-9606},
}

@article{ChenHuangYi2021,
  author  = {Chen, Yaoyao and Huang, Yunqing and Yi, Nianyu},
  title   = {{A decoupled energy stable adaptive finite element method for Cahn--Hilliard--Navier--Stokes equations}},
  journal = {Communications in Computational Physics},
  year    = {2021},
  volume  = {29},
  pages   = {1186--1212},
  doi     = {10.4208/cicp.OA-2020-0032},
}

@article{ChenHuangYi2022,
  author  = {Yaoyao Chen and Yunqing Huang and Nianyu Yi},
  title   = {{Error analysis of a decoupled, linear and stable finite element method for Cahn--Hilliard--Navier--Stokes equations}},
  journal = {Applied Mathematics and Computation},
  year    = {2022},
  volume  = {421},
  pages   = {126928},
  doi     = {10.1016/j.amc.2022.126928},
  issn    = {0096-3003},
}

@article{ChenLiYangYin2026,
  author  = {Chen, Yaoyao and Li, Dongqian and Yang, Yin and Yin, Peimeng},
  title   = {Unconditional energy stable hybrid {IEQ-FEMs} for the {Cahn--Hilliard--Navier--Stokes} equations},
  journal = {International Journal of Numerical Analysis and Modeling},
  volume  = {23},
  number  = {6},
  pages   = {881--916},
  year    = {2026}
}

@article{ChenShen2016,
  author  = {Ying Chen and Jie Shen},
  title   = {{Efficient, adaptive energy stable schemes for the incompressible Cahn--Hilliard Navier--Stokes phase-field models}},
  journal = {Journal of Computational Physics},
  year    = {2016},
  volume  = {308},
  pages   = {40--56},
  doi     = {10.1016/j.jcp.2015.12.006},
  issn    = {0021-9991},
}

@article{ChenYang2021,
  author  = {Chen, Chuanjun and Yang, Xiaofeng},
  title   = {{Fully-discrete finite element numerical scheme with decoupling structure and energy stability for the Cahn--Hilliard phase-field model of two-phase incompressible flow system with variable density and viscosity}},
  journal = {ESAIM: Mathematical Modelling and Numerical Analysis},
  year    = {2021},
  volume  = {55},
  number  = {5},
  pages   = {2323--2347},
  doi     = {10.1051/m2an/2021056},
}

@article{DapognyDobrzynskiFrey2014,
  author  = {C. Dapogny and C. Dobrzynski and P. Frey},
  title   = {{Three-dimensional adaptive domain remeshing, implicit domain meshing, and applications to free and moving boundary problems}},
  journal = {Journal of Computational Physics},
  year    = {2014},
  volume  = {262},
  pages   = {358--378},
  doi     = {10.1016/j.jcp.2014.01.005},
  issn    = {0021-9991},
}

@article{DiegelWangWangWise2017,
  author  = {Diegel, Amanda E. and Wang, Cheng and Wang, Xiaoming and Wise, Steven M.},
  title   = {{Convergence analysis and error estimates for a second order accurate finite element method for the Cahn--Hilliard--Navier--Stokes system}},
  journal = {Numerische Mathematik},
  year    = {2017},
  volume  = {137},
  number  = {3},
  pages   = {495--534},
  doi     = {10.1007/s00211-017-0887-5},
  issn    = {0945-3245},
}

@article{DingSpeltShu2007,
  author  = {Hang Ding and Peter D. M. Spelt and Chang Shu},
  title   = {{Diffuse interface model for incompressible two-phase flows with large density ratios}},
  journal = {Journal of Computational Physics},
  year    = {2007},
  volume  = {226},
  number  = {2},
  pages   = {2078--2095},
  doi     = {10.1016/j.jcp.2007.06.028},
  issn    = {0021-9991},
}

@inproceedings{DobrzynskiFrey2008,
  author    = {Dobrzynski, C. and Frey, P.},
  title     = {{Anisotropic Delaunay mesh adaptation for unsteady simulations}},
  editor    = {Garimella, Rao V.},
  booktitle = {Proceedings of the 17th International Meshing Roundtable},
  publisher = {Springer Berlin Heidelberg},
  address   = {Berlin, Heidelberg},
  year      = {2008},
  pages     = {177--194},
  isbn      = {978-3-540-87921-3},
  doi       = {10.1007/978-3-540-87921-3_11},
}

@article{DoneaGiulianiHalleux1982,
  author  = {J. Donea and S. Giuliani and J.-P. Halleux},
  title   = {{An arbitrary Lagrangian--Eulerian finite element method for transient dynamic fluid--structure interactions}},
  journal = {Computer Methods in Applied Mechanics and Engineering},
  year    = {1982},
  volume  = {33},
  number  = {1},
  pages   = {689--723},
  doi     = {10.1016/0045-7825(82)90128-1},
  issn    = {0045-7825},
}

@article{DuanLiYang2022,
  author  = {Beiping Duan and Buyang Li and Zongze Yang},
  title   = {{An energy diminishing arbitrary Lagrangian--Eulerian finite element method for two-phase Navier--Stokes flow}},
  journal = {Journal of Computational Physics},
  year    = {2022},
  volume  = {461},
  pages   = {111215},
  doi     = {10.1016/j.jcp.2022.111215},
  issn    = {0021-9991},
}

@article{EtienneGaronPelletier2009,
  author  = {S. {\'E}tienne and A. Garon and D. Pelletier},
  title   = {{Perspective on the geometric conservation law and finite element methods for ALE simulations of incompressible flow}},
  journal = {Journal of Computational Physics},
  year    = {2009},
  volume  = {228},
  number  = {7},
  pages   = {2313--2333},
  doi     = {10.1016/j.jcp.2008.11.032},
  issn    = {0021-9991},
}

@article{FarhatGeuzaineGrandmont2001,
  author  = {Charbel Farhat and Philippe Geuzaine and C{\'e}line Grandmont},
  title   = {{The discrete geometric conservation law and the nonlinear stability of ALE schemes for the solution of flow problems on moving grids}},
  journal = {Journal of Computational Physics},
  year    = {2001},
  volume  = {174},
  number  = {2},
  pages   = {669--694},
  doi     = {10.1006/jcph.2001.6932},
  issn    = {0021-9991},
}

@article{FarrellEtAl2009,
  author  = {P. E. Farrell and M. D. Piggott and C. C. Pain and G. J. Gorman and C. R. Wilson},
  title   = {{Conservative interpolation between unstructured meshes via supermesh construction}},
  journal = {Computer Methods in Applied Mechanics and Engineering},
  year    = {2009},
  volume  = {198},
  number  = {33},
  pages   = {2632--2642},
  doi     = {10.1016/j.cma.2009.03.004},
  issn    = {0045-7825},
}

@article{FarrellMaddison2011,
  author  = {P. E. Farrell and J. R. Maddison},
  title   = {{Conservative interpolation between volume meshes by local Galerkin projection}},
  journal = {Computer Methods in Applied Mechanics and Engineering},
  year    = {2011},
  volume  = {200},
  number  = {1},
  pages   = {89--100},
  doi     = {10.1016/j.cma.2010.07.015},
  issn    = {0045-7825},
}

@article{Feng2006,
  author  = {Feng, Xiaobing},
  title   = {{Fully discrete finite element approximations of the Navier--Stokes--Cahn--Hilliard diffuse interface model for two-phase fluid flows}},
  journal = {SIAM Journal on Numerical Analysis},
  year    = {2006},
  volume  = {44},
  number  = {3},
  pages   = {1049--1072},
  doi     = {10.1137/050638333},
}

@article{FormaggiaNobile1999,
  author  = {Formaggia, Luca and Nobile, Fabio},
  title   = {{A stability analysis for the arbitrary Lagrangian Eulerian formulation with finite elements}},
  journal = {East-West Journal of Numerical Mathematics},
  year    = {1999},
  volume  = {7},
  pages   = {105--132},
}

@article{FormaggiaNobile2004,
  author  = {Luca Formaggia and Fabio Nobile},
  title   = {{Stability analysis of second-order time accurate schemes for ALE-FEM}},
  journal = {Computer Methods in Applied Mechanics and Engineering},
  year    = {2004},
  volume  = {193},
  number  = {39},
  pages   = {4097--4116},
  doi     = {10.1016/j.cma.2003.09.028},
  issn    = {0045-7825},
}

@article{Fu2020,
  author  = {Guosheng Fu},
  title   = {{A divergence-free HDG scheme for the Cahn--Hilliard phase-field model for two-phase incompressible flow}},
  journal = {Journal of Computational Physics},
  year    = {2020},
  volume  = {419},
  pages   = {109671},
  doi     = {10.1016/j.jcp.2020.109671},
  issn    = {0021-9991},
}

@article{GarckeNuernbergZhao2023,
  author  = {Harald Garcke and Robert N{\"u}rnberg and Quan Zhao},
  title   = {{Structure-preserving discretizations of two-phase Navier--Stokes flow using fitted and unfitted approaches}},
  journal = {Journal of Computational Physics},
  year    = {2023},
  volume  = {489},
  pages   = {112276},
  doi     = {10.1016/j.jcp.2023.112276},
  issn    = {0021-9991},
}

@article{GarckeNuernbergZhao2024,
  author  = {Harald Garcke and Robert N{\"u}rnberg and Quan Zhao},
  title   = {{Arbitrary Lagrangian--Eulerian finite element approximations for axisymmetric two-phase flow}},
  journal = {Computers \& Mathematics with Applications},
  year    = {2024},
  volume  = {155},
  pages   = {209--223},
  doi     = {10.1016/j.camwa.2023.12.013},
  issn    = {0898-1221},
}

@article{GeuzaineGrandmontFarhat2003,
  author  = {Philippe Geuzaine and C{\'e}line Grandmont and Charbel Farhat},
  title   = {{Design and analysis of ALE schemes with provable second-order time-accuracy for inviscid and viscous flow simulations}},
  journal = {Journal of Computational Physics},
  year    = {2003},
  volume  = {191},
  number  = {1},
  pages   = {206--227},
  doi     = {10.1016/S0021-9991(03)00311-5},
  issn    = {0021-9991},
}

@article{GruenKlingbeil2014,
  author  = {G. Gr{\"u}n and F. Klingbeil},
  title   = {{Two-phase flow with mass density contrast: stable schemes for a thermodynamic consistent and frame-indifferent diffuse-interface model}},
  journal = {Journal of Computational Physics},
  year    = {2014},
  volume  = {257},
  pages   = {708--725},
  doi     = {10.1016/j.jcp.2013.10.028},
  issn    = {0021-9991},
}

@article{GuermondMinevShen2006,
  author  = {J.-L. Guermond and P. Minev and Jie Shen},
  title   = {{An overview of projection methods for incompressible flows}},
  journal = {Computer Methods in Applied Mechanics and Engineering},
  year    = {2006},
  volume  = {195},
  number  = {44},
  pages   = {6011--6045},
  doi     = {10.1016/j.cma.2005.10.010},
  issn    = {0045-7825},
}

@article{GuermondQuartapelle2000,
  author  = {J.-L. Guermond and L. Quartapelle},
  title   = {{A projection FEM for variable density incompressible flows}},
  journal = {Journal of Computational Physics},
  year    = {2000},
  volume  = {165},
  number  = {1},
  pages   = {167--188},
  doi     = {10.1006/jcph.2000.6609},
  issn    = {0021-9991},
}

@article{GuermondSalgado2009,
  author  = {J.-L. Guermond and Abner Salgado},
  title   = {{A splitting method for incompressible flows with variable density based on a pressure Poisson equation}},
  journal = {Journal of Computational Physics},
  year    = {2009},
  volume  = {228},
  number  = {8},
  pages   = {2834--2846},
  doi     = {10.1016/j.jcp.2008.12.036},
  issn    = {0021-9991},
}

@article{HanWang2015,
  author  = {Daozhi Han and Xiaoming Wang},
  title   = {{A second order in time, uniquely solvable, unconditionally stable numerical scheme for Cahn--Hilliard--Navier--Stokes equation}},
  journal = {Journal of Computational Physics},
  year    = {2015},
  volume  = {290},
  pages   = {139--156},
  doi     = {10.1016/j.jcp.2015.02.046},
  issn    = {0021-9991},
}

@article{HohenbergHalperin1977,
  author  = {Hohenberg, P. C. and Halperin, B. I.},
  title   = {{Theory of dynamic critical phenomena}},
  journal = {Reviews of Modern Physics},
  year    = {1977},
  volume  = {49},
  pages   = {435--479},
  doi     = {10.1103/RevModPhys.49.435},
}

@article{Huang2005,
  author  = {Weizhang Huang},
  title   = {{Metric tensors for anisotropic mesh generation}},
  journal = {Journal of Computational Physics},
  year    = {2005},
  volume  = {204},
  number  = {2},
  pages   = {633--665},
  doi     = {10.1016/j.jcp.2004.10.024},
  issn    = {0021-9991},
}

@article{HuangKamenski2015,
  author  = {Weizhang Huang and Lennard Kamenski},
  title   = {{A geometric discretization and a simple implementation for variational mesh generation and adaptation}},
  journal = {Journal of Computational Physics},
  year    = {2015},
  volume  = {301},
  pages   = {322--337},
  doi     = {10.1016/j.jcp.2015.08.032},
  issn    = {0021-9991},
}

@article{HuangKamenski2018,
  author  = {Huang, Weizhang and Kamenski, Lennard},
  title   = {{On the mesh nonsingularity of the moving mesh PDE method}},
  journal = {Mathematics of Computation},
  year    = {2018},
  volume  = {87},
  number  = {312},
  pages   = {1887--1911},
  doi     = {10.1090/mcom/3271},
  issn    = {0025-5718},
}

@article{HuangKamenskiRussell2015,
  author  = {Huang, Weizhang and Kamenski, Lennard and Russell, Robert},
  title   = {{A comparative numerical study of meshing functionals for variational mesh adaptation}},
  journal = {Journal of Mathematical Study},
  year    = {2015},
  volume  = {48},
  number  = {2},
  pages   = {168--186},
  doi     = {10.4208/jms.v48n2.15.04},
}

@book{HuangRussell2011,
  author    = {Huang, Weizhang and Russell, Robert D.},
  title     = {{Adaptive Moving Mesh Methods}},
  series    = {Applied Mathematical Sciences},
  volume    = {174},
  edition   = {1},
  publisher = {Springer},
  address   = {New York, NY},
  year      = {2011},
  isbn      = {978-1-4419-7915-5},
  doi       = {10.1007/978-1-4419-7916-2},
}

@article{HysingEtAl2009,
  author  = {Hysing, S. and Turek, S. and Kuzmin, D. and Parolini, N. and Burman, E. and Ganesan, S. and Tobiska, L.},
  title   = {{Quantitative benchmark computations of two-dimensional bubble dynamics}},
  journal = {International Journal for Numerical Methods in Fluids},
  year    = {2009},
  volume  = {60},
  number  = {11},
  pages   = {1259--1288},
  doi     = {10.1002/fld.1934},
}

@article{Jacqmin1999,
  author  = {David Jacqmin},
  title   = {{Calculation of two-phase Navier--Stokes flows using phase-field modeling}},
  journal = {Journal of Computational Physics},
  year    = {1999},
  volume  = {155},
  number  = {1},
  pages   = {96--127},
  doi     = {10.1006/jcph.1999.6332},
  issn    = {0021-9991},
}

@article{KayStylesWelford2008,
  author  = {Kay, David and Styles, Vanessa and Welford, Richard},
  title   = {{Finite element approximation of a Cahn--Hilliard--Navier--Stokes system}},
  journal = {Interfaces and Free Boundaries},
  year    = {2008},
  volume  = {10},
  number  = {1},
  pages   = {15--43},
  doi     = {10.4171/IFB/178},
  issn    = {1463-9963},
}

@article{KhanwaleEtAl2020,
  author  = {Makrand A. Khanwale and Alec D. Lofquist and Hari Sundar and James A. Rossmanith and Baskar Ganapathysubramanian},
  title   = {{Simulating two-phase flows with thermodynamically consistent energy stable Cahn--Hilliard Navier--Stokes equations on parallel adaptive octree based meshes}},
  journal = {Journal of Computational Physics},
  year    = {2020},
  volume  = {419},
  pages   = {109674},
  doi     = {10.1016/j.jcp.2020.109674},
  issn    = {0021-9991},
}

@article{KhanwaleEtAl2023,
  author  = {Makrand A. Khanwale and Kumar Saurabh and Masado Ishii and Hari Sundar and James A. Rossmanith and Baskar Ganapathysubramanian},
  title   = {{A projection-based, semi-implicit time-stepping approach for the Cahn--Hilliard Navier--Stokes equations on adaptive octree meshes}},
  journal = {Journal of Computational Physics},
  year    = {2023},
  volume  = {475},
  pages   = {111874},
  doi     = {10.1016/j.jcp.2022.111874},
  issn    = {0021-9991},
}

@article{LiChengShenZhang2021,
  author  = {Maojun Li and Yongping Cheng and Jie Shen and Xiangxiong Zhang},
  title   = {{A bound-preserving high order scheme for variable density incompressible Navier--Stokes equations}},
  journal = {Journal of Computational Physics},
  year    = {2021},
  volume  = {425},
  pages   = {109906},
  doi     = {10.1016/j.jcp.2020.109906},
  issn    = {0021-9991},
}

@article{LiLiuShenZheng2025,
  author  = {Li, Xiaoli and Liu, Zhengguang and Shen, Jie and Zheng, Nan},
  title   = {{On a class of higher-order fully decoupled schemes for the Cahn--Hilliard--Navier--Stokes system}},
  journal = {Journal of Scientific Computing},
  year    = {2025},
  volume  = {103},
  number  = {1},
  pages   = {27},
  doi     = {10.1007/s10915-025-02835-y},
  issn    = {1573-7691},
}

@article{LiShen2022,
  author  = {Li, Xiaoli and Shen, Jie},
  title   = {{On fully decoupled MSAV schemes for the Cahn--Hilliard--Navier--Stokes model of two-phase incompressible flows}},
  journal = {Mathematical Models and Methods in Applied Sciences},
  year    = {2022},
  volume  = {32},
  number  = {03},
  pages   = {457--495},
  doi     = {10.1142/S0218202522500117},
}

@article{Liu2013,
  author  = {Liu, Jie},
  title   = {{Simple and efficient ALE methods with provable temporal accuracy up to fifth order for the Stokes equations on time varying domains}},
  journal = {SIAM Journal on Numerical Analysis},
  year    = {2013},
  volume  = {51},
  number  = {2},
  pages   = {743--772},
  doi     = {10.1137/110825996},
}

@article{LiuShenYang2015,
  author  = {Liu, Chun and Shen, Jie and Yang, Xiaofeng},
  title   = {{Decoupled energy stable schemes for a phase-field model of two-phase incompressible flows with variable density}},
  journal = {Journal of Scientific Computing},
  year    = {2015},
  volume  = {62},
  number  = {2},
  pages   = {601--622},
  doi     = {10.1007/s10915-014-9867-4},
  issn    = {0885-7474},
}

@article{LowengrubTruskinovsky1998,
  author  = {Lowengrub, John S. and Truskinovsky, Lev},
  title   = {{Quasi-incompressible Cahn--Hilliard fluids and topological transitions}},
  journal = {Proceedings of the Royal Society of London. Series A: Mathematical, Physical and Engineering Sciences},
  year    = {1998},
  volume  = {454},
  number  = {1978},
  pages   = {2617--2654},
  doi     = {10.1098/rspa.1998.0273},
  issn    = {1364-5021},
}

@article{PontCodinaBaiges2017,
  author  = {Pont, Arnau and Codina, Ramon and Baiges, Joan},
  title   = {{Interpolation with restrictions between finite element meshes for flow problems in an ALE setting}},
  journal = {International Journal for Numerical Methods in Engineering},
  year    = {2017},
  volume  = {110},
  number  = {13},
  pages   = {1203--1226},
  doi     = {10.1002/nme.5444},
}

@article{ShenXuYang2018,
  author  = {Jie Shen and Jie Xu and Jiang Yang},
  title   = {{The scalar auxiliary variable (SAV) approach for gradient flows}},
  journal = {Journal of Computational Physics},
  year    = {2018},
  volume  = {353},
  pages   = {407--416},
  doi     = {10.1016/j.jcp.2017.10.021},
  issn    = {0021-9991},
}

@article{ShenXuYang2019,
  author  = {Shen, Jie and Xu, Jie and Yang, Jiang},
  title   = {{A new class of efficient and robust energy stable schemes for gradient flows}},
  journal = {SIAM Review},
  year    = {2019},
  volume  = {61},
  number  = {3},
  pages   = {474--506},
  doi     = {10.1137/17M1150153},
}

@article{ShenYang2010,
  author  = {Shen, Jie and Yang, Xiaofeng},
  title   = {{A phase-field model and its numerical approximation for two-phase incompressible flows with different densities and viscosities}},
  journal = {SIAM Journal on Scientific Computing},
  year    = {2010},
  volume  = {32},
  number  = {3},
  pages   = {1159--1179},
  doi     = {10.1137/09075860X},
}

@article{ShenYang2015,
  author  = {Shen, Jie and Yang, Xiaofeng},
  title   = {{Decoupled, energy stable schemes for phase-field models of two-phase incompressible flows}},
  journal = {SIAM Journal on Numerical Analysis},
  year    = {2015},
  volume  = {53},
  number  = {1},
  pages   = {279--296},
  doi     = {10.1137/140971154},
}

@article{WangHuangWei2026,
  title={A Trace--Logarithmic Variational Functional for Equidistribution and Alignment in Moving Mesh Adaptation}, 
  author={Wenbin Wang and Yunqing Huang and Huayi Wei},
  year={2026},
  eprint={2601.20235},
  archivePrefix={arXiv},
  primaryClass={math.NA},
}

@article{WangLiWang2024,
  author  = {Jiancheng Wang and Maojun Li and Cheng Wang},
  title   = {{Efficient finite element schemes for a phase field model of two-phase incompressible flows with different densities}},
  journal = {Journal of Computational Physics},
  year    = {2024},
  volume  = {518},
  pages   = {113331},
  doi     = {10.1016/j.jcp.2024.113331},
  issn    = {0021-9991},
}

@article{YueFengLiuShen2004,
  author  = {Yue, Pengtao and Feng, James J. and Liu, Chun and Shen, Jie},
  title   = {{A diffuse-interface method for simulating two-phase flows of complex fluids}},
  journal = {Journal of Fluid Mechanics},
  year    = {2004},
  volume  = {515},
  pages   = {293--317},
  doi     = {10.1017/S0022112004000370},
}

@article{ZhangLuoWang2025,
  author  = {Jinpeng Zhang and Li Luo and Xiaoping Wang},
  title   = {{A fully-decoupled second-order-in-time and unconditionally energy stable scheme for the Cahn--Hilliard--Navier--Stokes equations with variable density}},
  journal = {Journal of Computational Physics},
  year    = {2025},
  volume  = {532},
  pages   = {113943},
  doi     = {10.1016/j.jcp.2025.113943},
  issn    = {0021-9991},
}

@article{ZhaoHan2021,
  author  = {Jia Zhao and Daozhi Han},
  title   = {{Second-order decoupled energy-stable schemes for Cahn--Hilliard--Navier--Stokes equations}},
  journal = {Journal of Computational Physics},
  year    = {2021},
  volume  = {443},
  pages   = {110536},
  doi     = {10.1016/j.jcp.2021.110536},
  issn    = {0021-9991},
}

@article{ZhaoRen2020,
  author  = {Quan Zhao and Weiqing Ren},
  title   = {{An energy-stable finite element method for the simulation of moving contact lines in two-phase flows}},
  journal = {Journal of Computational Physics},
  year    = {2020},
  volume  = {417},
  pages   = {109582},
  doi     = {10.1016/j.jcp.2020.109582},
  issn    = {0021-9991},
}

@article{ZhengEtAl2026,
  author  = {Zheng, Yangyang and Wei, Huayi and Huang, Yunqing and Chen, Chunyu and Tian, Tian and Liu, Hanbin and Wang, Wenbin and He, Liang},
  title   = {{FEALPy: A cross-platform intelligent numerical simulation engine}},
  journal = {Communications in Computational Physics},
  year    = {2026},
  volume  = {40},
  number  = {5},
  pages   = {1676--1704},
  doi     = {10.4208/cicp.oa-2025-0327},
}

@article{ZhengKarniadakis2016,
  author  = {X. Zheng and G. E. Karniadakis},
  title   = {{A phase-field/ALE method for simulating fluid--structure interactions in two-phase flow}},
  journal = {Computer Methods in Applied Mechanics and Engineering},
  year    = {2016},
  volume  = {309},
  pages   = {19--40},
  doi     = {10.1016/j.cma.2016.04.035},
  issn    = {0045-7825},
}

@article{ZouEtAl2026,
  author  = {Zou, Guang-An and Wang, Meiting and Pan, Kejia and Yang, Yin and Yang, Xiaofeng},
  title   = {{Efficient energy-stable discontinuous Galerkin scheme for the non-isothermal Cahn--Hilliard--Navier--Stokes two-phase fluid flow system}},
  journal = {International Journal for Numerical Methods in Engineering},
  year    = {2026},
  volume  = {127},
  number  = {7},
  pages   = {e70319},
  doi     = {10.1002/nme.70319},
}
\endgroup

\end{document}